\documentclass[11pt]{amsart}

\usepackage[margin=1in]{geometry}
\usepackage{amsmath,amssymb,amsthm,mathtools}
\usepackage{algorithm}
\usepackage{algpseudocode}
\usepackage{xcolor}
\usepackage{tikz}
\usetikzlibrary{arrows.meta}
\usepackage[hidelinks]{hyperref}

\makeatletter
\renewcommand\subsection{\@startsection{subsection}{2}{\z@}%
  {2.5ex \@plus 1ex \@minus .2ex}%
  {-1em}%
  {\normalfont\normalsize\bfseries}}
\def\l@subsection{\@tocline{2}{0pt}{1pc}{5pc}{\small}}
\makeatother

\newtheorem{lemma}{Lemma}[section]
\newtheorem{theorem}[lemma]{Theorem}
\newtheorem{proposition}[lemma]{Proposition}
\newtheorem{definition}[lemma]{Definition}
\newtheorem{corollary}[lemma]{Corollary}
\newtheorem{problem}[lemma]{Problem}
\theoremstyle{remark}
\newtheorem{remark}[lemma]{Remark}
\theoremstyle{plain}
\newtheorem*{claim}{Claim}

\newcommand{\R}{\mathbb R}
\newcommand{\norm}[1]{\left\lVert #1\right\rVert}
\newcommand{\abs}[1]{\left\lvert #1\right\rvert}
\newcommand{\FM}{\operatorname{FM}}

\algrenewcommand\algorithmicrequire{\textbf{Input:}}
\algrenewcommand\algorithmicensure{\textbf{Output:}}

\title[Level-set entropy and sparse embeddings]{Level-set entropy and sparse randomized embeddings}
\author{Konstantin Tikhomirov}
\address{Department of Mathematical Sciences, Carnegie Mellon University}
\email{ktikhomi@andrew.cmu.edu}
\date{}

\begin{document}

\begin{abstract}
Let $\Pi$ be a $k\times n$ sparse random matrix.  For a fixed
$r$-dimensional subspace $V\subset\R^n$, let $U_V:\R^r\to\R^n$ denote an isometry from $\R^r$
onto $V$.  The product $\Pi U_V$ is a central model in randomized
dimension reduction and has been studied primarily through trace and Gaussian comparison inequalities.  In this work,
we develop an approach to the spectral norm
of the matrix product $\Pi U_V$, based on entropy estimates for
level sets of vectors $x\in V$.
Combining the method
with existing estimates, we show the following.
Assume that
\[
    k\ge C\,r(\log\log r)^2,\qquad
    p\ge (\log k)/k.
\]
Let $\Pi$ be a $k\times n$ matrix with i.i.d. entries
equidistributed with the product $b\,\xi$, where $b$ is a Bernoulli($p$)
random variable and $\xi$ is mean-zero, independent of $b$, and satisfies
$|\xi|\le1$ almost surely.
Then with high probability
\[
    \|\Pi U_V\|\le C\sqrt{kp}.
\]
Matching results hold for other random models with negatively associated entries.
\end{abstract}

\maketitle

\setcounter{tocdepth}{1}
\tableofcontents

\section{Introduction}

\subsection{Literature Overview}

Let $V\subset\mathbb R^n$ be a fixed $r$-dimensional subspace, and let
$\Pi:\mathbb R^n\to\mathbb R^k$ be a random linear map.  A basic question in
randomized dimension reduction is to quantify the action of $\Pi$ on the whole
of $V$.  
This question has
become a standard component of the modern theory of randomized numerical
linear algebra
\cite{Sarlos06,HalkoMartinssonTropp11,
MartinssonTropp20,Woodruff14}.

Throughout the paper,
$U_V:\R^{r}\to\R^n$ is an isometry from $\R^{r}$ to $V$. 
In context of the above question, one seeks information about the singular values of
$\Pi U_V$, independently of the structure of $V$.
The random map $\Pi$ is an {\it Oblivious Subspace Embedding\footnote{In this paper we use the term OSE for the ``flat'' matrix orientation ($k\times n$). In the literature, the OSE matrix is often defined with the transposed orientation $n\times k$.}} (OSE) \cite{NelsonNguyen13} with dimension parameter $r$
and distortion $\varepsilon>0$, if
for every fixed $r$--dimensional subspace $V$, we have
\begin{equation}\label{eq: NN OSE}
(1+\varepsilon)^{-1} \leq s_{\min}(\Pi U_V)\leq \|\Pi U_V\|\leq (1+\varepsilon)\quad\mbox{with probability close to one,}
\end{equation}
with $s_{\min}(\cdot)$ denoting the smallest singular value of the corresponding matrix.
Subspace embeddings are used to
compress regression problems, construct preconditioners,
and obtain low-rank approximations.  
As an illustration, for an OSE matrix $\Pi$ and any $n\times d$ matrix $A$ of rank $r$,
the non-trivial singular values $\|A\|=s_1(A)\geq s_2(A)\geq \dots\geq s_r(A)$ and the singular values of the product $\Pi A$ are related as
$$
(1+\varepsilon)^{-1}s_i(\Pi A)\leq s_i(A)\leq (1+\varepsilon)s_i(\Pi A)
\quad \mbox{w.h.p}
$$

Establishing the OSE property for a given random model
is a very active line of research.
Dense random maps give strong concentration and nearly optimal tradeoff
between embedding
dimension $k$, the subspace dimension $r$, and distortion parameter $\varepsilon$ in \eqref{eq: NN OSE}, but applying them may be computationally inefficient.
This led to development of sparse transforms with simple discrete entries
\cite{Achlioptas03,AilonChazelle09,DasguptaKumarSarlos10,KaneNelson14}.
Sparse
embeddings can be applied in time nearly proportional to
the number of entries of the input matrix
\cite{ClarksonWoodruff13,NelsonNguyen13,Cohen16,Woodruff14}.
Necessary tradeoffs
between embedding dimension, column sparsity, and distortion for two-sided
oblivious subspace embeddings were established in
\cite{NelsonNguyen14}.
Sparsity changes the mathematical character of the problem: a small number
of unusually large coordinates of a vector may interact with atypical row or
column occupancies, and estimates that are immediate for dense subgaussian
matrices may no longer be uniform over an entire subspace.

Two sparse random models traditionally considered in the literature in this context are the i.i.d. Bernoulli-sparse model and the {\it SparseStack} \cite{NelsonNguyen13}
(or, more precisely, the transpose of SparseStack when considering the ``flat" $k\times n$ matrix orientation).
In the former case, the matrix is populated with independent copies of a product $b\,\xi$ where $b$ is a Bernoulli($p$) variable, and $\xi$
is an independent symmetric sign variable, or, more generally, a centered variable satisfying extra moment/boundedness conditions.
In the latter case, the $k\times n$ matrix is constructed by dividing $[k]$ into a few equal-sized blocks and generating i.i.d. columns, where each column has 
exactly one non-zero entry in each block, distributed uniformly within the block\footnote{We shall formally define the SparseStack model later in the paper.}.

Among recent results, 
sign matrices {\it of polylogarithmic sparsity} (both i.i.d and SparseStack$^\top$ models)
were shown to attain essentially minimal linear
embedding dimension at constant embedding distortion
\cite{ChenakkodDerezinskiDongRudelson24}: for every fixed $\theta>0$, one
can take $k\ge(1+\theta)r$ with $O_\theta(\log^4 r)$ nonzero entries per
column to guarantee,
for every choice of an $r$--dimensional non-random subspace $V$,
$\|\Pi U_V\|\leq C_\theta\,s_{\min}(\Pi U_V)$ with high probability.
The corresponding small-distortion bounds gave the optimal order
$k=\Theta(r/\varepsilon^2)$ with
$O\bigl(\log^4 (r)/\varepsilon^6\bigr)$ nonzero entries per column.  Subsequent work retained the
optimal embedding dimension while reducing the column sparsity to
$O\bigl(\log^2(r/(\varepsilon))/\varepsilon
    +\log^3(r/(\varepsilon))\bigr)$ \cite{ChenakkodDerezinskiDong25}.
More recently, $(1+\varepsilon)$--Oblivious Subspace Embedding property
was verified for the SparseStack$^\top$
random matrix $\Pi$ for $\varepsilon\ge r^{-O(1)}$, with
$k=O(r \log^{o(1)}r/\varepsilon^2)$ rows and
$(\log^{1+o(1)} r)/\varepsilon$ nonzero entries per column,
where $o(1)$ decays roughly as $1/\log\log\log(r)$
\cite{ChenakkodDerezinskiDong25b}.
Whereas the result of \cite{ChenakkodDerezinskiDong25b} comes relatively close to establishing the seminal Nelson--Nguyen conjecture \cite{NelsonNguyen13},
it does not provide constant-distortion bounds for truly proportional dimension $k=C\,r$ and logarithmic sparsity (see open problem below).
We further remark that the lower edge $s_{\min}(\Pi U_V)$ has been
actively studied recently \cite{OymakTropp18,CamanoEpperlyMeyerTropp25,Tropp26,HuangRudelsonTikhomirov26},
and, in particular,
\cite{Tropp26} provided estimates for very sparse random maps which resolve
the {\it lower bound} in the Nelson--Nguyen conjecture, up to exact dependence on $\varepsilon$.

Despite the very active research and substantial progress in the last few years,
several fundamental questions regarding spectral properties
of sparse random maps remain unresolved as of this writing.
The concrete problem which motivated this paper is the following fixed-distortion form
of the Nelson--Nguyen conjecture
\cite{NelsonNguyen13,TroppICERM26}:

\begin{problem}[Fixed-distortion sparse embeddings; see \cite{TroppICERM26}]\label{problem main}
Fix $\varepsilon>0$.  Does there exist a constant
$C_\varepsilon\ge1$ such that, for every $n\ge r$, one can construct a
random normalized sparse sign matrix
$\Pi=(kp)^{-1/2}(\eta_{ui}\varepsilon_{ui})_{u,i}$ with
$k\le C_\varepsilon r$ and entry density
$p=C_\varepsilon(\log r)/k$ (equivalently, expected or exact column sparsity
$kp=C_\varepsilon\log r$, according to the model) for which, for every fixed
$r$-dimensional subspace $V\subset\mathbb R^n$,
\[
    \mathbb P\left\{
    (1+\varepsilon)^{-1}\leq
    s_{\min}(\Pi U_V)
    \leq
        \norm{\Pi U_V}
        \le 1+\varepsilon
    \right\}
    =1-o(1)?
\]
\end{problem}
While the leftmost inequality above is verified in \cite{Tropp26}
for standard sparse random models, the upper estimate on the
spectral norm of $\Pi U_V$ remains a challenging open problem.

Regarding the available methods,
spectral moments and matrix trace
inequalities provide one route to controlling the singular spectrum; in the sparse setting,
the resulting high trace expansions can be organized by combinatorial
multigraphs \cite{KaneNelson14,NelsonNguyen13,Cohen16}.  Non-asymptotic
comparison inequalities provide another route: matrix universality permits
one to compare the spectrum of a sum of independent random matrices with
that of a Gaussian model having the same mean and covariance
\cite{BrailovskayaVanHandel24}.  This principle, combined with
model-specific trace estimates and other ingredients, underlies several recent sparse OSE analyses
\cite{ChenakkodDerezinskiDongRudelson24,ChenakkodDerezinskiDong25,
ChenakkodDerezinskiDong25b}.  The recent lower-edge comparison for sums of
positive-semidefinite random matrices \cite{Tropp26} yields sparse subspace-injection
estimates.  Another route uses generic chaining to treat
relevant supremum for dimension reduction and subspace
embeddings \cite{BourgainDirksenNelson15}.

\subsection{Main Results}

The present paper develops a completely different
approach to the upper spectral edge based on evaluating
the entropy (number of possible realizations) of level
sets of unit vectors in $V$. To formulate the main result, we introduce an admissible random matrix model
which comprises the classical i.i.d and SparseStack constructions.

\begin{definition}[Negatively associated support mask]
\label{def:negatively-associated-support-mask}
Let $0<p\le1$.  A random mask
$\eta=(\eta_{ui})_{u\in[k],\,i\in[n]}\in\{0,1\}^{k\times n}$ is called a
\emph{negatively associated support mask with density $p$} if
\[
    \mathbb E\eta_{ui}=p
    \qquad (u\in[k],\ i\in[n])
\]
and the family $(\eta_{ui})_{u,i}$ is negatively associated.  Thus, whenever
$S,T\subset[k]\times[n]$ are disjoint and $f,g$ are bounded nonnegative
coordinatewise non-decreasing functions,
\[
    \mathbb E f(\eta_S)g(\eta_T)
    \le
    \mathbb E f(\eta_S)\,\mathbb E g(\eta_T).
\]
\end{definition}

\begin{definition}[Admissible sparse-entry model]
\label{def:admissible-sparse-entry-model}
Let $0<p\le1$, let $\eta$ be a negatively associated support mask with
density $p$, and let $\xi$ be a real random variable.  A $k\times n$ random
matrix $\Pi=(\pi_{ui})$ follows the \emph{admissible sparse-entry model} with
parameters $p$ and $\xi$ if
\[
    \pi_{ui}=\eta_{ui}\xi_{ui},
    \qquad u\in[k],\ i\in[n],
\]
where the variables $\xi_{ui}$ are independent copies of $\xi$, independent
of the mask $\eta$.
\end{definition}

\begin{theorem}[Main result]
\label{thm:centered-entry-extension}
For every $B\ge1$ there is a constant
$C_{\text{\tiny\ref*{thm:centered-entry-extension}}}=C(B)$ with the
following property.  Let
\[
    n\ge k\ge r\ge3,
    \qquad
    n\le r^{10},
    \qquad
    k\ge r\log^2\left(\frac{en}{r}\right),
    \qquad
    \frac{\log k}{k}\le p\le1.
\]
Let $\Pi=(\pi_{ui})$ follow the admissible sparse-entry model with parameter $p$
and an entry variable $\xi$ satisfying
\[
    \mathbb E\xi=0,
    \qquad
    |\xi|\le1\quad\text{almost surely}.
\]
Then, for every non-random $r$-dimensional subspace $V\subset\mathbb R^n$,
\[
    \mathbb P\left\{
        \|\Pi U_V\|>
        C_{\text{\tiny\ref*{thm:centered-entry-extension}}}\sqrt{kp}
    \right\}
    \le k^{-B}.
\]
\end{theorem}

\begin{remark}[Concrete random models]\label{rem:main-thm-spec}
Theorem~\ref{thm:centered-entry-extension} yields the
high-probability bound
\[
    \norm{\Pi U_V}\le C_B\sqrt{kp}
\]
for three standard sparse models considered in this paper: the i.i.d.
Bernoulli-sparse model, the fixed-column-degree combinatorial model, and the
unnormalized SparseStack$^{\mathsf T}$.  The precise definitions,
parameter restrictions, and reductions to Theorem~\ref{thm:centered-entry-extension}
are given in Section~\ref{sec:proofs-main-results}.
\end{remark}

Theorem~\ref{thm:centered-entry-extension}
can be combined with the matrix universality inequality of
\cite{BrailovskayaVanHandel24} to yield an $n$--independent bound on $k$
for each of the three concrete models listed in the preceding remark:

\begin{corollary}[A hybrid leverage-score argument]
\label{rem:hybrid-leverage-score-argument}
Assume
\[
r\geq 3,\quad
k\ge C\,r\bigl(\log\log(r)\bigr)^2,\quad p\ge (\log k)/k,
\]
where $C>0$ is a sufficiently large universal constant.
Let $\Pi$ be a $k\times n$ matrix (for arbitrary $n\geq r$) with i.i.d. entries
equidistributed with the product $b\,\xi$, where $b$ is Bernoulli($p$)
and $\xi$ is centered, independent of $b$, and satisfies $|\xi|\leq 1$ a.e.
Then for every non-random $V\subset\R^n$ of dimension $r$,
$$\norm{\Pi U_V}= O(\sqrt{kp})$$
with probability at least $1-r^{-10}$.
Matching results hold for the fixed-column-degree and SparseStack models\footnote{See Section~\ref{sec:proofs-main-results} for details.}.
\end{corollary}

\begin{remark}
The required bound for $k$ in the above statement is a factor of
$(\log\log(r))^2$ greater than the conjectured optimal lower bound
$k=\Omega(r)$ \cite{NelsonNguyen13,TroppICERM26},
while the sparsity assumption $p\ge (\log k)/k$
is optimal for $k$ polynomial in $r$.
Compared to the above result, the strongest available lower
bound on $k$ in the constant-distortion regime
and with {\it logarithmic average column sparsity} prior to this writing 
is $k=\Omega(r\log r)$ \cite{Cohen16},
or $k=\Omega(r\log^{c/\log\log\log r} r)$
for near-logarithmic sparsity \cite{ChenakkodDerezinskiDong25b},
with both bounds asymptotically larger than $r\,(\log\log(r))^2$.
\end{remark}

\begin{remark}[Extensions of the entropy-guided framework]
The entropy-guided framework is expected to provide
$(1+\varepsilon)$--sharp estimates for the upper spectral edge, as well as
control of the lower spectral edge $\inf_{x\in V\cap S^{n-1}}\norm{\Pi x}_2$.  These
extensions are not pursued in this paper and are intended to be explored in
future work.
\end{remark}

\subsection{Technical Overview}
In this subsection, we discuss the architecture of the proof.

\subsubsection{The classical Kahn--Szemeredi argument}
\label{sec:technical-overview-kahn-szemeredi}

The Kahn--Szemeredi argument is a standard non-asymptotic tool 
in estimating the spectral norm of sparse random matrices with bounded entries,
and is a starting point of our investigation.
The argument originates in the work \cite{FriedmanKahnSzemeredi89} on the
second eigenvalue of random regular graphs; further developments were obtained
by Feige and Ofek for sparse
Erd\H{o}s--R\'enyi graphs \cite{FeigeOfek05}, and by Keshavan, Montanari, and Oh for rectangular
matrices \cite{KeshavanMontanariOh10}. 

Let $\Pi$ follow the admissible sparse-entry model of
Definition~\ref{def:admissible-sparse-entry-model}, so that
$\pi_{ui}=\eta_{ui}\xi_{ui}$, where $\eta$ is the support mask of density
$p$ and the entry variables are centered and independent of the mask.  We
write
\[
    |\xi|\le1\qquad\text{almost surely}.
\]
The independent-copy symmetrization used in the proof of the
main theorem reduces the general centered case to a symmetric entry variable,
at the cost of an absolute factor; the estimates below otherwise use only the
bound $|\xi|\le1$ and the support-mask geometry.

In the context of the present paper, the objective of the Kahn--Szemeredi argument
is the supremum of the bilinear form 
\[
    B_\Pi(x,y)
    :=\langle y,\Pi x\rangle
    =\sum_{u=1}^k\sum_{i=1}^n y_u\pi_{ui}x_i,
    \qquad
    x\in V\cap S^{n-1},\quad y\in S^{k-1},
\]
which, in view of standard variational formulas, coincides with 
$\|\Pi U_V\|$.  Constant-resolution nets
${\mathcal N}_V\subset V\cap S^{n-1}$ and
$\mathcal N_k\subset S^{k-1}$, of cardinalities $\exp(O(r))$ and
$\exp(O(k))$, reduce the problem, up to an absolute factor, to pairs
$(x,y)\in{\mathcal N}_V\times\mathcal N_k$. 

Put $\rho:=\sqrt{p/k}$.  For a fixed pair $(x,y)$, consider the
decomposition
\[
    B_\Pi(x,y)
    =\sum_{|x_i y_u|\le\rho}y_u\pi_{ui}x_i
     +\sum_{|x_i y_u|>\rho}y_u\pi_{ui}x_i
    =:B_{\rm light}(x,y)+B_{\rm heavy}(x,y).
\]
The supremum of $B_{\rm light}(x,y)$ ({\it light couples}) is expected to be dealt with
using standard Bernstein--type inequalities, at scale $\sqrt{kp}$.
The {\it heavy couples} require a more elaborate treatment.
Decompose the two
vectors into dyadic level sets
\[
    I_j(x):=\{i:2^{-j}<|x_i|\le 2^{-j+1}\},
    \qquad
    J_\ell(y):=\{u:2^{-\ell}<|y_u|\le 2^{-\ell+1}\}.
\]
Only pairs of levels with $2^{-j-\ell+2}>\rho$ occur, and, writing
\[
    e_\eta(I,J):=\sum_{u\in J,\ i\in I}\eta_{ui}
\]
for the number of support edges in a rectangle, the entry bound gives the
deterministic
majorization
\[
    |B_{\rm heavy}(x,y)|
    \le
    \sum_{2^{-j-\ell+2}>\rho}
    2^{-j-\ell+2}e_\eta(I_j(x),J_\ell(y)),
\]
where the sum is over $j,\ell\in\mathbb Z$.
In the unrestricted-coordinate version of the argument
\cite{FriedmanKahnSzemeredi89,FeigeOfek05,KeshavanMontanariOh10}, the
summation is controlled via the edge discrepancy estimate
\[
    e_\eta(I,J)
    \lesssim
    p\,|I|\,|J|+\min\left\{
        |I|\, |J|,\,kp |I|,\,np |J|,\,
        \frac{Q}{\log(e+Q/(p\,|I|\,|J|))}
    \right\},
\]
where
\[
    Q=|I|\log\frac{en}{|I|}+|J|\log\frac{ek}{|J|}.
\]
Indeed, $I\subset[n]$ is a set of columns and $J\subset[k]$ is a set of
rows: the two degree bounds are therefore $kp|I|$ and $np|J|$, respectively.
This light/heavy architecture is the
common thread in \cite{FriedmanKahnSzemeredi89,FeigeOfek05,
KeshavanMontanariOh10}.

In our setting of $n$ being possibly much larger than $r$,
a direct adaptation of the heavy-couples summation argument
would require the
estimate $e_\eta(I,J)\lesssim p k\,|J|$, where $I$ is an $x$-level and $J$
is a $y$-level.
Such an estimate is not available (is false) in our setting, leading
to a blow-up of the dyadic summation formula.
The issue is deeper than high edge discrepancy:
it can be easily checked that without the assumption that
the discrete set ${\mathcal N}_V$ is confined within an $r$--dimensional
linear subspace and using only cardinality bounds for ${\mathcal N}_V$,
the corresponding supremum of bilinear forms can be much larger than
$O(\sqrt{kp})$,
i.e for the proof to close, the linear structure of $V$ must
be exploited in an essential way.

\subsubsection{Entropy of coordinate level sets}
\label{sec:technical-overview-entropy}

The required use of the linear structure of $V$ enters through an {\it entropy
bound} for coordinate level sets.  Fix an $r$--dimensional subspace
$V\subset\R^n$.  For $\beta>0$ and an integer $s\ge1$, consider the family
\[
    \mathcal F_V(\beta,s)
    :=
    \left\{
        I\subset[n]: |I|=s,\ \exists x\in V,\ \norm{x}_2\le1,\
        |x_i|\ge\beta\ \text{for every }i\in I
    \right\}.
\]
Thus, $\mathcal F_V(\beta,s)$ comprises all $s$--element subsets of $[n]$
which are contained in a ``$\geq\beta$--level'' set of some unit vector $x\in V$.
The trivial bound $|\mathcal F_V(\beta,s)|\le\binom ns$ ignores the
subspace and the level set condition, 
and is useless when $n$ is much bigger than $r$.  The entropy
bound, verified in Lemma~\ref{lem:full-trace-counting} of this paper, is
\begin{equation}
\label{eq:technical-overview-full-trace-entropy}
    |\mathcal F_V(\beta,s)|
    \le
    \frac{\beta^{-2s}}{s!}\binom{r+s-1}{s}
    \le
    \left(
        C\,\frac{r+s}{\beta^2s^2}
    \right)^s.
\end{equation}
Note that the estimate is independent of the ambient dimension $n$.
It counts only those coordinate patterns that can actually be realized by a
unit vector in $V$.

The exact first bound in
\eqref{eq:technical-overview-full-trace-entropy} 
is proved by observing that for every $I$ in $\mathcal F_V(\beta,s)$,
there is a unit vector $x$ in $V$ satisfying $\big|\prod_{i\in I} \langle x,P_V e_i\rangle\big|\geq \beta^s$,
where $P_V:\R^n\to\R^n$ is the orthogonal projection onto $V$.
Thus, counting the subsets $I$ can be reduced to estimating the number
of homogeneous degree $s$ polynomials in $r$ variables satisfying certain
point estimates, which can further be interpreted as a lower bound condition for their
Bombieri--Weyl norms. A key part in the proof is played
by the Parseval identity $\sum_{i=1}^n P_V e_i\,\otimes \,P_V e_i=I_V$,
which encodes the linear subspace structure into the argument.
The identity enables an exact estimate on the total squared Bombieri--Weyl
norm of these polynomials over all ordered $s$--tuples, namely
$\binom{r+s-1}{s}$; passing to unordered sets produces the additional factor $1/s!$.
This identity, which has no analogue for an arbitrary collection of vectors known
only through its cardinality, is the precise point at which linearity of $V$ is
exploited.

The estimate \eqref{eq:technical-overview-full-trace-entropy} is nearly optimal, up to a factor $C^s$, throughout the natural
range $\beta^2s\le1$. Indeed, assume $s\leq r$ and $n\geq r(\beta^2s)^{-1}$, let
$m=\lfloor(\beta^2s)^{-1}\rfloor$ and take $V$ to be spanned by the
indicators of $r$ disjoint blocks of size $m$.  Choosing one coordinate from
each of $s$ distinct blocks shows that
\[
    |\mathcal F_V(\beta,s)|
    \ge \binom rs m^s
    \ge \left(c\,\frac{r}{\beta^2s^2}\right)^s.
\]

\medskip

The entropy estimate \eqref{eq:technical-overview-full-trace-entropy} partially resolves the edge-discrepancy obstruction
described above.  For a coordinate level $I$ of size $s$ and height $\beta$,
a union bound no longer pays for all $\binom ns$ possible column sets, but
only for the members of $\mathcal F_V(\beta,s)$; equivalently, the ambient
cost $s\log(en/s)$ is replaced by an intrinsic cost depending on $r$, $s$,
and $\beta$.  Combined with fixed-rectangle tail estimates, this makes it
possible to control discrepancy simultaneously over the coordinate levels
that can actually arise from vectors in $V$.  It does not, however, control
the local degree of each individual row into such a level.  A small number of
very heavy rows may still collect far more than $pr$ (and even $pk$) incidences, so the
missing estimate for $e_\eta(I,J)$ cannot be recovered
from entropy alone.  The proof must therefore separate these
exceptional row--level interactions from the remainder: the latter is
accessible to the Kahn--Szemeredi discrepancy argument, whereas the former
requires a different heavy-row mechanism.  This is the reason for introducing
the {\it Tall--Flat decomposition}.

\subsubsection{Tall--Flat decomposition}
\label{sec:technical-overview-tall-flat}

We now describe the vector-dependent decomposition that repairs the missing
row-side degree estimate in the preceding discussion.  We first give the
definitions in their general form.  Let
$\eta=(\eta_{ui})\in\{0,1\}^{k\times n}$ be a mask.  For a row $u\in[k]$
and a set $J\subset[n]$, write
\[
    d_u(\eta,J):=\sum_{i\in J}\eta_{ui}.
\]
Thus the degree is computed from the designated support mask, independently
of whether an entry amplitude vanishes.  When the mask is fixed, its
dependence is suppressed from the notation for the Tall and Flat objects.
For $x\in\R^n$ and a subset $\mathcal J\subset[n]$, write
\[
    I_j(x):=\{i\in[n]:2^{-j}<\abs{x_i}\le2^{-j+1}\},
    \qquad
    I_j^{\mathcal J}(x):=I_j(x)\cap\mathcal J.
\]

\begin{definition}[Tall contribution matrix]
Let $\eta\in\{0,1\}^{k\times n}$ be a mask, and let
$\Pi=(\pi_{ui})$ be a $k\times n$ matrix supported on $\eta$, in the sense
that $\pi_{ui}=0$ whenever $\eta_{ui}=0$.  Let
\[
    \mathfrak Q=\bigl((J_q,K_q)\bigr)_{q=1}^m
\]
be a finite sequence of pairs, where the sets $J_q\subset[n]$ are pairwise
disjoint and the real numbers $K_q$ are positive.  Define the
tall contribution matrix $\mathcal T(\Pi,\mathfrak Q)$ by
\[
    \bigl(\mathcal T(\Pi,\mathfrak Q)\bigr)_{u i}
    :=
    \begin{cases}
    0, & \text{if there is }1\le q\le m\text{ such that }
        i\in J_q\text{ and }d_u(\eta,J_q)\ge K_q,\\
    \pi_{ui}, & \text{otherwise}.
    \end{cases}
\]
Equivalently, $\mathcal T(\Pi,\mathfrak Q)$ is obtained from $\Pi$ by zeroing out
every entry $\pi_{ui}$ for which $i\in J_q$ for some $q$ and row $u$ is
$K_q$-heavy in the mask $\eta$ with respect to $J_q$, while leaving all other
entries unchanged.
In particular, entries outside $\bigcup_{q=1}^mJ_q$ are unchanged.
\end{definition}

\begin{definition}[Flat contribution matrix]
Under the assumptions of the preceding definition, define the flat
contribution matrix by
\[
    \bigl(\mathcal F(\Pi,\mathfrak Q)\bigr)_{u i}
    :=
    \begin{cases}
    \pi_{ui}, & \text{if there is }1\le q\le m\text{ such that }
        i\in J_q\text{ and }d_u(\eta,J_q)\ge K_q,\\
    0, & \text{otherwise}.
    \end{cases}
\]
Thus $\mathcal F(\Pi,\mathfrak Q)$ retains precisely the row--set
interactions removed from the Tall matrix; in particular, it vanishes on
all columns outside $\bigcup_{q=1}^mJ_q$.  Entrywise,
\[
    \Pi
    =
    \mathcal T(\Pi,\mathfrak Q)
    +\mathcal F(\Pi,\mathfrak Q).
\]
\end{definition}

For the proof of the main result we use the following canonical specialization.

\begin{definition}[Canonical tall--flat partition]
\label{def:canonical-tall-flat-partition}
Fix $L\ge1$.  For $x\in\R^n$, define the threshold profile
\[
    K_j^{(L)}(x)
    :=
    L\max\{pr_*,p\abs{I_j(x)}\}
    \qquad j\in\mathbb Z ,
\]
where
$$
r_*=\max\Big(r,\frac{\log k}{p}\Big).
$$
Write $\kappa^{(L)}(x)=(K_j^{(L)}(x))_{j\in\mathbb Z}$ and
\[
    \mathfrak Q_L(x)
    :=
    \bigl(
        (I_j(x),K_j^{(L)}(x))
    \bigr)_{j\in\mathbb Z:\,I_j(x)\ne\varnothing}.
\]
The canonical tall and flat matrices associated with $x$ and $L$ are
\[
    \mathcal T_L^{\rm can}(\Pi,x):=\mathcal T(\Pi,\mathfrak Q_L(x)),
    \qquad
    \mathcal F_L^{\rm can}(\Pi,x):=\mathcal F(\Pi,\mathfrak Q_L(x))
    =\Pi-\mathcal T_L^{\rm can}(\Pi,x).
\]
Thus
\[
    \Pi x
    =\mathcal T_L^{\rm can}(\Pi,x)x
     +\mathcal F_L^{\rm can}(\Pi,x)x.
\]
\end{definition}

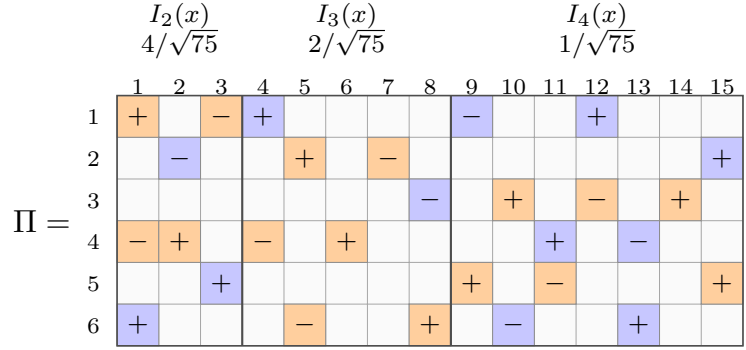
\begin{figure}[htbp]
\centering
\begin{tikzpicture}[x=0.46cm,y=0.46cm,scale=1.2,transform shape]
    \fill[gray!4] (0,-6) rectangle (15,0);
    \draw[step=1,black!35,line width=0.3pt] (0,-6) grid (15,0);

    \foreach \i/\u/\sgn in {
        4/1/+,9/1/-,12/1/+,
        2/2/-,15/2/+,
        8/3/-,
        11/4/+,13/4/-,
        3/5/+,
        1/6/+,10/6/-,13/6/+
    }{
        \filldraw[fill=blue!22,draw=black!45,line width=0.3pt]
            ({\i-1},{-\u}) rectangle (\i,{1-\u});
        \node[font=\scriptsize] at ({\i-0.5},{0.5-\u}) {$\sgn$};
    }

    \foreach \i/\u/\sgn in {
        1/1/+,3/1/-,
        5/2/+,7/2/-,
        10/3/+,12/3/-,14/3/+,
        1/4/-,2/4/+,4/4/-,6/4/+,
        9/5/+,11/5/-,15/5/+,
        5/6/-,8/6/+
    }{
        \filldraw[fill=orange!38,draw=black!45,line width=0.3pt]
            ({\i-1},{-\u}) rectangle (\i,{1-\u});
        \node[font=\scriptsize] at ({\i-0.5},{0.5-\u}) {$\sgn$};
    }

    \draw[black!70,line width=0.8pt] (3,0) -- (3,-6);
    \draw[black!70,line width=0.8pt] (8,0) -- (8,-6);
    \draw[black!70,line width=0.6pt] (0,-6) rectangle (15,0);

    \foreach \i in {1,...,15}
        \node[font=\tiny] at ({\i-0.5},0.20) {\i};
    \foreach \u in {1,...,6}
        \node[font=\tiny,anchor=e] at (-0.12,{0.5-\u}) {\u};
    \node[font=\small] at (-1.8,-3) {$\Pi=$};

    \node[font=\scriptsize,align=center] at (1.5,1.55)
        {$I_2(x)$\\[-1pt]$4/\sqrt{75}$};
    \node[font=\scriptsize,align=center] at (5.5,1.55)
        {$I_3(x)$\\[-1pt]$2/\sqrt{75}$};
    \node[font=\scriptsize,align=center] at (11.5,1.55)
        {$I_4(x)$\\[-1pt]$1/\sqrt{75}$};

\end{tikzpicture}
\caption{A canonical Tall--Flat partition for a fixed vector and matrix
realization.  Here $k=6$, $n=15$, $r=3$,
$p=(\log k)/k=(\log6)/6$, $L=1$, and
$x=75^{-1/2}(4,4,4,2,2,2,2,2,1,1,1,1,1,1,1)$.  Thus $r_*=6$ and the
integer row-degree cutoffs for $I_2(x),I_3(x),I_4(x)$ are, respectively,
$2,2,3$.  A nonzero entry is orange precisely when its row meets the cutoff
on the corresponding level; these entries form the Flat matrix.  The blue
entries form the Tall matrix, blank cells are zero, and the signs specify the
chosen realization of $\Pi$.}
\label{fig:canonical-tall-flat-example}
\end{figure}

The triangle inequality gives
\begin{equation}
\label{eq:technical-overview-canonical-tall-flat}
    \norm{\Pi x}_2
    \le
    \norm{\mathcal T_L^{\rm can}(\Pi,x)x}_2
    +\norm{\mathcal F_L^{\rm can}(\Pi,x)x}_2.
\end{equation}
Consequently, if ${\mathcal N}_V\subset V\cap S^{n-1}$ is a $1/2$-net, then
\begin{equation}
\label{eq:technical-overview-canonical-net-reduction}
    \norm{\Pi U_V}
    \le
    2\sup_{z\in{\mathcal N}_V}
        \norm{\mathcal T_L^{\rm can}(\Pi,z)z}_2
    +2\sup_{z\in{\mathcal N}_V}
        \norm{\mathcal F_L^{\rm can}(\Pi,z)z}_2.
\end{equation}

\begin{figure}[htbp]
\centering
\begin{tikzpicture}[
    proofarrow/.style={-{Latex[length=2mm]},thick},
    rootnode/.style={draw,rounded corners=2pt,fill=blue!12,
        minimum width=2.6cm,minimum height=0.9cm,align=center,font=\small},
    branchnode/.style={draw,rounded corners=2pt,fill=blue!6,
        text width=4.1cm,minimum height=1.25cm,align=center,font=\small,
        inner sep=5pt},
    flatnode/.style={draw,rounded corners=2pt,fill=gray!14,
        text width=4.1cm,minimum height=1.25cm,align=center,font=\small,
        inner sep=5pt},
    leafnode/.style={draw,rounded corners=2pt,fill=green!7,
        text width=5.0cm,minimum height=1.65cm,align=center,font=\small,
        inner sep=5pt},
    couplenode/.style={draw,rounded corners=2pt,fill=green!7,
        text width=6.8cm,minimum height=2.0cm,align=center,font=\scriptsize,
        inner sep=5pt}
]
\node[rootnode] (root) at (0,0) {$\Pi\, x$};

\node[branchnode] (tall) at (-2.5,-1.7)
    {\textbf{Canonical Tall}\\[-1pt]
     $\mathcal T_L^{\rm can}(\Pi,x)\,x$};
\node[flatnode] (flat) at (2.5,-1.7)
    {\textbf{Canonical Flat}\\[-1pt]
     $\mathcal F_L^{\rm can}(\Pi,x)\,x$};

\node[leafnode] (lightcol) at (-4.5,-4.3)
    {\textbf{Light columns}\\[1pt]
     $\mathcal T_L^{\rm can}(\Pi,x)\,x_{\mathcal J(x)^c}$,\\[5pt]
     {\scriptsize classical Bernstein--type row\\
     concentration is sufficient to\\
     union-bound over all $x$ in the net}};
\node[branchnode] (heavycol) at (0.5,-4.3)
    {\textbf{Heavy columns}\\[1pt]
     $\mathcal T_L^{\rm can}(\Pi,x)\,x_{\mathcal J(x)}$};

\node[couplenode] (lightcouple) at (-3.65,-7.0)
    {\textbf{Light couples}\\[1pt]
     $\bigl(\mathcal T_L^{\rm can}(\Pi,x)\bigr)_{ui} x_i\,y_u$, $i\in \mathcal J(x)$, $|x_i y_u|\leq\rho$\\[5pt]
     Kahn--Szemeredi:\\
     Bernstein concentration,\\
     union bound over nets};
\node[couplenode] (heavycouple) at (3.65,-7.0)
    {\textbf{Heavy couples}\\[1pt]
     $\bigl(\mathcal T_L^{\rm can}(\Pi,x)\bigr)_{ui} x_i\,y_u$, $i\in \mathcal J(x)$, $|x_i y_u|>\rho$\\[5pt]
     modified Kahn--Szemeredi:\\
     level set entropy, edge discrepancy,\\
     and dyadic summation};

\draw[proofarrow] (root) -- (tall);
\draw[proofarrow] (root) -- (flat);
\draw[proofarrow] (tall) -- (lightcol);
\draw[proofarrow] (tall) -- (heavycol);
\draw[proofarrow] (heavycol) -- (lightcouple);
\draw[proofarrow] (heavycol) -- (heavycouple);
\end{tikzpicture}
\caption{High-level organization of the proof. The bound on the supremum of $\|\Pi x\|_2$ over the net is obtained via the Tall--Flat canonical decomposition, followed
by the modified Kahn--Szemeredi argument for the Tall contribution, relying on the level set entropy bounds. Here
$\mathcal J(x)=\{i\leq n:\;|x_i|>r_*^{-1/2}\}$ and $\rho=\sqrt{p/k}$.}
\label{fig:high-level-partition}
\end{figure}
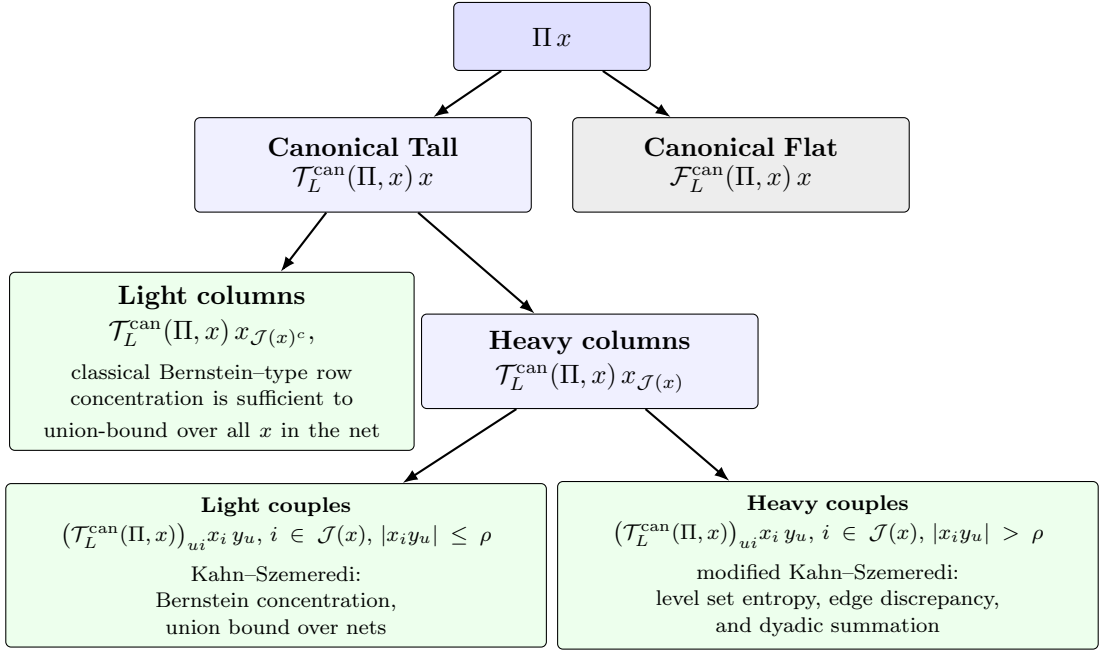

We next argue why the first term in
\eqref{eq:technical-overview-canonical-net-reduction} is accessible to the
Kahn--Szemeredi method.  Let
\[
    \mathcal J(x):=\{i\in[n]:|x_i|>r_*^{-1/2}\}.
\]
The Tall contribution of the complementary, small coordinates
$x_{\mathcal J(x)^c}$ ({\it light columns}) is controlled by a row-wise
concentration argument at scale $\sqrt{kp}$.
For the coordinates $x_{\mathcal J(x)}$, one bilinearizes against
$y\in S^{k-1}$ and applies the light/heavy-couple split from the preceding
subsection. The light couples, for which
$|x_i y_u|\le\rho=\sqrt{p/k}$, are controlled by Bernstein's inequality and
a union bound over the two nets, again at scale $\sqrt{kp}$.

For the heavy couples, let
\[
    I_j(x):=\{i:2^{-j}<|x_i|\le 2^{-j+1}\},
    \qquad
    J_\ell(y):=\{u:2^{-\ell}<|y_u|\le 2^{-\ell+1}\}.
\]
The support-edge count remaining in the
canonical Tall matrix is
\[
    e_{\rm can}^{(L)}(I_j(x),J_\ell(y))
    :=
    \sum_{u\in J_\ell(y)}d_u(\eta,I_j(x))
        \mathbf 1_{\{d_u(\eta,I_j(x))<K_L(I_j(x))\}},
    \quad
    K_L(I_j(x)):=L\max\{pr_*,p|I_j(x)|\}.
\]
By construction,
\begin{equation}
\label{eq:technical-overview-canonical-row-envelope}
    e_{\rm can}^{(L)}(I_j(x),J_\ell(y))
    \le K_L(I_j(x))|J_\ell(y)|.
\end{equation}
If $2^{-j+1}\ge r_*^{-1/2}$, then
$|I_j(x)|\le2^{2j}\le4r_*$, and hence
\[
    K_L(I_j(x))\le4Lpr_*.
\]
Thus \eqref{eq:technical-overview-canonical-row-envelope} gives exactly the
missing row-side estimate
\[
    e_{\rm can}^{(L)}(I_j(x),J_\ell(y))
    \lesssim_L pr_*|J_\ell(y)|.
\]
The remaining ingredients reproduce the other components of the classical
rectangle envelope, with the {\it effective coordinate scale} $r_*$ replacing the unavailable ambient
scale $n$.  The standard dyadic summation now
applies and bounds the heavy-couple contribution by $O_L(\sqrt{kp})$.
Together with the light-column and light-couple estimates, it gives, for
every $D>0$ and every deterministic ${\mathcal N}_V\subset V\cap S^{n-1}$ with
$|{\mathcal N}_V|\le5^r$,
\[
    \mathbb P\left\{
        \sup_{z\in{\mathcal N}_V}
        \norm{\mathcal T_L^{\rm can}(\Pi,z)z}_2
        >C(D,L)\sqrt{kp}
    \right\}
    \le k^{-D}.
\]
In view of \eqref{eq:technical-overview-canonical-net-reduction}, the Tall
matrix is therefore completely handled by the Kahn--Szemeredi argument.  The
remaining task is to control
$\norm{\mathcal F_L^{\rm can}(\Pi,z)z}_2$ uniformly over the net.

\subsubsection{Flat contributions}
\label{sec:technical-overview-flat-contribution}

The Flat matrix is signed and depends on the realization of the entries of
$\Pi$.  For its analysis we use a nonnegative row profile which records only
the mask geometry of the heavy row--set interactions. 

\begin{definition}[Flat row profile and majorant]
\label{def:flat-row-profile-majorant}
Let $\eta\in\{0,1\}^{k\times n}$ be a mask.  Let
\[
    \mathfrak P=\bigl((J_q,\omega_q,K_q)\bigr)_{q=1}^m
\]
be a finite sequence of triples, where the sets $J_q\subset[n]$ are pairwise
disjoint, the weights $\omega_q$ are positive, and the thresholds $K_q$ are
positive.  Define the Flat row profile $\mathbf m(\eta,\mathfrak P)\in\R^k$ by
\[
    \bigl(\mathbf m(\eta,\mathfrak P)\bigr)_u
    :=
        \sum_{\substack{1\le q\le m\\ d_u(\eta,J_q)\ge K_q}}
        \omega_q d_u(\eta,J_q),
    \qquad u\in[k],
\]
and define the associated Flat majorant by
\[
    \FM(\eta,\mathfrak P):=\norm{\mathbf m(\eta,\mathfrak P)}_2.
\]
If the mask is clear from context, we write $\mathbf m(\mathfrak P)$ and $\FM(\mathfrak P)$.
\end{definition}

\begin{definition}[Canonical Flat profile and majorant]
\label{def:canonical-flat-profile-majorant}
For the canonical threshold profile, define
\[
    \mathfrak P_L^{\rm can}(x)
    :=
    \bigl(I_j(x),2^{-j},K_j^{(L)}(x)\bigr)_{
        j\in\mathbb Z:\,I_j(x)\ne\varnothing},
\]
and write
\[
    \mathbf m_L^{\rm can}(\eta,x)
    :=\mathbf m\bigl(\eta,\mathfrak P_L^{\rm can}(x)\bigr),
    \qquad
    \FM_L^{\rm can}(\eta,x)
    :=\FM\bigl(\eta,\mathfrak P_L^{\rm can}(x)\bigr)
    =\norm{\mathbf m_L^{\rm can}(\eta,x)}_2.
\]
\end{definition}

\bigskip

Taking the Euclidean norm gives
\begin{equation}
\label{eq:technical-overview-flat-contribution-bound}
    \norm{\mathcal F_L^{\rm can}(\Pi,x)x}_2
    \le
    2\FM_L^{\rm can}(\eta,x).
\end{equation}
Indeed, the row profile replaces $|\pi_{ui}|$ by the mask entry $\eta_{ui}$
and $|x_i|$ on $I_j(x)$ by its lower dyadic scale $2^{-j}$; the factor two
in the display accounts for the upper endpoint of the level.

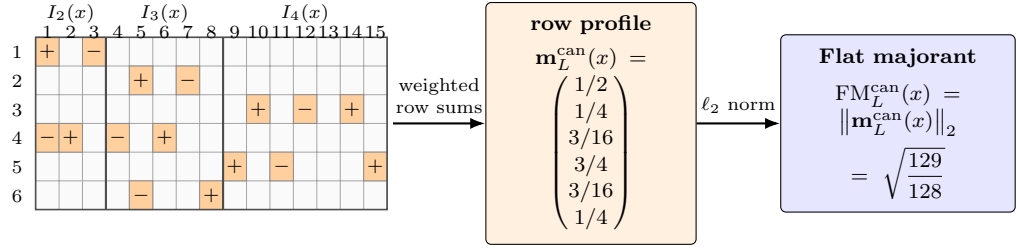
\begin{figure}[htbp]
\centering
\begin{tikzpicture}[x=0.31cm,y=0.38cm]
    \fill[gray!4] (0,-6) rectangle (15,0);
    \draw[step=1,black!35,line width=0.3pt] (0,-6) grid (15,0);
    \foreach \i/\u/\sgn in {
        1/1/+,3/1/-,
        5/2/+,7/2/-,
        10/3/+,12/3/-,14/3/+,
        1/4/-,2/4/+,4/4/-,6/4/+,
        9/5/+,11/5/-,15/5/+,
        5/6/-,8/6/+
    }{
        \filldraw[fill=orange!38,draw=black!45,line width=0.3pt]
            ({\i-1},{-\u}) rectangle (\i,{1-\u});
        \node[font=\tiny] at ({\i-0.5},{0.5-\u}) {$\sgn$};
    }
    \draw[black!70,line width=0.8pt] (3,0) -- (3,-6);
    \draw[black!70,line width=0.8pt] (8,0) -- (8,-6);
    \draw[black!70,line width=0.6pt] (0,-6) rectangle (15,0);
    \foreach \i in {1,...,15}
        \node[font=\tiny] at ({\i-0.5},0.20) {\i};
    \foreach \u in {1,...,6}
        \node[font=\tiny,anchor=e] at (-0.12,{0.5-\u}) {\u};
    \node[font=\tiny] at (1.5,0.85) {$I_2(x)$};
    \node[font=\tiny] at (5.5,0.85) {$I_3(x)$};
    \node[font=\tiny] at (11.5,0.85) {$I_4(x)$};

    \node[draw,rounded corners=2pt,fill=orange!12,
        text width=2.5cm,minimum height=3.2cm,align=center,
        inner sep=4pt,font=\scriptsize] (profile) at (23.7,-3)
        {\textbf{row profile}\\[3pt]
         $\displaystyle
         \mathbf m_L^{\rm can}(x)=
         \begin{pmatrix}
             1/2\\ 1/4\\ 3/16\\ 3/4\\ 3/16\\ 1/4
         \end{pmatrix}$};

    \node[draw,rounded corners=2pt,fill=blue!10,
        text width=2.75cm,minimum height=1.8cm,align=center,
        inner sep=5pt,font=\scriptsize] (majorant) at (36.8,-3)
        {\textbf{Flat majorant}\\[4pt]
         $\FM_L^{\rm can}(x)
          =\norm{\mathbf m_L^{\rm can}(x)}_2$\\[3pt]
         $\displaystyle=\sqrt{\frac{129}{128}}$};

    \draw[-{Latex[length=2mm]},thick]
        (15.25,-3) -- node[midway,above,font=\tiny,align=center]
        {weighted\\row sums} (profile.west);
    \draw[-{Latex[length=2mm]},thick]
        (profile.east) -- node[midway,above,font=\tiny] {$\ell_2$ norm}
        (majorant.west);
\end{tikzpicture}
\caption{Construction of the canonical Flat row profile and majorant for the
Flat matrix in Figure~\ref{fig:canonical-tall-flat-example}.  The signs are
discarded, entries from $I_j(x)$ receive weight $2^{-j}$, and the weighted
degrees are summed within each row.  For example, row $4$ contains two Flat
entries in $I_2(x)$ and two in $I_3(x)$, giving
$2\cdot\frac14+2\cdot\frac18=\frac34$.  Taking the Euclidean norm of the six
row values gives $\FM_L^{\rm can}(x)=\sqrt{129/128}$.}
\label{fig:flat-row-profile-majorant}
\end{figure}

Combining \eqref{eq:technical-overview-flat-contribution-bound} with
\eqref{eq:technical-overview-canonical-net-reduction} gives
\[
    \norm{\Pi U_V}
    \le
    2\sup_{z\in{\mathcal N}_V}
        \norm{\mathcal T_L^{\rm can}(\Pi,z)z}_2
    +4\sup_{z\in{\mathcal N}_V}\FM_L^{\rm can}(\eta,z).
\]

In the actual proof we replace the canonical Flat profile by a reduced
profile and majorant, introduced in Section~\ref{sec:flat-majorants}, for
technical reasons.  The reduced thresholds depend only on the coordinate
scale, rather than on the individual vector, which makes the probabilistic
estimates uniform.  The comparison with the canonical profile costs only
$O(\sqrt{pr})$ and is therefore harmless at the target scale
$O(\sqrt{kp})$.

The mechanism by which entropy rules out a large Flat contribution is
clearest in the balanced case.  Fix a coordinate scale $\beta$, let
$I\in\mathcal F_V(\beta,s)$, and suppose, for simplicity, that the
supports of $m$ rows restricted to $I$ form a disjoint partition
$I=E_1\sqcup\cdots\sqcup E_m$, where $m=s/K$, every $E_q$ has cardinality
$K$, and the sets are assigned to distinct rows. We assume that $K\gtrsim\log k$. Thus the supporting rows
account for exactly $s$ selected support incidences, with no overlap. For a
fixed $I$, the number of choices of the rows and the partition is at most
$k^m s!/(K!)^m$; since each prescribed incidence costs a factor $p$ and
$K\gtrsim\log k$, the probability of such a configuration is bounded by
\[
    k^m\frac{s!}{(K!)^m}p^s
    \le
    \left(C\,\frac{ps}{K}\right)^s.
\]
The entropy bound gives
$|\mathcal F_V(\beta,s)|\le(C(r+s)\beta^{-2}/s^2)^s$.  On the other hand, the
corresponding squared row-profile contribution is bounded above by
\[
    m(\beta K)^2=\beta^2sK.
\]
Consequently, multiplying probability by entropy yields
\[
    \left(C\,\frac{(r+s)\beta^{-2}}{s^2}\right)^s
    \left(C\,\frac{ps}{K}\right)^s
    =
    \left(C\,\frac{p(r+s)}{\beta^2sK}\right)^s.
\]
Thus a squared contribution $\beta^2sK\gg p(r+s)$ is exponentially unlikely
after the union bound; at this stage the resulting norm scale is
$\sqrt{p(r+s)}$, not $\sqrt{pr}$.  The linear restriction
$s=O(r)$, established later for the selected partition, reduces this to the
target scale $\sqrt{pr}$.  A direct argument as described here would lead to
logarithmic losses; the actual proof groups unequal
row supports and different coordinate levels dyadically, but uses
the same cancellation mechanism.

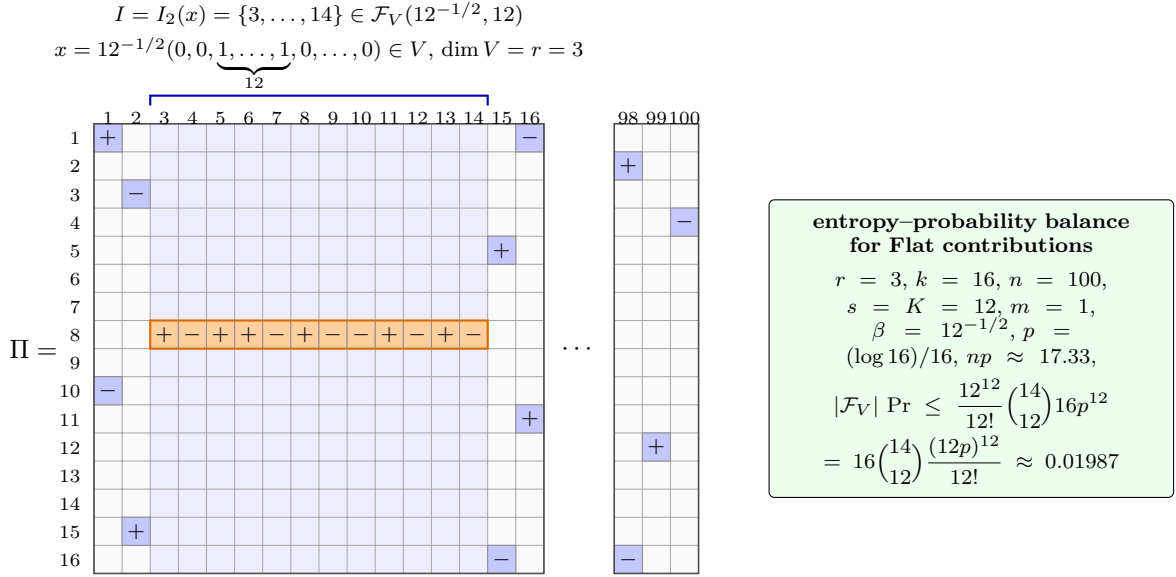
\begin{figure}[htbp]
\centering
\begin{tikzpicture}[x=0.372cm,y=0.372cm]
    \fill[gray!4] (0,-16) rectangle (16,0);
    \fill[blue!7] (2,-16) rectangle (14,0);
    \draw[step=1,black!35,line width=0.3pt] (0,-16) grid (16,0);

    \foreach \i/\u/\sgn in {
        1/1/+,16/1/-,
        2/3/-,15/5/+,
        1/10/-,16/11/+,
        2/15/+,15/16/-
    }{
        \filldraw[fill=blue!22,draw=black!45,line width=0.3pt]
            ({\i-1},{-\u}) rectangle (\i,{1-\u});
        \node[font=\scriptsize] at ({\i-0.5},{0.5-\u}) {$\sgn$};
    }

    \foreach \i/\u/\sgn in {
        3/8/+,4/8/-,5/8/+,6/8/+,
        7/8/-,8/8/+,9/8/-,10/8/-,
        11/8/+,12/8/-,13/8/+,14/8/-
    }{
        \filldraw[fill=orange!38,draw=black!45,line width=0.3pt]
            ({\i-1},{-\u}) rectangle (\i,{1-\u});
        \node[font=\scriptsize] at ({\i-0.5},{0.5-\u}) {$\sgn$};
    }

    \draw[orange!85!black,line width=0.8pt] (2,-8) rectangle (14,-7);

    \draw[black!70,line width=0.6pt] (0,-16) rectangle (16,0);

    \node[font=\small] at (17.25,-8) {$\cdots$};
    \fill[gray!4] (18.5,-16) rectangle (21.5,0);
    \foreach \x in {19.5,20.5}
        \draw[black!35,line width=0.3pt] (\x,-16) -- (\x,0);
    \foreach \y in {-1,...,-15}
        \draw[black!35,line width=0.3pt] (18.5,\y) -- (21.5,\y);
    \foreach \i/\u/\sgn in {
        1/2/+,3/4/-,
        2/12/+,1/16/-
    }{
        \filldraw[fill=blue!22,draw=black!45,line width=0.3pt]
            ({17.5+\i},{-\u}) rectangle ({18.5+\i},{1-\u});
        \node[font=\scriptsize] at ({18+\i},{0.5-\u}) {$\sgn$};
    }
    \draw[black!70,line width=0.6pt] (18.5,-16) rectangle (21.5,0);

    \foreach \i in {1,...,16}
        \node[font=\tiny] at ({\i-0.5},0.22) {\i};
    \foreach \i/\lab in {1/98,2/99,3/100}
        \node[font=\tiny] at ({18+\i},0.22) {\lab};
    \foreach \u in {1,...,16}
        \node[font=\tiny,anchor=e] at (-0.12,{0.5-\u}) {\u};

    \node[font=\small] at (-2.15,-8) {$\Pi=$};

    \draw[blue!65!black,line width=0.8pt]
        (2,0.65) -- (2,1.0) -- (14,1.0) -- (14,0.65);
    \node[font=\scriptsize,align=center] at (8,2.05)
        {$I=I_2(x)=\{3,\ldots,14\}\in
         \mathcal F_V(12^{-1/2},12)$\\[3pt]
         $x=12^{-1/2}(0,0,\underbrace{1,\ldots,1}_{12},0,\ldots,0)\in V$,
         $\dim V=r=3$\\[14pt]};

    \node[draw,rounded corners=2pt,fill=green!7,
        text width=5.0cm,minimum height=3.0cm,align=center,
        inner sep=5pt,font=\scriptsize] (balance) at (31.2,-8.0)
        {\textbf{entropy--probability balance}\\[-1pt]
         \textbf{for Flat contributions}\\[3pt]
         $r=3$, $k=16$, $n=100$, $s=K=12$, $m=1$,\\[-1pt]
         $\beta=12^{-1/2}$, $p=(\log16)/16$, $np\approx17.33$,\\[3pt]
         $\displaystyle
         |\mathcal F_V|\,\Pr
         \le
         \frac{12^{12}}{12!}\binom{14}{12}
         16p^{12}$\\[2pt]
         $\displaystyle
         =16\binom{14}{12}\frac{(12p)^{12}}{12!}
         \approx 0.01987$};
\end{tikzpicture}
\caption{The entropy--probability balance illustration.
The blue band is a coordinate level
$I\in\mathcal F_V(12^{-1/2},12)$ for an $r=3$ dimensional subspace
$V\subset\mathbb R^{100}$.  The first $16$ columns and the final three
columns of $\Pi$ are displayed, while the dots represent columns
$17,\ldots,97$.  The orange incidences form
$E_1=I$, so one row meets $I$ in $K=12$ coordinates.  Although the level set may vary over the entropy
class $\mathcal F_V$, producing this heavy-row alignment for any fixed
$I$ has probability at most
$16p^{12}\approx1.173\cdot10^{-8}$ for
$k=16$, $n=100$, and $p=(\log16)/16$.
The first bound in \eqref{eq:technical-overview-full-trace-entropy} gives
$|\mathcal F_V|\le 12^{12}\binom{14}{12}/12!$; multiplying the two bounds
gives a total of roughly $0.02$, as displayed in the green box.}
\label{fig:balanced-entropy-probability}
\end{figure}

\bigskip

\noindent\textbf{Funding acknowledgement.}
K.T. was partially supported by NSF grant DMS 2452120.

\medskip

\noindent\textbf{Acknowledgement of AI Assistance.}
The author used ChatGPT for language editing, literature search, and
assistance in developing and checking some proof arguments during the
preparation of this manuscript. All mathematical statements, proofs, and
final wording were independently reviewed and verified by the author, who
takes full responsibility for the content of the paper.

\section{Notation and Preliminaries}

\subsection{Global notation}

\begin{definition}[Effective scale]
\label{def:effective-scale}
For parameters $k\ge r\ge3$ and $p\ge(\log k)/k$, define the effective scale
\[
    r_*:=\max\left\{r,\frac{\log k}{p}\right\}.
\]
Thus
\[
    r\le r_*\le k,
    \qquad
    pr_*=\max\{pr,\log k\}.
\]
\end{definition}

\begin{definition}[Coordinate and column restrictions]
For $J\subset[n]$ and $x\in\R^n$, write
$x_J:=(x_i)_{i\in J}\in\R^J$. By a slight abuse of notation, we will sometimes denote by $x_J$
a vector in $\R^n$ obtained from $x$ by zeroing out its components in $J^c$. If $M$ is a matrix whose columns are indexed
by $[n]$, write $M_{\cdot\, j}$ for its $j$th column and $M_J$ for its
restriction to the columns indexed by $J$.
\end{definition}

\begin{definition}[Restricted dyadic level sets]
For a vector $x\in\R^n$ and $j\in\mathbb Z$, define its $j$th dyadic level set
by
\[
    I_j(x):=\{i\in[n]:2^{-j}<\abs{x_i}\le2^{-j+1}\}.
\]
Further, if $\mathcal J\subset[n]$ is any nonempty subset, we denote the
restriction of this level set to $\mathcal J$ by
\[
    I_j^{\mathcal J}(x):=I_j(x)\cap\mathcal J.
\]
\end{definition}

\begin{definition}[Restricted row degree]
Let $\eta=(\eta_{ui})\in\{0,1\}^{k\times n}$ be a mask.  For a row $u\in[k]$
and a set $J\subset[n]$, write
\[
    d_u(\eta,J):=\sum_{i\in J}\eta_{ui}.
\]
When $\eta$ is clear from context, write $d_u(J)$.
\end{definition}

\begin{definition}[Heavy row with respect to a set]
Let $\eta\in\{0,1\}^{k\times n}$ be a mask, let $I\subset[n]$, let
$u\in[k]$, and let $\tau>0$.  We say that row $u$ is $\tau$-heavy with
respect to $I$ if
\[
    d_u(\eta,I)\ge\tau .
\]
\end{definition}

\begin{definition}[Heavy row degree sum]
\label{def:heavy-row-degree-sum}
Let $\eta\in\{0,1\}^{k\times n}$ be a mask, let $J\subset[n]$, and let
$\tau>0$.  Define
\[
    R_\tau(\eta,J):=
    \sum_{\substack{u\in[k]\\ d_u(\eta,J)\ge\tau}}d_u(\eta,J).
\]
When $\eta$ is clear from context, write $R_\tau(J)$.
\end{definition}

\subsection{Concentration of Negatively Associated Variables}

\begin{lemma}[Upper-tail calculus for admissible support masks]
\label{lem:negative-association-calculus}
Let $\eta=(\eta_q)_{q\in\mathcal Q}$ be a finite negatively associated
family of $\{0,1\}$-valued variables with common mean $p$.  Then the
following statements hold.

First, if $Q_1,\ldots,Q_m\subset\mathcal Q$ are pairwise disjoint and
$F_1,\ldots,F_m$ are nonnegative coordinate-wise non-decreasing functions on $\R^{Q_1},\dots,\R^{Q_m}$, respectively,
then
\begin{equation}
\label{eq:negative-association-block-product}
    \mathbb E\prod_{\ell=1}^m F_\ell(\eta_{Q_\ell})
    \le
    \prod_{\ell=1}^m\mathbb EF_\ell(\eta_{Q_\ell}).
\end{equation}
Second, for arbitrary nonnegative coefficients $(c_q)_{q\in\mathcal Q}$
and $\lambda\ge0$,
\begin{equation}
\label{eq:negative-association-laplace-domination}
    \mathbb E\exp\left(\lambda\sum_qc_q\eta_q\right)
    \le
    \prod_q\left(1-p+pe^{\lambda c_q}\right).
\end{equation}
Consequently, for $X=\sum_{q\in Q}\eta_q$ and every real number
$t\ge p|Q|$,
\begin{equation}
\label{eq:negative-association-binomial-tail}
    \mathbb P\{X\ge t\}
    \le
    \left(\frac{ep|Q|}{t}\right)^t.
\end{equation}
More generally, there is an absolute constant $c>0$ such that
\begin{equation}
\label{eq:negative-association-weighted-bernstein}
    \mathbb P\left\{
        \sum_qc_q\eta_q
        >p\sum_qc_q+t
    \right\}
    \le
    \exp\left[-c\min\left\{
        \frac{t^2}{p\sum_qc_q^2},
        \frac{t}{\max_qc_q}
    \right\}\right]
\end{equation}
for every $t>0$, with the usual conventions when a denominator vanishes.
\end{lemma}

\begin{proof}
The two-function definition of negative association extends to any finite
collection of disjoint blocks by induction, because a product of
nonnegative non-decreasing functions is again non-decreasing.  This proves
\eqref{eq:negative-association-block-product}.
Applying that estimate to
the singleton functions $s\mapsto e^{\lambda c_qs}$ gives
\eqref{eq:negative-association-laplace-domination}.  Exponential Markov's
inequality then gives the usual binomial Chernoff bound \eqref{eq:negative-association-binomial-tail}.
Finally, the standard
Bernstein optimization of the centered version of
\eqref{eq:negative-association-laplace-domination} gives
\eqref{eq:negative-association-weighted-bernstein}.
We refer to \cite{JoagDevProschan83,DubhashiRanjan98,BoucheronLugosiMassart13} for details.
\end{proof}

\subsection{Row-degree upper bound}

\begin{definition}[Row-degree upper bound]
\label{def:row-degree-upper-bound}
Let $\eta\in\{0,1\}^{k\times n}$ be a mask, and let $0<p\le1$ and $L\ge1$.
We say that $\eta$ satisfies the \emph{row-degree upper bound} with
parameters $p,L$ if, for every row $u\in[k]$,
\[
    d_u(\eta,[n])
    \le
    L\max\{\log(e k),pn\}.
\]
\end{definition}

\begin{lemma}[Row-degree upper bound under admissible support laws]
\label{lem:row-degree-upper-bound}
Let $D\ge1$.  There is a constant
$L_{\text{\tiny\ref*{lem:row-degree-upper-bound}}}(D)\ge1$ with the
following property.  Let $0<p\le1$, and let
$\Pi=(\pi_{ui})$ follow the admissible sparse-entry model of
Definition~\ref{def:admissible-sparse-entry-model} with parameter $p$, some
entry variable $\xi$, and support mask $\eta$.  Then, with probability at
least $1-(e k)^{-D}$, the mask $\eta$ satisfies the row-degree upper bound
with parameters
$p,L_{\text{\tiny\ref*{lem:row-degree-upper-bound}}}(D)$.
\end{lemma}
\begin{proof}
Put $M:=\max\{\log(e k),pn\}$. By Lemma~\ref{lem:negative-association-calculus}, for every $L>e$,
\[
    \mathbb P\{d_u(\eta,[n])>LM\}
    \le
    \left(\frac{epn}{LM}\right)^{LM}
    \le
    \left(\frac eL\right)^{LM}
    \le
    \exp(-c_L\,\log(e k)),
\]
where $c_L:=L\log(L/e)\to\infty$ as $L\to\infty$.  Choose
$L=L_{\text{\tiny\ref*{lem:row-degree-upper-bound}}}(D)$ so large that
$c_L\ge D+2$.  A union bound over the $k$ rows gives
\[
    \mathbb P\{\exists u\in[k]:d_u(\eta,[n])>LM\}
    \le
    k\exp(-c_L\,\log(e k))
    \le
    (e k)^{-D},
\]
which proves the assertion.
\end{proof}

\subsection{Heavy-row statistics}

\begin{lemma}[Fixed-set heavy-row tail]
\label{lem:fixed-set-heavy-row-tail}
Let $\Pi$ follow the admissible sparse-entry model of
Definition~\ref{def:admissible-sparse-entry-model} with parameter $p$ and
support mask $\eta$.  Let
$I\subset[n]$ have cardinality $s\ge1$, and let $K>0$.  If
\[
    \lceil K\rceil\ge2\log(e k)
    \qquad\text{and}\qquad
    \lceil K\rceil\ge e^3ps,
\]
then, for every $t\ge1$,
\[
    \mathbb P\{R_K(\eta,I)\ge t\}
    \le
    e^2\left(e^2\frac{ps}{\lceil K\rceil}\right)^t
    \le
    \left(C\frac{ps}{\lceil K\rceil}\right)^t
\]
for an absolute constant $C\ge1$.
\end{lemma}

\begin{proof}
Write $\pi_{ui}=\eta_{ui}\xi_{ui}$ as in
Definition~\ref{def:admissible-sparse-entry-model}.  The row degrees
$d_u(\eta,I)$, $u\in[k]$, depend on pairwise disjoint row blocks of the
negatively associated mask. Put
\[
    \lambda:=\log\left(\frac{\lceil K\rceil}{e^2ps}\right)>0.
\]
For every $u$ and $\lceil K\rceil\le b\le s$, the Chernoff estimate
\eqref{eq:negative-association-binomial-tail} gives
\[
    e^{\lambda b}\,\mathbb P\{d_u(\eta,I)=b\}
    \le
    \left(\frac{\lceil K\rceil}{eb}\right)^b
    \le e^{-b}.
\]
Consequently, for
$Y_u:=d_u(\eta,I)\mathbf 1_{\{d_u(\eta,I)\ge \lceil K\rceil\}}$,
\[
    \mathbb E e^{\lambda Y_u}
    \le
    1+\sum_{b=\lceil K\rceil}^{\infty}e^{-b}
    \le
    \exp(2e^{-\lceil K\rceil}).
\]
The functions $d_u(\eta,I)\mapsto e^{\lambda Y_u}$ are nonnegative and
non-decreasing.  Hence
\eqref{eq:negative-association-block-product} and Markov's inequality yield
\[
\begin{aligned}
    \mathbb P\{R_K(\eta,I)\ge t\}
    &\le
    \exp(-\lambda t+2ke^{-\lceil K\rceil}).
\end{aligned}
\]
The first assumption gives $ke^{-\lceil K\rceil}\le1$.  If
$\lceil K\rceil>s$, the event under consideration is empty.  Otherwise,
\[
    \mathbb P\{R_K(\eta,I)\ge t\}
    \le
    e^2\left(e^2\frac{ps}{\lceil K\rceil}\right)^t
    \le
    \left(C\frac{ps}{\lceil K\rceil}\right)^t,
\]
as claimed.
\end{proof}

\begin{lemma}[Total degree of canonical-heavy rows]
\label{lem:total-degree-canonical-heavy-rows}
For every $\alpha\ge0$ and $B\ge1$ there are constants
$L_{\text{\tiny\ref*{lem:total-degree-canonical-heavy-rows}}}\ge1$ and
$C_{\text{\tiny\ref*{lem:total-degree-canonical-heavy-rows}}}(\alpha,B)\ge1$
with the following property.  Assume
\[
    r^{10}\ge n\ge k\ge r\ge3,
    \qquad
    \frac{\log k}{k}\le p\le1.
\]
Let $\Pi$ follow the admissible sparse-entry model of
Definition~\ref{def:admissible-sparse-entry-model} with parameter $p$ and
support mask $\eta$, and let
$\mathcal N\subset S^{n-1}$ be a finite set with
$\abs{\mathcal N}\le \exp(\alpha r)$.  Fix
$L\ge L_{\text{\tiny\ref*{lem:total-degree-canonical-heavy-rows}}}$, and for
$x\in\mathcal N$, let 
$K_{j}^{(L)}(x)$
be the canonical threshold profile from
Definition~\ref{def:canonical-tall-flat-partition}.
Then, with probability at least $1-k^{-B}$,
\[
    \sup_{x\in\mathcal N}
    \sum_{j\in\mathbb Z}\sum_{u=1}^k
    d_u(\eta,I_j(x))
    \mathbf 1_{\{d_u(\eta,I_j(x))\ge K_{j}^{(L)}(x)\}}
    \le
    C_{\text{\tiny\ref*{lem:total-degree-canonical-heavy-rows}}}(\alpha,B)\,r .
\]
\end{lemma}

\begin{proof}
Fix $x\in\mathcal N$ and put
\[
    s_j:=\abs{I_j(x)},
    \qquad
    K_j:=K_j^{(L)}(x).
\]
Choose $L_0$ large enough so that, for all $L\ge L_0$,
\[
    \delta:=\frac{e^2}{L}<\frac12
    \qquad\text{and}\qquad
    L\log(1/\delta)>12 .
\]
If $b\ge\lceil K_j\rceil$ and
$s_j>0$, then
\[
    p s_j
    \le
    K_j/L
    \le b/L .
\]
Therefore \eqref{eq:negative-association-binomial-tail} gives
\[
    \mathbb P\{d_u(\eta,I_j(x))\ge b\}
    \le
    \left(\frac{eps_j}{b}\right)^b
    \le
    \left(\frac eL\right)^b .
\]
Consequently,
\[
\begin{aligned}
    \mathbb E\exp\{d_u(\eta,I_j(x))
        \mathbf 1_{\{d_u(\eta,I_j(x))\ge\lceil K_j\rceil\}}\}
    &\le
    1+\sum_{b=\lceil K_j\rceil}^{s_j}
        e^b\,\mathbb P\{d_u(\eta,I_j(x))=b\}                          \\
    &\le
    1+\sum_{b=\lceil K_j\rceil}^{s_j}\delta^b
    \le
    \exp(2\delta^{\lceil K_j\rceil}).
\end{aligned}
\]
Set
\[
    R_L^{\rm can}(x)
    :=
    \sum_{j\in\mathbb Z}\sum_{u=1}^k
    d_u(\eta,I_j(x))
    \mathbf 1_{\{d_u(\eta,I_j(x))\ge\lceil K_j\rceil\}} .
\]
The degrees $d_u(\eta,I_j(x))$ depend on pairwise disjoint blocks of the negatively
associated support mask as $(u,j)$ varies.
Since the functions
$d_u(\eta,I_j(x))\mapsto
\exp\{d_u(\eta,I_j(x))
\mathbf 1_{\{d_u(\eta,I_j(x))\ge\lceil K_j\rceil\}}\}$ are
nonnegative and non-decreasing, the block-product estimate
\eqref{eq:negative-association-block-product} gives
\[
    \mathbb E e^{R_L^{\rm can}(x)}
    \le
    \exp\left(2k
    \sum_{\substack{j\in\mathbb Z:\\s_j\ge\lceil K_j\rceil}}
    \delta^{\lceil K_j\rceil}\right).
\]
We now bound
$\sum_{\substack{j\in\mathbb Z:\,s_j\ge\lceil K_j\rceil}}
\delta^{\lceil K_j\rceil}$ uniformly in $x$.  Since
$\lceil K_j\rceil\ge K_j$ and
$K_j\ge Lpr_*\ge L\log k$,
\[
    \sum_{\substack{j\in\mathbb Z:\\s_j\ge\lceil K_j\rceil}}
    \delta^{\lceil K_j\rceil}
    \le
    n\delta^{Lpr_*}
    \le
    k^{10-L\log(1/\delta)} .
\]
Here we used that at most $n$ dyadic level sets are nonempty.  Since
$L\log(1/\delta)>12$, this implies
\[
    2k\sum_{\substack{j\in\mathbb Z:\\s_j\ge\lceil K_j\rceil}}
    \delta^{\lceil K_j\rceil}
    \le 1.
\]
By Markov's inequality, for every $T>0$,
\[
    \mathbb P\{R_L^{\rm can}(x)\ge T\}
    \le
    \exp\{-T+1\}.
\]
Taking the union bound over $\mathcal N$ gives
\[
    \mathbb P\left\{
        \sup_{x\in\mathcal N}R_L^{\rm can}(x)\ge T
    \right\}
    \le
    \exp\{-T+\alpha r+1\}.
\]
Choose
\[
    T:=(\alpha+C_0 B)r+1 ,
\]
where $C_0$ is a sufficiently large absolute constant.  Since
$\log k\le C_0r$, the right-hand side
$\exp\{-T+\alpha r+1\}$ is at most $k^{-B}$, which proves the lemma.
\end{proof}

\section{Entropy of Level Sets}
\label{sec:entropy-of-level-sets}

\subsection{A generic entropy bound}

\begin{lemma}
\label{lem:full-trace-counting}
Let $W\subset\R^n$ be a subspace of dimension $d\ge1$.  For $\beta>0$ and
$s\ge1$, let
\[
    \mathcal F_W(\beta,s)
    :=
    \left\{
        I\subset[n]:
        \abs I=s,\ \exists x\in W,\ \norm x_2\le1,\
        \abs{x_i}\ge\beta\ \text{for every }i\in I
    \right\}.
\]
Then
\[
    \abs{\mathcal F_W(\beta,s)}
    \le
    \frac{\beta^{-2s}}{s!}\binom{d+s-1}{s}
    \le
    \left(
        C_{\text{\tiny\ref*{lem:full-trace-counting}}}
        \frac{d+s}{\beta^2s^2}
    \right)^s
\]
for an absolute constant
$C_{\text{\tiny\ref*{lem:full-trace-counting}}}\ge1$.
\end{lemma}

\begin{proof}
Let $P_W:\R^n\to W$ be the orthogonal projection.  Choose an orthonormal basis of $W$ and identify $W$ with
$\R^d$.  For $i\in[n]$, put
$u_i:=P_We_i\in W$.  The projected coordinate vectors satisfy
\begin{equation}
\label{eq:full-trace-resolution-identity}
    \sum_{i=1}^n u_i u_i^{\mathsf T}=I_W .
\end{equation}

In what follows, given a multi-index $\alpha$, denote by $\abs\alpha$ its sum.
For a homogeneous polynomial of degree $s$,
\[
    p(y)=\sum_{\abs\alpha=s}c_\alpha y^\alpha,
\]
write
\[
    M_\alpha:=\frac{s!}{\alpha_1!\cdots\alpha_d!},
    \qquad
    \norm p_*^2
    :=
    \sum_{\abs\alpha=s}\frac{c_\alpha^2}{M_\alpha}.
\]
Here, $\norm p_*$ is the {\it Bombieri norm};
see \cite{BeauzamyBombieriEnfloMontgomery90}.
By Cauchy--Schwarz and the multinomial theorem,
\begin{equation}
\label{eq:full-trace-polynomial-cauchy}
    \abs{p(y)}
    \le
    \norm p_*\norm y_2^s.
\end{equation}

For an ordered tuple
$\mathbf i=(i_1,\ldots,i_s)\in[n]^s$, define
\[
    p_{\mathbf i}(y)
    :=
    \prod_{t=1}^s\langle y,u_{i_t}\rangle.
\]
We claim that
\begin{equation}
\label{eq:full-trace-ordered-polynomial-sum}
    \sum_{\mathbf i\in[n]^s}\norm{p_{\mathbf i}}_*^2
    =
    \binom{d+s-1}{s}.
\end{equation}
Indeed, fix a multi-index $\alpha$ with $\abs\alpha=s$, and let
$\Phi_\alpha$ be the set of maps $\phi:[s]\to[d]$ satisfying
$\abs{\phi^{-1}(a)}=\alpha_a$ for every $a\in[d]$.  If
$c_\alpha(\mathbf i)$ denotes the coefficient of $y^\alpha$ in
$p_{\mathbf i}$, then
\[
    c_\alpha(\mathbf i)
    =
    \sum_{\phi\in\Phi_\alpha}
    \prod_{t=1}^s (u_{i_t})_{\phi(t)}.
\]
Using \eqref{eq:full-trace-resolution-identity},
\begin{align*}
    \sum_{\mathbf i\in[n]^s}c_\alpha(\mathbf i)^2
    &=
    \sum_{\phi,\psi\in\Phi_\alpha}
    \prod_{t=1}^s
    \sum_{i=1}^n
    (u_i)_{\phi(t)}(u_i)_{\psi(t)} \\
    &=
    \sum_{\phi,\psi\in\Phi_\alpha}
    \prod_{t=1}^s\mathbf 1_{\{\phi(t)=\psi(t)\}}
    =
    \abs{\Phi_\alpha}
    =
    M_\alpha.
\end{align*}
Summing over the $\binom{d+s-1}{s}$ multi-indices $\alpha$ proves
\eqref{eq:full-trace-ordered-polynomial-sum}.

For an unordered set $I\subset[n]$ with $\abs I=s$, put
\[
    p_I(y):=\prod_{i\in I}\langle y,u_i\rangle.
\]
Every ordering of $I$ gives the same polynomial, and hence
\[
    s!\sum_{\substack{I\subset[n]\\\abs I=s}}\norm{p_I}_*^2
    \le
    \binom{d+s-1}{s}.
\]
If $I\in\mathcal F_W(\beta,s)$, choose a corresponding vector $x\in W$.
Since $\langle x,u_i\rangle=x_i$, we have
\[
    \abs{p_I(x)}\ge\beta^s.
\]
The polynomial estimate \eqref{eq:full-trace-polynomial-cauchy} therefore
gives $\norm{p_I}_*^2\ge\beta^{2s}$.  Consequently,
\[
    \abs{\mathcal F_W(\beta,s)}
    \le
    \beta^{-2s}\frac1{s!}\binom{d+s-1}{s}.
\]
Finally,
\[
    s!\ge(s/e)^s,
    \qquad
    \binom{d+s-1}{s}
    \le
    \left(e\frac{d+s}{s}\right)^s.
\]
Combining these estimates proves the lemma.
\end{proof}

\begin{corollary}[A basic entropy-probability balancing]
\label{lem:entropy-support-certificate}
Let $W\subset\mathbb R^n$ be a subspace of dimension $d\ge1$, let
$\beta>0$, let $q\ge1$ be an integer, and let $\kappa>0$.  Suppose that
$\Pi$ follows the admissible sparse-entry model with support mask $\eta$ and
density $p$, and assume
\[
    \lceil\kappa\rceil\ge2\log(e k),
    \qquad
    \lceil\kappa\rceil\ge e^3pq.
\]
Then
\begin{equation}
\label{eq:entropy-support-certificate}
    \sum_{J\in\mathcal F_W(\beta,q)}
    \mathbb P\{R_\kappa(\eta,J)\ge q\}
    \le
    \left(
        C\,\frac{p(d+q)}{\lceil\kappa\rceil\beta^2q}
    \right)^q.
\end{equation}
\end{corollary}

\begin{proof}
Lemma~\ref{lem:full-trace-counting} bounds the number of sets in the sum by
\[
    \left(C\,\frac{d+q}{\beta^2q^2}\right)^q.
\]
For each fixed set $J$ in that family,
Lemma~\ref{lem:fixed-set-heavy-row-tail}, with $K=\kappa$ and $t=q$,
bounds the corresponding probability by
$(Cpq/\lceil\kappa\rceil)^q$.  Multiplication gives
\eqref{eq:entropy-support-certificate}.
\end{proof}

\subsection{Application I: Subspace-sensitive edge count estimates}

As the first application of the entropy lemma, consider the problem of edge counting.
Let $k\ge3$,
let $1\le r\le k$ and $(\log k)/k\le p\le1$, let $V\subset\R^n$ be a fixed
$r$-dimensional subspace. Recall the notation
\[
    r_*:=\max\left\{r,\frac{\log k}{p}\right\}.
\]
Let $\eta=(\eta_{ui})_{u\in[k],\,i\in[n]}$ be a $\{0,1\}$-valued mask whose
coordinates are negatively associated and satisfy
$\mathbb E\eta_{ui}=p$.  
We can view $\eta$ as the adjacency structure of a random bipartite graph on $[k]\sqcup[n]$;
the quantities $e_\eta(I,J)$ defined in
Subsection~\ref{sec:technical-overview-kahn-szemeredi} are then interpreted as the number of edges of the graph
connecting the vertex subsets $I$ and $J$. 
We assume that $I$ comes as a subset of a level set of some unit vector $x\in V$.
The next lemma provides an upper bound on the edge count
independent of the ambient dimension $n$ and uniform over all admissible choices of $I$ and $J$.
This is a key component in repairing the Kahn--Szemeredi argument
in our setting.

\begin{lemma}[Subspace-sensitive edge count estimates]
\label{lem:ks-tall-entropy-rectangle-bound}
There is an absolute constant $C_H\ge4e$ with the following property.  For
every $j\in\mathbb Z$ satisfying
$1/\sqrt{r_*}\le2^{-j+1}\le1$ and every $s\ge1$, let
\[
    \mathcal F_j(s)
    :=
    \left\{
    I\subset[n]: |I|=s,\ \exists x\in V,\ \norm x_2\le1,\
    |x_i|>2^{-j}\ \text{for every }i\in I
    \right\}
\]
and define the level-set entropy
\[
    \Phi(j,s)
    :=
    s\log\left(
        C_H\frac{2^{2j-2}(r+s)}{s^2}
    \right).
\]
Then for every $B>0$ there is
$C_{\rm rect}=C_{\rm rect}(B)<\infty$ such that, with probability at
least $1-k^{-B}$, the following holds simultaneously for every
$j\in\mathbb Z$ satisfying $1/\sqrt{r_*}\le2^{-j+1}\le1$, every
$s,t\ge1$, every $I\in\mathcal F_j(s)$, and every
$J\subset[k]$ with $|J|=t$:
\begin{equation}
\label{eq:ks-tall-entropy-rectangle-bound}
    e_\eta(I,J)
    \le
    C_{\rm rect}
    \left[
        pst
        +
        \frac{Q(j,s,t)}
        {\log\left(e+Q(j,s,t)/(pst)\right)}
    \right].
\end{equation}
Here
\[
    Q(j,s,t):=\Phi(j,s)+t\log\frac{ek}{t}.
\]
\end{lemma}

\begin{proof}
Choose $C_H\ge4e$ large enough so that
Lemma~\ref{lem:full-trace-counting}, applied with $W=V$ and $\beta=2^{-j}$,
gives
\[
    |\mathcal F_j(s)|\le e^{\Phi(j,s)}
\]
for every $j\in\mathbb Z$ satisfying
$1/\sqrt{r_*}\le2^{-j+1}\le1$ and every $s\ge1$.

Fix $j,s,t,I,J$ with $I\in\mathcal F_j(s)$ and $|J|=t$.  The edge count
$e_\eta(I,J)$ has mean $\mu=pst$.  For every $A>0$, an elementary
calculus estimate gives a constant $C_A\ge e^2$ such that
\[
    C_A\left(1+\frac{x}{\log(e+x)}\right)
    \log\left[
        \frac{C_A}{e}\left(1+\frac{x}{\log(e+x)}\right)
    \right]
    \ge Ax
    \qquad (x\ge0).
\]
Set
\[
    x:=\frac{Q(j,s,t)}{\mu}
    \qquad\text{and}\qquad
    T:=C_A\mu\left(1+\frac{x}{\log(e+x)}\right)
    =C_A\left[
        \mu+\frac{Q(j,s,t)}{\log(e+Q(j,s,t)/\mu)}
    \right].
\]
Then
\[
\begin{aligned}
    T\log\frac{T}{e\mu}
    &=
    \mu C_A\left(1+\frac{x}{\log(e+x)}\right)
    \log\left[
        \frac{C_A}{e}\left(1+\frac{x}{\log(e+x)}\right)
    \right] \\
    &\ge A\mu x
    =A Q(j,s,t).
\end{aligned}
\]
The binomial upper-tail estimate
\eqref{eq:negative-association-binomial-tail} yields
\[
    \mathbb P\left\{
        e_\eta(I,J)>T
    \right\}
    \le
    \left(\frac{e\mu}{T}\right)^T
    =
    \exp\left(-T\log\frac{T}{e\mu}\right)
    \le
    \exp\{-A Q(j,s,t)\}.
\]
Thus, after choosing $A=A(B)$ sufficiently large, the claimed rectangle
estimate holds with $C_{\rm rect}=C_A$ outside an event of probability at
most $\exp\{-A Q(j,s,t)\}$.
For fixed $j,s,t$, the number of choices of $I$ is at most
$e^{\Phi(j,s)}$, and the number of choices of $J$ is at most
\[
    \binom kt\le \exp\left(t\log\frac{ek}{t}\right).
\]
Thus the failure probability for this fixed triple $(j,s,t)$ is at most
\[
    \exp\{-(A-1)Q(j,s,t)\}.
\]
Since $t\ge1$, we have
\[
    Q(j,s,t)\ge t\log\frac{ek}{t}\ge \log k .
\]
Also, if $\mathcal F_j(s)$ is non-empty and
$2^{-j+1}\ge1/\sqrt{r_*}$, then
$s\le2^{2j}\le4r_*$.  There are at most $C\log k$ integers $j$ satisfying
$1/\sqrt{r_*}\le2^{-j+1}\le1$, at most $k$ possible values of $t$, and at most
$4r_*\le4k$ possible values of $s$.  Therefore, after choosing
$A=A(B)$ sufficiently large and increasing constants for the finitely
many small values of $k$, the union bound over all $j,s,t$ gives failure
probability at most $k^{-B}$.

On the resulting event, \eqref{eq:ks-tall-entropy-rectangle-bound} holds for
all admissible rectangles.
This proves the lemma.
\end{proof}

\subsection{Application II: Multi-set entropy-probability balancing}

As a second application of the entropy lemma, we 
consider a setting of counter-balancing entropy and heavy-row
probabilities for a collection of index subsets.
The lemma below is used at later stages of the proof
to bound the Flat majorant.
In this subsection,
let $k\ge3$, let $1\le r\le k$ and $(\log k)/k\le p\le1$, let
$V\subset\mathbb R^n$ be a fixed $r$-dimensional subspace, and let $\Pi$
follow the admissible sparse-entry model of
Definition~\ref{def:admissible-sparse-entry-model} with parameter $p$ and
support mask $\eta$.

\begin{lemma}
\label{lem:weighted-prefix-count}
Let $L\ge1$, let $j_1<\cdots<j_a$, and let
$M_1,\ldots,M_a$ be positive integers.  Set
\[
    P_a:=M_1+\cdots+M_a.
\]
Assume $P_a\le r$, and let $b$ be a positive integer such that, for every
$c\le a$,
\[
    b\,2^{-2\,j_c}\ge Lp,
    \qquad
    b\ge2\log(e k),
    \qquad
    b\ge e^3pM_c.
\]
Let $\mathcal P_a$ be the family of tuples
of pairwise disjoint subsets
$(E_1,\ldots,E_a)$ for which $\abs{E_c}=M_c$, $c\leq a$, and there is a unit vector
$x=x(E_1,\ldots,E_a)\in V$ satisfying $E_c\subset I_{j_c}(x)$ for every $c\le a$.  Then
\[
    \sum_{(E_1,\ldots,E_a)\in\mathcal P_a}
    \mathbb P\left\{
        R_b(\eta,E_c)\ge M_c\text{ for every }c\le a
    \right\}
    \le
    \left(C\frac r{LP_a}\right)^{P_a}.
\]
\end{lemma}

\begin{proof}
Every coordinate of $x=x(E_1,\ldots,E_a)$ in $E_1\cup\cdots\cup E_a$ has absolute value at least
$2^{-j_a}$.  The exact estimate obtained in the proof of
Lemma~\ref{lem:full-trace-counting}, followed by the number
$P_a!/\prod_cM_c!$ of partitions of a fixed union into labeled sets of the
prescribed sizes, gives a rough upper bound 
\begin{equation}
\label{eq:prefix-tuple-entropy}
    \abs{\mathcal P_a}
    \le
    (2^{-j_a})^{-2P_a}
    \frac1{\prod_{c\le a}M_c!}
    \binom{r+P_a-1}{P_a}.
\end{equation}
For a fixed tuple, the support-degree events are increasing functions of
pairwise disjoint column blocks.  Thus
\eqref{eq:negative-association-block-product} and
Lemma~\ref{lem:fixed-set-heavy-row-tail}, applied to $E_c$ with $K=b$ and
$t=M_c$, give
\[
    \mathbb P\left\{
        R_b(\eta,E_c)\ge M_c\text{ for every }c\le a
    \right\}
    \le
    \prod_{c\le a}
    \left(C\frac{pM_c}{b}\right)^{M_c}.
\]
Multiplying this with \eqref{eq:prefix-tuple-entropy} and using
$M_c!\ge(M_c/e)^{M_c}$, we obtain
\[
\begin{aligned}
    \sum_{(E_1,\ldots,E_a)\in\mathcal P_a}
    \mathbb P\left\{
        R_b(\eta,E_c)\ge M_c\text{ for every }c\le a
    \right\}
    \le
    \binom{r+P_a-1}{P_a}
    \left(C\frac{p}{b\,2^{-2j_a}}\right)^{P_a}.
\end{aligned}
\]
Since $b\,2^{-2j_a}\ge Lp$ and $P_a\le r$, the lemma follows.
\end{proof}

\section{A Kahn--Szemeredi--type argument for Tall contributions}
\label{sec:ks-tall-argument}

\subsection{Main Objectives and Matrix Decomposition}

Throughout this section assume
\[
    r^{10}\ge n\ge k\ge r\ge3,
    \qquad
    \frac{\log k}{k}\le p\le1.
\]
Let $\Pi=(\pi_{ui})_{u\in[k],\,i\in[n]}$ follow the admissible sparse-entry model
of Definition~\ref{def:admissible-sparse-entry-model} with parameter $p$ and
a symmetric entry variable $\xi$ satisfying $|\xi|\le1$ almost surely.  Write
$\pi_{ui}=\eta_{ui}\xi_{ui}$.
Using the symmetry of $\xi$, we further write
\[
    \xi_{ui}=\sigma_{ui}|\xi_{ui}|,
\]
where the variables $\sigma_{ui}$ are independent Rademacher signs,
independent of the family $(\eta_{ui},|\xi_{ui}|)$.  Let $V\subset\R^n$ be a
fixed $r$-dimensional subspace.
Recall from Subsection~\ref{sec:technical-overview-tall-flat} the canonical
vector-wise Tall--Flat decomposition and its level-dependent threshold profile
$K_j^{(L)}(x)=L\max\{pr_*,p|I_j(x)|\}$.

\bigskip

The main estimate of this section is the following netted tall contribution
bound.  The restriction to a net is the form needed in the Kahn--Szemeredi
reduction.

\begin{proposition}[Canonical tall contribution on a finite net]
\label{prop:ks-tall-contribution-on-net}
For every $D>0$ and every $L\ge1$ there is
$C=C(D,L)<\infty$ such that the following holds.
Let ${\mathcal N}_V\subset V\cap S^{n-1}$ be a deterministic finite set satisfying
\[
    |{\mathcal N}_V|\le 5^r .
\]
Then
\[
    \mathbb P\left\{
        \sup_{z\in{\mathcal N}_V}
        \norm{\mathcal T_L^{\rm can}(\Pi,z)z}_2
        >
        C\sqrt{kp}
    \right\}
    \le k^{-D}.
\]
Here $\mathcal T_L^{\rm can}(\Pi,z)$ denotes the canonical Tall matrix from
Definition~\ref{def:canonical-tall-flat-partition}.
\end{proposition}

We recall the proof organization related to the Tall matrix (see
Figure~\ref{fig:high-level-partition}): we decompose
the matrix columns into {\it light} and {\it heavy}, with the latter in turn treated
using the light-heavy decomposition of couples (individual summands in the corresponding bilinear expression).
Our goal in this subsection is to formally define the notions and to show how the estimate in
Proposition~\ref{prop:ks-tall-contribution-on-net} can be reduced to analyzing separately the three corresponding terms.

\begin{definition}[Light and heavy columns relative to a vector]
\label{def:light-heavy-columns-relative-to-vector}
Fix $x\in\R^n$.  
Write
\[
    \mathcal J(x):=\{i\in[n]: |x_i|>1/\sqrt{r_*}\},
\]
Columns indexed by $\mathcal J(x)^c$ are called \emph{light with respect to
$x$}, and those indexed by $\mathcal J(x)$ are called \emph{heavy with
respect to $x$}.
\end{definition}

\bigskip

We now introduce the light and heavy bilinear forms for the canonical
heavy-column tall contribution.  Fix $L\ge1$, $x\in V\cap S^{n-1}$, and
$y\in S^{k-1}$.  For $i\in I_j(x)$ write
\begin{equation}
\label{eq:canonical-heavy-column-selector}
    \theta_{ui}^{(L)}(x)
    :=
    \mathbf 1_{\{i\in\mathcal J(x)\}}
    \mathbf 1_{\{d_u(\eta,I_j(x))<K_j^{(L)}(x)\}} .
\end{equation}
Then
\[
    \left\langle
    \mathcal T_L^{\rm can}(\Pi,x)x_{\mathcal J(x)},y
    \right\rangle
    =
    \sum_{u=1}^k\sum_{i=1}^n
    \eta_{ui}\sigma_{ui}|\xi_{ui}|\theta_{ui}^{(L)}(x)x_i y_u .
\]
Put
\[
    \rho:=\frac{\sqrt{kp}}{k}.
\]
\begin{definition}[Light couples]
\label{def:light-couples}
The corresponding light-couple term is
\[
    L_{\rm tall}(x,y)
    :=
    \sum_{u=1}^k\sum_{i=1}^n
    \eta_{ui}\sigma_{ui}|\xi_{ui}|
    \theta_{ui}^{(L)}(x)x_i y_u
    \mathbf 1_{\{|x_i y_u|\le \rho\}}.
\]
\end{definition}

\begin{definition}[Heavy couples]
\label{def:heavy-couples}
The heavy-couple term is the complementary part
\[
    H_{\rm tall}(x,y)
    :=
    \sum_{u=1}^k\sum_{i=1}^n
    \eta_{ui}\sigma_{ui}|\xi_{ui}|\theta_{ui}^{(L)}(x)x_i y_u\,
    \mathbf 1_{\{|x_i y_u|> \rho\}} .
\]
\end{definition}
\begin{definition}[Unsigned heavy-couple envelope]
\label{def:unsigned-heavy-couple-envelope}
It is useful to keep beside it the unsigned heavy-couple envelope
\[
    \mathcal H_{\rm tall}(x,y)
    :=
    \sum_{u=1}^k\sum_{i=1}^n
    \eta_{ui}\theta_{ui}^{(L)}(x)|x_i y_u|\,
    \mathbf 1_{\{|x_i y_u|> \rho\}} ,
\]
so that $|H_{\rm tall}(x,y)|\le \mathcal H_{\rm tall}(x,y)$.
\end{definition}

Next, we formulate three main ingredients for the proof of the central result of the section:

\begin{proposition}[Light columns]
\label{prop:ks-tall-light-columns-on-net}
For every $B>0$ and every $L\ge1$ there is
$C_{\rm lc}=C_{\rm lc}(B,L)<\infty$ such that the following holds.
Let ${\mathcal N}_V\subset V\cap S^{n-1}$ be a deterministic finite set satisfying
\[
    |{\mathcal N}_V|\le 5^r .
\]
Then
\[
    \mathbb P\left\{
        \sup_{z\in{\mathcal N}_V}
        \norm{\mathcal T_L^{\rm can}(\Pi,z)z_{\mathcal J(z)^c}}_2
        >
        C_{\rm lc}\sqrt{kp}
    \right\}
    \le e^{-Bk}.
\]
\end{proposition}

\begin{proposition}[Light couples]
\label{prop:ks-tall-light-on-net}
Let $L\ge1$.  Let ${\mathcal N}_V\subset V\cap S^{n-1}$ and
$\mathcal N'\subset S^{k-1}$ be deterministic finite sets satisfying
\[
    |{\mathcal N}_V|\le 5^r,
    \qquad
    |\mathcal N'|\le 5^k .
\]
For every $B>0$ there is $C_{\rm light}(B)<\infty$ such that
\[
    \mathbb P\left\{
    \sup_{\substack{z\in{\mathcal N}_V\\ y\in\mathcal N'}}
    |L_{\rm tall}(z,y)|
    >
    C_{\rm light}(B)\sqrt{kp}
    \right\}
    \le e^{-Bk}.
\]
\end{proposition}

\begin{proposition}[Heavy couples]
\label{prop:ks-tall-heavy-on-net}
For every $B>0$ and $L\ge1$ there is $C_{\rm heavy}=C_{\rm heavy}(B,L)<\infty$
such that
\[
    \mathbb P\left\{
        \sup_{\substack{z\in V\cap S^{n-1}\\ y\in S^{k-1}}}
        \mathcal H_{\rm tall}(z,y)
        > C_{\rm heavy}\sqrt{kp}
    \right\}
    \le k^{-B}.
\]
Consequently the same bound holds with the supremum restricted to any finite
nets ${\mathcal N}_V\subset V\cap S^{n-1}$ and
$\mathcal N'\subset S^{k-1}$.
\end{proposition}

The above statements will be treated in respective subsections.
For now, we can complete the proof of Proposition~\ref{prop:ks-tall-contribution-on-net}:

\begin{proof}[Proof of Proposition~\ref{prop:ks-tall-contribution-on-net}]
Let $\mathcal N'\subset S^{k-1}$ be a deterministic $1/2$-net with
$|\mathcal N'|\le5^k$.  Apply
Proposition~\ref{prop:ks-tall-light-columns-on-net} with this $L$ and
$B=D+3$.  Apply Proposition~\ref{prop:ks-tall-light-on-net} with this $L$ and
$B=D+3$, and apply Proposition~\ref{prop:ks-tall-heavy-on-net} with this $L$
and $B=D+3$.  With probability at least $1-k^{-D}$, after increasing
constants and treating finitely many small values of $k$ by enlarging $C$,
all three corresponding events hold.

For every $z\in{\mathcal N}_V$, the norming-net estimate and
Definitions~\ref{def:light-couples}--\ref{def:unsigned-heavy-couple-envelope}
give
\[
\begin{aligned}
    \norm{\mathcal T_L^{\rm can}(\Pi,z)z_{\mathcal J(z)}}_2
    &\le
    2\sup_{y\in\mathcal N'}
    \left|\left\langle
        \mathcal T_L^{\rm can}(\Pi,z)z_{\mathcal J(z)},y
    \right\rangle\right| \\
    &\le
    2\sup_{y\in\mathcal N'}
    \left(
        |L_{\rm tall}(z,y)|+\mathcal H_{\rm tall}(z,y)
    \right).
\end{aligned}
\]
Thus Proposition~\ref{prop:ks-tall-heavy-on-net} and the light-couple estimate
give
\[
    \sup_{z\in{\mathcal N}_V}
    \norm{\mathcal T_L^{\rm can}(\Pi,z)z_{\mathcal J(z)}}_2
    \le C(D,L)\sqrt{kp}.
\]
Together with Proposition~\ref{prop:ks-tall-light-columns-on-net}, this gives
\[
    \sup_{z\in{\mathcal N}_V}
    \norm{\mathcal T_L^{\rm can}(\Pi,z)z}_2
    \le
    \sup_{z\in{\mathcal N}_V}
    \norm{\mathcal T_L^{\rm can}(\Pi,z)z_{\mathcal J(z)^c}}_2
    +
    \sup_{z\in{\mathcal N}_V}
    \norm{\mathcal T_L^{\rm can}(\Pi,z)z_{\mathcal J(z)}}_2
    \le C(D,L)\sqrt{kp}.
\]
This proves the proposition.
\end{proof}

\subsection{Light Columns}

\begin{proof}[Proof of Proposition~\ref{prop:ks-tall-light-columns-on-net}]
Fix $z\in{\mathcal N}_V$.  Then
\[
    \mathcal T_L^{\rm can}(\Pi,z)z_{\mathcal J(z)^c}
    =
    \sum_{u=1}^k
    \left(
        \sum_{j\in\mathbb Z}\sum_{i\in I_j(z)\cap\mathcal J(z)^c}
        \eta_{ui}\sigma_{ui}|\xi_{ui}|z_i
        \mathbf 1_{\{d_u(\eta,I_j(z))<K_j^{(L)}(z)\}}
    \right)e_u .
\]
Define the $\sigma$--field
\[
    \mathcal G:=\sigma\bigl((\eta_{ui}),(|\xi_{ui}|)\bigr).
\]
On conditioning on $\mathcal G$, for $i\in I_j(z)$ the coefficients
\[
    c_{ui}(z):=
    \eta_{ui}|\xi_{ui}|z_i
    \mathbf 1_{\{i\in\mathcal J(z)^c\}}
    \mathbf 1_{\{d_u(\eta,I_j(z))<K_j^{(L)}(z)\}}
\]
are deterministic.  The remaining randomness
is in the independent signs $\sigma_{ui}$.  Put
\[
    X_u(z):=\sum_{i=1}^n\sigma_{ui}c_{ui}(z),
    \qquad
    s_u(z):=\sum_{i=1}^n c_{ui}(z)^2,
    \qquad
    S(z):=\sum_{u=1}^k s_u(z).
\]
Since $|\xi_{ui}|\le1$, for every row $u$,
\[
    s_u(z)
    \le
    4
    \sum_{\substack{j\in\mathbb Z\\ I_j(z)\cap\mathcal J(z)^c\ne\varnothing}}
    2^{-2j}d_u(\eta,I_j(z))
    \mathbf 1_{\{d_u(\eta,I_j(z))<K_j^{(L)}(z)\}} .
\]
By the canonical cutoff, this is at most
\[
    4L
    \sum_{\substack{j\in\mathbb Z\\ I_j(z)\cap\mathcal J(z)^c\ne\varnothing}}
    2^{-2j}\max\{pr_*,p\abs{I_j(z)}\}.
\]
For every index $j$ occurring in the preceding sum,
$I_j(z)\cap\mathcal J(z)^c\ne\varnothing$, and hence
$2^{-j}<r_*^{-1/2}$.  Consequently, after factoring out $p$, the contribution
of the $pr_*$ term is controlled by the geometric-series estimate
\[
    r_*\sum_{\substack{j\in\mathbb Z\\2^{-j}<r_*^{-1/2}}}2^{-2j}\le C.
\]
The contribution of the $p\abs{I_j(z)}$ term is controlled by the normalization
of $z$: since $i\in I_j(z)$ implies $|z_i|>2^{-j}$,
\[
    \sum_{j\in\mathbb Z}\abs{I_j(z)}2^{-2j}\le \norm z_2^2=1 .
\]
Consequently,
\[
    s_u(z)\le C Lp
    \qquad\text{and}\qquad
    S(z)\le C Lkp
\]
deterministically, for every $u$ and $z$.

For each fixed $z$, the random variables $X_u(z)$
are independent, centered, and subgaussian with
$$\|X_u(z)\|_{\psi_2}^2\le C_Lp.$$  Hence $X_u(z)^2$ are independent
subexponential random variables with $\|X_u(z)^2\|_{\psi_1}\le C_Lp$ and
$\sum_{u=1}^k\mathbb E_\sigma X_u(z)^2\le C_Lkp$.  Bernstein's inequality gives,
after choosing $C_{\rm lc}=C_{\rm lc}(B,L)$ sufficiently large,
\[
    \mathbb P_\sigma\left\{
        \sum_{u=1}^k X_u(z)^2
        >
        C_{\rm lc}^2kp
        \,\middle|\,\mathcal G
    \right\}
    \le
    \exp\{-(B+\log5+1)k\}.
\]
Since $r\le k$ and $|{\mathcal N}_V|\le5^r$, the union bound gives the same
estimate uniformly over $z\in{\mathcal N}_V$, with failure probability at most
$e^{-Bk}$.
Since
\[
    \sum_{u=1}^kX_u(z)^2
    =
    \norm{\mathcal T_L^{\rm can}(\Pi,z)z_{\mathcal J(z)^c}}_2^2 ,
\]
this proves the claim.
\end{proof}

\subsection{Light Couples}

Conceptually, the treatment of light couples follows the standard
Kahn--Szemer\'edi approach: after conditioning on the support and entry
magnitudes, one applies concentration to the signed bilinear form and then
takes a union bound over the relevant nets.  See
\cite{FriedmanKahnSzemeredi89,FeigeOfek05,KeshavanMontanariOh10}.  The only
additional feature here is the canonical tall selector inherited from the
Tall--Flat decomposition.

\begin{proof}[Proof of Proposition~\ref{prop:ks-tall-light-on-net}]
Fix $z\in{\mathcal N}_V$ and $y\in\mathcal N'$.  Recall the definition of
$\theta_{ui}^{(L)}(z)$ from \eqref{eq:canonical-heavy-column-selector}.  By
Definition~\ref{def:light-couples},
\[
    L_{\rm tall}(z,y)
    =
    \sum_{u=1}^k\sum_{i=1}^n
    \eta_{ui}\sigma_{ui}|\xi_{ui}|\theta_{ui}^{(L)}(z)z_i y_u
    \mathbf 1_{\{|z_i y_u|\le\rho\}},
    \qquad
    \rho=\frac{\sqrt{kp}}{k}.
\]
Let
\[
    \mathcal G:=\sigma\bigl((\eta_{ui}),(|\xi_{ui}|)\bigr).
\]
The selector $\theta_{ui}^{(L)}(z)$ is $\mathcal G$-measurable; in fact, it
depends only on the support mask $\eta$.

For fixed $z,y$, set
\[
    V_0(z,y)
    :=
    \sum_{u=1}^k\sum_{i=1}^n
    \eta_{ui}z_i^2y_u^2
    \mathbf 1_{\{i\in\mathcal J(z)\}}
    \mathbf 1_{\{|z_i y_u|\le\rho\}} .
\]
The summands are non-negative, bounded by $\rho^2$, and indexed by a
negatively associated family.  Moreover
\[
    \mathbb E V_0(z,y)
    =
    p\sum_{u,i}z_i^2y_u^2
    \mathbf 1_{\{i\in\mathcal J(z)\}}
    \mathbf 1_{\{|z_i y_u|\le\rho\}}
    \le
    p\norm z_2^2\norm y_2^2
    =
    p,
\]
Now, the weighted Bernstein estimate
\eqref{eq:negative-association-weighted-bernstein} is applied with the index
$q=(u,i)$, and weights
\[
    c_{ui}:=z_i^2y_u^2
    \mathbf 1_{\{i\in\mathcal J(z)\}}
    \mathbf 1_{\{|z_i y_u|\le\rho\}}.
\]
The preceding bound gives
$\sum_{u=1}^k\sum_{i=1}^n c_{ui}\le1$; further, obviously
$\max_{u,i}c_{ui}\le\rho^2$. 
Therefore, $\sum_{u=1}^k\sum_{i=1}^n c_{ui}^2\le\rho^2$.
Applying \eqref{eq:negative-association-weighted-bernstein} with a parameter $A_0\geq 1$,
we get
\[
    \mathbb P\{V_0(z,y)>(A_0+1)p\}
    \le
    \exp(-cA_0k).
\]
Since
$|{\mathcal N}_V||\mathcal N'|\le 5^{r+k}\le25^k$, choosing
$A_0=A_0(B)$ sufficiently large gives a $\mathcal G$--measurable event $\mathcal E_{\rm var}$ with
\[
    \mathbb P(\mathcal E_{\rm var}^c)\le \frac12 e^{-Bk}
\]
on which, simultaneously for all $z\in{\mathcal N}_V$ and $y\in\mathcal N'$,
\[
    V_0(z,y)\le C_B p.
\]

Now condition on $\mathcal G$ and work on $\mathcal E_{\rm var}$.
For fixed $z,y$ define the deterministic coefficients
\[
    b_{ui}(z,y)
    :=
    \eta_{ui}|\xi_{ui}|\theta_{ui}^{(L)}(z)z_i y_u
    \mathbf 1_{\{|z_i y_u|\le\rho\}} .
\]
Then
\[
    \sum_{u=1}^k\sum_{i=1}^n b_{ui}(z,y)^2
    \le
    V_0(z,y)
    \le C_Bp .
\]
The only remaining randomness is in the independent Rademacher signs
$\sigma_{ui}$.  The subgaussian tail estimate for Rademacher sums therefore
gives, for every $t>0$,
\[
    \mathbb P_\sigma\left\{
        \left|\sum_{u=1}^k\sum_{i=1}^n
        \sigma_{ui}b_{ui}(z,y)\right|>t
        \,\middle|\,\mathcal G,\;\mathcal E_{\rm var}
    \right\}
    \le
    2\exp\left(-\frac{ct^2}{C_Bp}\right).
\]
Take $t=M\sqrt{kp}$.  The exponent is
\[
    \frac{cM^2}{C_B}k.
\]
Choosing $M=M(B)$ sufficiently large and unioning over
${\mathcal N}_V\times\mathcal N'$ gives conditional failure probability at most
$\frac12 e^{-Bk}$ on $\mathcal E_{\rm var}$, proving the proposition.
\end{proof}

\subsection{Heavy Couples}

As we discussed in the technical overview,
the heavy-couple term is reduced to uniform bounds for support-edge
counts between dyadic coordinate rectangles. 
To deal with the expression,
we combine a trivial global
column-degree estimate (stated below) with the entropy-sensitive estimate for coordinate
levels from Section~\ref{sec:entropy-of-level-sets}, and sum the resulting envelope through a standard argument
from the literature on the Kahn--Szemeredi method and its developments.

\begin{lemma}[Column-degree bound for rectangles]
\label{lem:ks-tall-edge-column-bound}
For every $B>0$ there is $C_{\rm col}=C_{\rm col}(B)<\infty$ such that, with
probability at least $1-k^{-B}$, the following holds simultaneously for all
$I\subset[n]$ and $J\subset[k]$:
\[
    e_\eta(I,J)
    \le
    C_{\rm col}kp |I|.
\]
Here $e_\eta(I,J)$ is the support-edge count defined in
Subsection~\ref{sec:technical-overview-kahn-szemeredi}.
\end{lemma}

\begin{proof}
Let
\[
    D_i:=\sum_{u=1}^k \eta_{ui}
\]
be the designated support degree of column $i$.  By
\eqref{eq:negative-association-binomial-tail}, for each
$K\ge e$,
\[
    \mathbb P\{D_i>Kkp\}
    \le
    \left(\frac eK\right)^{Kkp}.
\]
Since $p\ge\log k/k$, we have $kp\ge\log k$.  Choose
$K=K(B)\ge e$ so large that
\[
    K\log(K/e)\ge B+10 .
\]
Using $n\le k^{10}$, the union bound gives
\[
    \mathbb P\left\{\max_{1\le i\le n}D_i>Kkp\right\}
    \le
    n\left(\frac eK\right)^{Kkp}
    \le
    k^{10-K\log(K/e)}
    \le k^{-B}.
\]
On the complementary event, for every $I\subset[n]$ and $J\subset[k]$,
\[
    e_\eta(I,J)
    =
    \sum_{i\in I}\sum_{u\in J}\eta_{ui}
    \le
    \sum_{i\in I}D_i
    \le
    Kkp |I|.
\]
Thus the claim holds with $C_{\rm col}=K$.
\end{proof}

The proof of the following lemma largely reproduces the classical
envelope-summation arguments from
\cite{FriedmanKahnSzemeredi89,FeigeOfek05,KeshavanMontanariOh10}.
For this reason, we defer it to
Appendix~\ref{app:ks-classical-envelope-summation}.

\begin{lemma}[Classical envelope summation]
\label{lem:ks-classical-envelope-summation}
Let $k\ge d\ge3$, $0<p\le1$, and put $\rho:=\sqrt{kp}/k$.  Let
$S\ge1$.  Let $(s_j)_{j\ge1}$ and $(t_\ell)_{\ell\ge1}$ be
finitely supported nonnegative sequences, and put
\[
    m_j:=2^{-2j+2}s_j,\qquad
    n_\ell:=2^{-2\ell+2}t_\ell,
\]
and assume
\[
    \sum_{j\ge1} m_j\le S,\qquad \sum_{\ell\ge1} n_\ell\le S,
    \qquad 0\le s_j\le4d,\qquad 0\le t_\ell\le k .
\]
For pairs with $s_jt_\ell>0$, define
\[
    \mu_{j\ell}:=ps_jt_\ell,\qquad
    Q_{j\ell}:=
    s_j\log\frac{4ed}{s_j}
    +
    t_\ell\log\frac{ek}{t_\ell},
\]
and set
\[
    U_{j\ell}^{\rm cl}
    :=
    \mu_{j\ell}
    +
    \min\left\{
        s_jt_\ell,\,
        kp s_j,\,
        pd t_\ell,\,
        \frac{Q_{j\ell}}{\log(e+Q_{j\ell}/\mu_{j\ell})}
    \right\}.
\]
Then
\[
    \sum_{\substack{j,\ell\ge1,\ s_jt_\ell>0\\
        2^{-j-\ell+2}>\rho}}
    2^{-j-\ell+2}U_{j\ell}^{\rm cl}
    \le
    C(S)\sqrt{kp} .
\]
\end{lemma}

\begin{proof}[Proof of Proposition~\ref{prop:ks-tall-heavy-on-net}]
Apply Lemmas~\ref{lem:ks-tall-edge-column-bound} and
\ref{lem:ks-tall-entropy-rectangle-bound} with $B+2$ in place of $B$, and
work on the intersection of the two resulting events.  This intersection has
probability at least
\[
    1-2k^{-(B+2)}\ge1-k^{-B}.
\]
Recall the unsigned envelope $\mathcal H_{\rm tall}$ from
Definition~\ref{def:unsigned-heavy-couple-envelope}.
Fix $z\in V\cap S^{n-1}$ and $y\in S^{k-1}$.  For $j,\ell\ge1$, use the
canonical coordinate levels $I_j(z)$ and set
\[
    J_\ell(y):=\{u:2^{-\ell}<|y_u|\le2^{-\ell+1}\}.
\]
Write
\[
    s_j:=|I_j(z)|,\qquad t_\ell:=|J_\ell(y)|,
    \qquad
    m_j:=2^{-2j+2}s_j,\qquad n_\ell:=2^{-2\ell+2}t_\ell.
\]
Then
\[
    \sum_{j\ge1}m_j\le4,\qquad \sum_{\ell\ge1}n_\ell\le4.
\]
By Definition~\ref{def:canonical-tall-flat-partition},
$K_j^{(L)}(z)=L\max\{pr_*,ps_j\}$.  Define the corresponding canonical tall
rectangle count by
\[
    e_{\rm can}^{(L)}(I_j(z),J_\ell(y))
    :=
    \sum_{u\in J_\ell(y)}d_u(\eta,I_j(z))
    \mathbf 1_{\{d_u(\eta,I_j(z))<K_j^{(L)}(z)\}} .
\]
If a heavy edge joins $I_j(z)$ to $J_\ell(y)$, then necessarily
$2^{-j-\ell+2}>\rho$, since
$|z_i y_u|\le2^{-j-\ell+2}$ on $I_j(z)\times J_\ell(y)$.  Therefore
\begin{equation}
\label{eq:ks-tall-heavy-canonical-rectangle-sum}
    \mathcal H_{\rm tall}(z,y)
    \le
    \sum_{\substack{j,\ell\ge1\\2^{-j+1}\ge1/\sqrt{r_*}\\
        2^{-j-\ell+2}>\rho}}
    2^{-j-\ell+2}e_{\rm can}^{(L)}(I_j(z),J_\ell(y)),
\end{equation}
and our goal is to bound $e_{\rm can}^{(L)}(I_j(z),J_\ell(y))$ for all admissible parameters.

If $2^{-j+1}\ge1/\sqrt{r_*}$ and $s_j\ge1$, then
\[
    s_j\le2^{2j}\le4r_*.
\]
Moreover,
\begin{equation}
\label{eq:ks-tall-level-entropy-classical-bound}
    \Phi(j,s_j)
    =
    s_j\log\left(C_H\frac{2^{2j-2}(r+s_j)}{s_j^2}\right)
    \le
    C'_H\,s_j\log\frac{4er_*}{s_j}.
\end{equation}
Indeed, $r\le r_*$, $s_j\le4r_*$, and $2^{2j-2}\le r_*$ imply
\[
    \frac{2^{2j-2}(r+s_j)}{s_j^2}
    \le
    5\left(\frac {r_*}{s_j}\right)^2.
\]
Further, whenever $2^{-j+1}\ge1/\sqrt{r_*}$ and $s_j\ge1$, the set $I_j(z)$ belongs to
$\mathcal F_j(s_j)$ from
Lemma~\ref{lem:ks-tall-entropy-rectangle-bound}.  Thus, on the two
conditioned events from the beginning of the proof, for every pair with $2^{-j+1}\ge1/\sqrt{r_*}$ and
$s_j,t_\ell\ge1$,
\[
    e_{\rm can}^{(L)}(I_j(z),J_\ell(y))
    \le
    C\left[
        \mu_{j\ell}
        +
        \min\left\{
            s_jt_\ell,\,
            kp s_j,\,
            4Lpr_* t_\ell,\,
            \frac{Q_{j\ell}^{\rm tall}}
            {\log(e+Q_{j\ell}^{\rm tall}/\mu_{j\ell})}
        \right\}
    \right],
\]
where
\[
    \mu_{j\ell}:=ps_jt_\ell,
    \qquad
    Q_{j\ell}^{\rm tall}:=
    \Phi(j,s_j)+
    t_\ell\log\frac{ek}{t_\ell}.
\]
Indeed, Lemma~\ref{lem:ks-tall-entropy-rectangle-bound} gives the
$\mu_{j\ell}+Q_{j\ell}^{\rm tall}/
\log(e+Q_{j\ell}^{\rm tall}/\mu_{j\ell})$ term, while
the deterministic inequality $e_\eta(I_j(z),J_\ell(y))\le s_jt_\ell$ gives
the first bound, Lemma~\ref{lem:ks-tall-edge-column-bound} gives the
$kp s_j$ bound, and the canonical cutoff gives
$e_{\rm can}^{(L)}(I_j(z),J_\ell(y))
\le K_j^{(L)}(z)t_\ell\le4Lpr_* t_\ell$.
Thus, the result is proved once we show that the above upper estimate together with \eqref{eq:ks-tall-heavy-canonical-rectangle-sum}
produces a quantity of order $O(\sqrt{kp})$.

By \eqref{eq:ks-tall-level-entropy-classical-bound},
$$
Q_{j\ell}^{\rm tall}\le C'_H\,\Big(s_j\log\frac{4er_*}{s_j}
    +
    t_\ell\log\frac{ek}{t_\ell}\Big)=:C_H'\,Q_{j\ell}^{\rm cl},
$$ 
and therefore
\[
    e_{\rm can}^{(L)}(I_j(z),J_\ell(y))
    \le
    C_0\left[
        \mu_{j\ell}
        +
        \min\left\{
            s_jt_\ell,\,
            kp s_j,\,
            pr_* t_\ell,\,
            \frac{Q_{j\ell}^{\rm cl}}
            {\log(e+Q_{j\ell}^{\rm cl}/\mu_{j\ell})}
        \right\}
    \right]
\]
with $C_0=C_0(B,C_H',L)$.  Pulling out this factor $C_0$, apply
Lemma~\ref{lem:ks-classical-envelope-summation} with $d=r_*$ and $S=4$ to
the sequence whose $j$-th term is $s_j$ when
$2^{-j+1}\ge1/\sqrt{r_*}$ and zero otherwise, and to the row-level sequence
$(t_\ell)$.  It gives
\[
    \mathcal H_{\rm tall}(z,y)
    \le
    C(B,L)\sqrt{kp} .
\]
The estimate is uniform in $z$ and $y$, so taking the supremum proves the
claim.
\end{proof}

\section{Flat-majorant reduction and disjoint extraction}
\label{sec:flat-majorants}

Throughout the Flat analysis, $\eta$ denotes the designated support mask of
$\Pi$, and all row degrees and heavy-row degree sums are computed from
$\eta$.  Recall from
Definition~\ref{def:canonical-flat-profile-majorant} that the canonical Flat
majorant is
\[
    \FM_L^{\rm can}(\eta,x)
    =
    \left[
    \sum_{u=1}^k
    \left(
        \sum_{\substack{j\in\mathbb Z\\
        d_u(\eta,I_j(x))\ge K_j^{(L)}(x)}}
        2^{-j}d_u(\eta,I_j(x))
    \right)^2
    \right]^{1/2},
\]
where
\[
    K_j^{(L)}(x)=L\max\{pr_*,p\abs{I_j(x)}\}.
\]
It dominates the norm of the canonical Flat contribution up to an absolute
factor; thus the goal of the Flat analysis is to prove, under the assumptions
of the main theorem and with high probability,
\[
    \sup_{x\in\mathcal N_V}\FM_L^{\rm can}(\eta,x)
    \lesssim_{B,L}\sqrt{pk}
\]
for the relevant norming net $\mathcal N_V\subset V\cap S^{n-1}$.
In this section, we perform a (rather tedious) technical analysis of $\FM_L^{\rm can}(\eta,x)$,
which allows to estimate the canonical flat contribution in terms of the {\it leaf square function}
of certain partition of $[n]$; the related notions will be introduced in the later parts of the section.

\subsection{A peeling algorithm}

The generic Flat row profile and majorant were introduced in
Definition~\ref{def:flat-row-profile-majorant} for a weighted family of
pairwise disjoint coordinate sets.  In this subsection we temporarily work
with an arbitrary such family, rather than directly with the dyadic levels
$I_j(x)$, in order to isolate the deterministic combinatorial mechanism from
the geometry of a particular vector.  The peeling algorithm successively
selects a row of largest residual degree and removes its support, thereby
producing disjoint row-supported pieces.  Lemma~\ref{lem:flat-contribution-z}
shows that, under certain {\it heavy-row packing condition}, the resulting pieces retain a
fixed proportion of the abstract Flat majorant.  Proposition~\ref{thm:flat-extraction}
later applies this statement with $P=I_j(x)$, $\omega_P=2^{-j}$, and the
reduced thresholds.

Recall that 
for a finite family $\mathcal I$ of pairwise disjoint subsets of $[n]$,
positive weights $\omega=(\omega_P)_{P\in\mathcal I}$, and positive
thresholds $K=(K_P)_{P\in\mathcal I}$,
\[
    \FM((P,\omega_P,K_P)_{P\in\mathcal I})
    =
    \left[
    \sum_{u=1}^k
    \left(
        \sum_{\substack{P\in\mathcal I\\ d_u(\eta,P)\ge K_P}}
        \omega_P\,d_u(\eta,P)
    \right)^2
    \right]^{1/2}.
\]
\begin{definition}[Heavy-row packing condition]
Fix a realization $\eta\in\{0,1\}^{k\times n}$ of the mask and regard it as
deterministic; no randomness of $\eta$ is involved in this definition.  Let
$\mathcal I$ be a finite family of pairwise disjoint subsets of $[n]$.  Fix positive thresholds
$K=(K_P)_{P\in\mathcal I}$ and a constant $L\ge1$.
We say that $\eta$ satisfies the heavy-row packing condition on $\mathcal I$ if, for
every $P\in\mathcal I$ and every subset $J\subset P$,
\[
    R_{K_P}(\eta,J)\le L\abs J .
\]
\end{definition}

\begin{remark}
The condition above should be thought of as forbidding
``overcrowding'' of the heavy row supports
within any fixed index subset of $[n]$.
\end{remark}

Given a parameter $\kappa>0$, and fix a deterministic tie-breaking rule for rows.
For each set $P\subset[n]$, the following algorithm
constructs disjoint subsets
$$J_t(P)\subset P,\quad 1\le t\le k,$$
row indices $u_t(P)$, and numbers $z_t(P)$.

\begin{algorithm}[H]
\caption{A peeling algorithm}
\label{alg:peeling}
\begin{algorithmic}[1]
\Require $P\subset[n]$, $\kappa>0$, and the fixed row tie-breaking rule.
\Ensure Pairwise disjoint sets $J_t(P)\subset P$, numbers $z_t(P)$, and row
indices $u_t(P)$, $1\leq t\leq k$.
\For{$t=1,\ldots,k$}
    \State $R_t(P)\gets
    P\setminus\bigcup_{\ell=1}^{t-1}J_\ell(P)$
    \State Let $u_t(P)$ be a row maximizing $d_u(\eta,R_t(P))$, $1\leq u\leq k$, with ties
    broken by the fixed rule.
    \State $z_t(P)\gets d_{u_t(P)}(\eta,R_t(P))$
    \If{$z_t(P)<\kappa$}
        \State $J_t(P)\gets\varnothing$ and $z_t(P)\gets0$
    \Else
        \State $J_t(P)\gets R_t(P)\cap\{i\in[n]:\eta_{u_t(P)i}=1\}$
    \EndIf
\EndFor
\end{algorithmic}
\end{algorithm}

By construction, the sets $J_t(P)$ are pairwise disjoint.  Further, whenever
$z_t(P)>0$, we have
\begin{equation}\label{eq:z-size}
    \kappa\le z_t(P)=\max_{u\in[k]}d_u(\eta,R_t(P)),
    \qquad
    \abs{J_t(P)}=z_t(P).
\end{equation}
The sequence $(z_t(P))_{t=1}^k$ is non-increasing in $t$.
The positive rows $u_t(P)$ are distinct: once $u_t(P)$ is selected, all of
its support in the current residual set is removed, so its degree is zero in
every later residual set.

\begin{lemma}[Weighted-family peeling majorization]
\label{lem:flat-contribution-z}
Fix a realization $\eta\in\{0,1\}^{k\times n}$ of the mask and regard it as
deterministic.  Let $\mathcal I$ be a finite family of pairwise disjoint subsets of $[n]$,
and let $\omega=(\omega_P)_{P\in\mathcal I}$ be positive weights.  Let
$K=(K_P)_{P\in\mathcal I}$ be a positive threshold profile, and let
$L\ge1$. Suppose that $\eta$ satisfies the heavy-row packing condition on
$\mathcal I$ with parameters $K$ and $L$.  For every $P\in\mathcal I$, set
\[
    \kappa_P:=\frac{K_P}{2L}
\]
and run Algorithm~\ref{alg:peeling} with parameter $\kappa_P$.
Then, deterministically,
\[
    \left[
    \sum_{t=1}^k
    \left(
        \sum_{P\in\mathcal I}\omega_Pz_t(P)
    \right)^2
    \right]^{1/2}
    \ge
    c_{\text{\tiny\ref*{lem:flat-contribution-z}}}
    \,\FM((P,\omega_P,K_P)_{P\in\mathcal I}),
\]
where
$c_{\text{\tiny\ref*{lem:flat-contribution-z}}}=(3L+1)^{-1}$.
\end{lemma}

\begin{proof}
For every $P\in\mathcal I$, define
\[
    h_u(P):=d_u(\eta,P)\mathbf 1_{\{d_u(\eta,P)\ge K_P\}},
    \qquad 1\le u\le k.
\]
\begin{claim}
For every $P\in\mathcal I$, let
$(h_t^*(P))_{t=1}^k$ be the non-increasing rearrangement of
$(h_u(P))_{u=1}^k$.  Then, for every $1\le s\le k$,
\begin{equation}\label{eq:one-level-majorization}
    \sum_{t=1}^s h_t^*(P)\le C_0\sum_{t=1}^s z_t(P),
\end{equation}
where $C_0=3L+1$.
\end{claim}

\begin{proof}[Proof of the claim]
Fix $P\in\mathcal I$.
Recall that $(z_t(P))_{t=1}^k$ is non-increasing.
Let
\[
    m:=\max\{t\in[k]:z_t(P)>0\},
\]
with $m=0$ when all $z_t(P)$ vanish.  For $1\le q\le m$, set
\[
    U_q:=\bigcup_{t=1}^qJ_t(P),
    \qquad
    S_q:=\sum_{t=1}^qz_t(P).
\]
By \eqref{eq:z-size}, $\abs{U_q}=S_q$.

It is enough to prove \eqref{eq:one-level-majorization} when the first $s$
entries of $h^*(P)$ are positive.  Let $H\subset[k]$ be a set of $s$ rows
realizing those entries.

Suppose first that $s\le m$.  Then $z_s(P)>0$, so Algorithm~\ref{alg:peeling}
did not reset it to zero and
\[
    z_s(P)=\max_{u\in[k]}d_u(\eta,R_s(P)).
\]
Since
$R_{s+1}(P)=P\setminus U_s\subset R_s(P)$, we have
$\max_{u\in[k]}d_u(\eta,R_{s+1}(P))\le z_s(P)$, and therefore
\[
\begin{aligned}
    \sum_{t=1}^s h_t^*(P)
    &=\sum_{u\in H}d_u(\eta,P)\\
    &\le \sum_{u\in H}d_u(\eta,U_s)+s z_s(P)\\
    &\le L\abs{U_s}+sK_P+s z_s(P).
\end{aligned}
\]
Above, rows with $d_u(\eta,U_s)\ge K_P$ are controlled by
the heavy-row packing condition, while all remaining rows contribute at most
$sK_P$.  Moreover, by \eqref{eq:z-size} and the definition of $\kappa_P$,
\[
    S_s\ge s\kappa_P=\frac{sK_P}{2L},
\]
while $S_s\ge s z_s(P)$ by monotonicity of $(z_t(P))_{t=1}^k$.  Hence
\[
    \sum_{t=1}^s h_t^*(P)
    \le
    (3L+1)S_s.
\]

Now suppose that $s>m$.  Then $m<k$, and the construction gives
\[
    \max_{u\in[k]}d_u(\eta,P\setminus U_m)<\kappa_P.
\]
For each $u\in H$, choose a set
$V_u\subset (P\setminus U_m)\cap\{i:\eta_{ui}=1\}$ of the smallest cardinality for
which
\[
    d_u(\eta,U_m\cup V_u)\ge K_P.
\]
Such a set exists because $d_u(\eta,P)\ge K_P$, and it can be chosen with
$\abs{V_u}<\kappa_P$.  Put
\[
    W:=U_m\cup\bigcup_{u\in H}V_u.
\]
Then $W\subset P$, every row in $H$ has degree at least $K_P$ on $W$, and
\[
    \abs W\le S_m+s\kappa_P.
\]
The heavy-row packing condition and the definition of $\kappa_P$ give
\[
    sK_P
    \le\sum_{u\in H}d_u(\eta,W)
    \le L\abs W
    \le LS_m+Ls\kappa_P
    \le LS_m+\frac{sK_P}{2}.
\]
Thus
\begin{equation}\label{eq:sK-by-Sm}
    sK_P\le2LS_m.
\end{equation}
Splitting the degrees on $U_m$ at the threshold $K_P$, and using
$d_u(\eta,P\setminus U_m)<\kappa_P\le K_P$, we get
\[
\begin{aligned}
    \sum_{t=1}^s h_t^*(P)
    &=\sum_{u\in H}d_u(\eta,P)\\
    &=\sum_{u\in H}d_u(\eta,U_m)
      +\sum_{u\in H}d_u(\eta,P\setminus U_m)\\
    &\le
      \sum_{\substack{u\in H\\ d_u(\eta,U_m)\ge K_P}}d_u(\eta,U_m)
      +\sum_{\substack{u\in H\\ d_u(\eta,U_m)<K_P}}d_u(\eta,U_m)
      +s\kappa_P\\
    &\le LS_m+sK_P+s\kappa_P
    \le (3L+1)S_m.
\end{aligned}
\]
The penultimate inequality uses the heavy-row packing condition for the rows with
$d_u(\eta,U_m)\ge K_P$, the trivial bound $d_u(\eta,U_m)<K_P$ for the remaining
rows, and
\eqref{eq:sK-by-Sm}, together with $s\kappa_P\le S_m$, which follows from
\eqref{eq:sK-by-Sm} and the definition of $\kappa_P$.
This again proves \eqref{eq:one-level-majorization}.
\end{proof}

Now combine the levels.  Define nonnegative vectors $F,Z\in\R^k$ by
\[
    F_u:=\sum_{P\in\mathcal I}\omega_Ph_u(P),
    \qquad
    Z_t:=\sum_{P\in\mathcal I}\omega_Pz_t(P).
\]
Then $\norm F_2=\FM((P,\omega_P,K_P)_{P\in\mathcal I})$, and $Z$ is
non-increasing.
For every $s$, the characterization of the sum of the $s$ largest coordinates
gives
\[
\begin{aligned}
    \sum_{t=1}^sF_t^*
    &=\max_{\substack{H\subset[k]\\\abs H=s}}
      \sum_{u\in H}F_u\\
    &\le \sum_{P\in\mathcal I}\omega_P\sum_{t=1}^s h_t^*(P)\\
    &\le C_0\sum_{P\in\mathcal I}\omega_P\sum_{t=1}^s z_t(P)
     =C_0\sum_{t=1}^s Z_t.
\end{aligned}
\]
Thus $F^*$ is weakly majorized by $C_0Z$.

Without loss of generality, $F\ne0$. Set $y=F^*/\norm F_2$ and $y_{k+1}=0$.  Then
\[
\begin{aligned}
    \norm F_2
    &=\sum_{t=1}^k F_t^*y_t
      =\sum_{s=1}^k
        \left(\sum_{t=1}^sF_t^*\right)(y_s-y_{s+1})\\
    &\le C_0\sum_{s=1}^k
        \left(\sum_{t=1}^sZ_t\right)(y_s-y_{s+1})
      =C_0\sum_{t=1}^k Z_ty_t\\
    &\le C_0\norm Z_2.
\end{aligned}
\]
Therefore
\[
    \left[
    \sum_{t=1}^k
    \left(\sum_{P\in\mathcal I}\omega_Pz_t(P)\right)^2
    \right]^{1/2}
    \ge C_0^{-1}\FM((P,\omega_P,K_P)_{P\in\mathcal I}).
\]
\end{proof}

\subsection{Canonical-to-reduced comparison}

Recall from Subsection~\ref{sec:technical-overview-flat-contribution} the
canonical Flat row profile $\mathbf m_L^{\rm can}(x)$ and its majorant
$\FM_L^{\rm can}(x)$.  We now define the reduced objects used throughout
this section.

\begin{definition}[Reduced Flat profile and majorant]
\label{def:reduced-flat-profile-majorant}
For $L\ge1$, let the reduced threshold profile be
\[
    K_{j,{\rm red}}^{(L)}
    :=
    L\max\{pr_*,p\,2^{2j}\},
    \qquad j\in\mathbb Z,
\]
where we recall that
\[
    r_*:=\max\left\{r,\frac{\log k}{p}\right\}.
\]
For $x\in\mathbb R^n$, define the reduced flat profile
\[
    \bigl(\mathbf m_L^{\rm red}(x)\bigr)_u
    :=
    \sum_{j\in\mathbb Z}2^{-j}d_u(\eta,I_j(x))\,
    \mathbf 1_{\{d_u(\eta,I_j(x))\ge K_{j,{\rm red}}^{(L)}\}},
    \qquad u\in[k],
\]
and set $\FM_L^{\rm red}(x):=\norm{\mathbf m_L^{\rm red}(x)}_2$.
\end{definition}

The reduced profile is introduced for technical reasons;
it is easier to operate with compared to the canonical
profile with vector-dependent thresholds.
In particular, the reduced profile allows for a simple
proof of the following key lemma, which will later be used together with the peeling procedure
from the beginning of the section:

\begin{lemma}[Heavy-row packing condition for reduced profile]
\label{lem:scale-adapted-local-degree}
For every $B\ge1$ there is
$L_{\text{\tiny\ref*{lem:scale-adapted-local-degree}}}(B)\ge1$ with the
following property.  Assume
\[
    r^{10}\ge n\ge k\ge r\ge3,
    \qquad
    \frac{\log k}{k}\le p\le1,
\]
let $V\subset\R^n$ be a fixed subspace of dimension $r$, and let $\Pi$ follow
the admissible sparse-entry model with parameter $p$ and support mask $\eta$.  If
$L\ge L_{\text{\tiny\ref*{lem:scale-adapted-local-degree}}}(B)$, then, with
probability at least $1-k^{-B}$, simultaneously for every $x\in V$ with
$\norm{x}_2\le1$, every $j\in\mathbb Z$, and every $J\subset I_j(x)$,
\[
    R_{K_{j,{\rm red}}^{(L)}}(\eta,J)<2\abs J.
\]
\end{lemma}
\begin{proof}
For $j\in\mathbb Z$, put
$K_j:=K_{j,{\rm red}}^{(L)}=Lp\max\{r_*,2^{2j}\}$.  For
$j\in\mathbb Z$ and $s\ge1$, let
\[
    \mathcal F_{j,s}
    :=
    \left\{
    I\subset[n]:\abs I=s,\ \exists x\in V,\ \norm{x}_2\le1,\
    \abs{x_i}\ge2^{-j}\ \text{for every }i\in I
    \right\}.
\]
Lemma~\ref{lem:full-trace-counting} gives
\[
    \abs{\mathcal F_{j,s}}
    \le
    \left(
        C\frac{(r+s)2^{2j}}{s^2}
    \right)^s.
\]
Moreover, $\mathcal F_{j,s}\ne\varnothing$ implies $s\le2^{2j}$.

If $L$ is larger than a sufficiently large absolute constant, then
$\lceil K_j\rceil\ge2\log(e k)$ because $pr_*\ge\log k$.
Whenever $\mathcal F_{j,s}\ne\varnothing$, we also have
\[
    \frac{ps}{\lceil K_j\rceil}
    \le
    \frac{p2^{2j}}{Lp\max\{r_*,2^{2j}\}}
    \le\frac1L.
\]
Thus Lemma~\ref{lem:fixed-set-heavy-row-tail}, applied with $t=2s$, and a
union bound give
\begin{equation}
\label{eq:scale-adapted-fixed-level-failure}
\begin{aligned}
    \mathbb P\left\{
        \exists I\in\mathcal F_{j,s}:
        R_{K_j}(\eta,I)\ge2s
    \right\}                                                  
    \le
    \left(
        C'\frac{(r+s)2^{2j}p^2}{K_j^2}
    \right)^s
    \le
    \left(\frac{C''}{L^2}\right)^s,
\end{aligned}
\end{equation}
where the final inequality in \eqref{eq:scale-adapted-fixed-level-failure} follows
from
\[
    (r+s)2^{2j}
    \le(r_*+2^{2j})2^{2j}
    \le2\max\{r_*,2^{2j}\}^2.
\]

If the bad event in \eqref{eq:scale-adapted-fixed-level-failure} occurs, then
$s\ge K_j\ge Lpr_*\ge L\log k$.  Further, the left hand side of \eqref{eq:scale-adapted-fixed-level-failure}
is non-zero for only $O(\log k)$ values of $j$.
Indeed, $\mathcal F_{j,s}$ is empty for $j<0$, whereas the bad event implies
$Lp2^{2j}\le K_j\le s\le n$.  The assumptions
$n\le k^{10}$ and $p\ge(\log k)/k$ show that there are at most
$12\log(e k)$ relevant values of $j$.

Choose
$L_{\text{\tiny\ref*{lem:scale-adapted-local-degree}}}(B)$ so large that
$C''/L^2\le1/4$ and
$L\log(L^2/C'')\ge2(B+2)$.  Summing over $j\in\mathbb Z$ and
$s\ge\lceil L\log k\rceil$ then shows that the
total probability of the overcrowding of heavy row degrees is at most
\[
    12\log(e k)
    \sum_{s\ge\lceil L\log k\rceil}\left(\frac {C''}{L^2}\right)^s
    \le k^{-B}.
\]
On the complementary event, if $J\subset I_j(x)$ and $s=\abs J>0$, then
$J\in\mathcal F_{j,s}$, and hence $R_{K_j}(\eta,J)<2s$.
\end{proof}

The next lemma quantifies the cost of replacing the vector-dependent
canonical thresholds by the larger, $x$-independent reduced thresholds.

\begin{lemma}[Canonical-to-reduced Flat-profile comparison]
\label{lem:canonical-to-reduced-flat-comparison}
There is an absolute constant
$C_{\text{\tiny\ref*{lem:canonical-to-reduced-flat-comparison}}}\ge1$ with
the following property.  Let $0<p\le1$,
$r\ge1$, $L\ge1$, and let
$\mathcal N\subset S^{n-1}$.  For $x\in\mathcal N$, let
$\mathbf m_L^{\rm can}(x),\mathbf m_L^{\rm red}(x)\in\R^k$ be the Flat canonical and reduced row
profiles from Definitions~\ref{def:canonical-flat-profile-majorant} and
\ref{def:reduced-flat-profile-majorant}, respectively.
Condition on a realization of the mask $\eta\in\{0,1\}^{k\times n}$ such that
for some $A_0\ge1$,
\[
    \sup_{x\in\mathcal N}
    \sum_{j\in\mathbb Z}\sum_{u=1}^k
    d_u(\eta,I_j(x))
    \mathbf 1_{\{d_u(\eta,I_j(x))\ge K_j^{(L)}(x)\}}
    \le A_0r .
\]
Then
\[
    \sup_{x\in\mathcal N}
    \norm{\mathbf m_L^{\rm can}(x)-\mathbf m_L^{\rm red}(x)}_2
    \le
    C_{\text{\tiny\ref*{lem:canonical-to-reduced-flat-comparison}}}
    \sqrt{LA_0pr}.
\]
\end{lemma}

\begin{proof}
Fix $x\in\mathcal N$.  Put $s_j:=\abs{I_j(x)}$.  Recall that, since,
$\norm{x}_2=1$, we have $s_j\le 2^{2j}$ for every $j$, and therefore
\[
    K_{j,{\rm red}}^{(L)}
        \ge
    L\max\{pr_*,ps_j\}
    =
    K_j^{(L)}(x).
\]
That is, $\mathbf m_L^{\rm red}(x)\le\mathbf m_L^{\rm can}(x)$
coordinatewise.

For $u\in[k]$ and $j\in\mathbb Z$, set
\[
    a_{u j}
    :=
    d_u(\eta,I_j(x))
    \mathbf 1_{\{
        K_j^{(L)}(x)
        \le d_u(\eta,I_j(x))
        <K_{j,{\rm red}}^{(L)}
    \}} .
\]
Then
\[
    \bigl(\mathbf m_L^{\rm can}(x)-\mathbf m_L^{\rm red}(x)\bigr)_u
    =
    \sum_{j\in\mathbb Z}2^{-j}a_{u j}.
\]
If $2^{2j}\le r_*$, then $s_j\le2^{2j}\le r_*$, so
$K_{j,{\rm red}}^{(L)}=K_j^{(L)}(x)=Lpr_*$.  Thus $a_{u j}=0$
for all such $j$.  For the remaining levels, $2^{2j}>r_*$, and the strict upper
cutoff in the definition of $a_{u j}$ gives
\[
    0\le a_{u j}\le Lp2^{2j} .
\]
Moreover, by the assumption of the lemma,
\[
    \sum_{u=1}^k\sum_{j\in\mathbb Z} a_{u j}\le A_0r .
\]

We claim that, for every row $u$,
\[
    \left(\sum_{j\in\mathbb Z}2^{-j}a_{u j}\right)^2
    \le
    C Lp\sum_{j\in\mathbb Z}a_{u j}.
\]
Let $D_u:=\sum_{j\in\mathbb Z}a_{u j}$.  If $D_u=0$, there is nothing to prove.  If
$D_u\le Lpr_*$, then, since $a_{u j}$ is supported on levels with
$2^{2j}>r_*$,
\[
    \sum_{j\in\mathbb Z}2^{-j}a_{u j}
    =
    \sum_{j\in\mathbb Z}\frac{a_{u j}}{2^j}
    \le
    \frac{D_u}{\sqrt{r_*}}
    \le
    \sqrt{LpD_u}.
\]
Assume now that $D_u>Lpr_*$.  Splitting at $2^{2j}=D_u/(Lp)$, we get
\[
\begin{aligned}
    \sum_{j\in\mathbb Z}\frac{a_{u j}}{2^j}
    \le
    Lp\sum_{\substack{j\in\mathbb Z:\,
        r_*<2^{2j}\le D_u/(Lp)}}2^j
    +
    \frac{D_u}{\sqrt{D_u/(Lp)}}.
\end{aligned}
\]
Here the first inequality uses $a_{u j}\le Lp2^{2j}$ on the first range.
Consequently,
\[
    \sum_{j\in\mathbb Z}\frac{a_{u j}}{2^j}
    \le
    2Lp\sqrt{\frac{D_u}{Lp}}
    +\frac{D_u}{\sqrt{D_u/(Lp)}}
    \le
    C\sqrt{LpD_u}.
\]
This proves the claim.

Summing the claim over $u$ and using the total degree assumption,
\[
\begin{aligned}
    \norm{\mathbf m_L^{\rm can}(x)-\mathbf m_L^{\rm red}(x)}_2^2
    &\le
    C Lp\sum_{u=1}^k\sum_{j\in\mathbb Z}a_{u j}             \\
    &\le
    C Lp\,A_0r .
\end{aligned}
\]
Taking square roots and then the supremum over $x\in\mathcal N$ completes the
proof.
\end{proof}

Lemma~\ref{lem:flat-contribution-z}, the uniform estimate in
Lemma~\ref{lem:scale-adapted-local-degree}, and Lemma~\ref{lem:canonical-to-reduced-flat-comparison}
yield the following extraction
proposition.

\begin{proposition}
\label{thm:flat-extraction}
For every $\alpha\ge0$ and $B\ge1$ there is
$L_{\text{\tiny\ref*{thm:flat-extraction}}}(\alpha,B)\ge1$ such that the
following holds.  Assume
\[
    r^{10}\ge n\ge k\ge r\ge3,
    \qquad
    \frac{\log k}{k}\le p\le1.
\]
Let $V\subset\R^n$ be a fixed subspace of dimension $r$, and let
${\mathcal N}_V\subset V\cap S^{n-1}$ be deterministic with
$\abs{{\mathcal N}_V}\le\exp(\alpha r)$.  If
$L\ge L_{\text{\tiny\ref*{thm:flat-extraction}}}(\alpha,B)$, then there is
$C=C(\alpha,B,L)\ge1$ such that, with probability at least $1-k^{-B}$,
simultaneously for every $x\in{\mathcal N}_V$ there are integers
$y_j=y_j(x)\ge0$, sets
\[
    J_{j,q}(x)\subset[n],
    \qquad j\in\mathbb Z,\quad 1\le q\le y_j,
\]
and row indices $h_{j,q}(x)\in[k]$ satisfying the following properties.
\begin{itemize}
\item The sets $J_{j,q}(x)$ are nonempty and pairwise disjoint over all
      pairs $(j,q)$.
\item For every $j\in\mathbb Z$ and $1\le q\le y_j$,
      \[
          J_{j,q}(x)
          \subset
          I_j(x)\cap\{i\in[n]:\eta_{h_{j,q}(x)i}=1\}.
      \]
\item For every fixed $j$, the rows
      $h_{j,1}(x),\ldots,h_{j,y_j}(x)$ are distinct.
\item For every fixed $j$, the sequence
      $(\abs{J_{j,q}(x)})_{q=1}^{y_j}$ is non-increasing.  Each
      $\abs{J_{j,q}(x)}$ is a power of two and satisfies
      \[
          \abs{J_{j,q}(x)}
          \ge
          \frac L{10}\max\{pr_*,p2^{2j}\}.
      \]
\item With $J_{j,q}(x)=\varnothing$ for $q>y_j$,
      \[
          \FM_L^{\rm can}(x)
          \le
          14\left[
          \sum_{q=1}^k
          \left(
              \sum_{j\in\mathbb Z}2^{-j}\abs{J_{j,q}(x)}
          \right)^2
          \right]^{1/2}
          +C\sqrt{pr}.
      \]
\end{itemize}
\end{proposition}

\begin{proof}
Take
\[
    L_{\text{\tiny\ref*{thm:flat-extraction}}}(\alpha,B)
    :=
    \max\left\{
        L_{\text{\tiny\ref*{lem:scale-adapted-local-degree}}}(B+1),
        L_{\text{\tiny\ref*{lem:total-degree-canonical-heavy-rows}}}
    \right\}.
\]
Apply Lemma~\ref{lem:total-degree-canonical-heavy-rows} to
${\mathcal N}_V$ with failure exponent $B+1$, and apply
Lemma~\ref{lem:scale-adapted-local-degree} with the same failure exponent.
Since $k\ge3$, the intersection of the resulting events has probability at
least $1-2k^{-(B+1)}\ge1-k^{-B}$.  Work on this intersection, and fix
$x\in{\mathcal N}_V$.  Apply Lemma~\ref{lem:flat-contribution-z} to the
pairwise disjoint family of nonempty sets $I_j(x)$, with
\[
    \omega_{I_j(x)}:=2^{-j},
    \qquad
    K_{I_j(x)}:=K_{j,{\rm red}}^{(L)},
\]
and heavy-row packing constant $2$.  Its hypotheses follow from Lemma~\ref{lem:scale-adapted-local-degree}, and its conclusion gives
\[
    \left[
    \sum_{q=1}^k
    \left(\sum_{j\in\mathbb Z}2^{-j}z_q(I_j(x))\right)^2
    \right]^{1/2}
    \ge
    \frac17\FM_L^{\rm red}(x).
\]

For every $j$, let $y_j$ be the number of positive terms in the
non-increasing sequence $(z_q(I_j(x)))_{q=1}^k$.  For $1\le q\le y_j$, put
\[
    \widehat z_{j,q}
    :=2^{\lfloor\log_2z_q(I_j(x))\rfloor},
\]
let $J_{j,q}(x)$ consist of the first $\widehat z_{j,q}$ elements of
$J_q(I_j(x))$ in the natural order on $[n]$, and set
$h_{j,q}(x):=u_q(I_j(x))$. The sequence $(\widehat z_{j,q})_q$ is non-increasing,
and, since
\[
    \frac12z_q(I_j(x))\le\widehat z_{j,q}\le z_q(I_j(x)),
\]
we have
\[
    \abs{J_{j,q}(x)}
    \ge\frac18K_{j,{\rm red}}^{(L)}
    \ge\frac L{10}\max\{pr_*,p2^{2j}\}.
\]
The peeling construction gives disjointness within each level, containment
in the support of the assigned row, and distinct assigned rows.  Different
levels are disjoint because the sets $I_j(x)$ are disjoint.

Finally, dyadic rounding loses at most a factor two in every coordinate of
the extracted vector.  Therefore
\[
\begin{aligned}
    \left[
    \sum_{q=1}^k
    \left(
        \sum_{j\in\mathbb Z}2^{-j}\abs{J_{j,q}(x)}
    \right)^2
    \right]^{1/2}
    &\ge
    \frac12
    \left[
    \sum_{q=1}^k
    \left(\sum_{j\in\mathbb Z}2^{-j}z_q(I_j(x))\right)^2
    \right]^{1/2} \\
    &\ge\frac1{14}\FM_L^{\rm red}(x).
\end{aligned}
\]
On the total-degree event,
Lemma~\ref{lem:canonical-to-reduced-flat-comparison}, with
\[
    A_0
    :=
    C_{\text{\tiny\ref*{lem:total-degree-canonical-heavy-rows}}}
        (\alpha,B+1),
\]
gives
\[
    \norm{\mathbf m_L^{\rm can}(x)-\mathbf m_L^{\rm red}(x)}_2
    \le C\sqrt{LA_0pr}.
\]
Consequently, the triangle inequality yields
\[
\begin{aligned}
    \FM_L^{\rm can}(x)
    &\le
    \FM_L^{\rm red}(x)
    +\norm{\mathbf m_L^{\rm can}(x)-\mathbf m_L^{\rm red}(x)}_2\\
    &\le
    14
    \left[
    \sum_{q=1}^k
    \left(
        \sum_{j\in\mathbb Z}2^{-j}\abs{J_{j,q}(x)}
    \right)^2
    \right]^{1/2}
    +C(\alpha,B,L)\sqrt{pr}.
\end{aligned}
\]
All events and estimates are uniform over ${\mathcal N}_V$, which completes
the proof.
\end{proof}

\subsection{Two-level partitions and leaf square functions}
\label{sec:two-level-partition}

This subsection reorganizes the disjoint family supplied by
Proposition~\ref{thm:flat-extraction} into a two-level rooted forest.
The roots record the dyadic coordinate levels of the vector, while each leaf
groups row-supported pieces of a common dyadic cardinality and records their
multiplicity.  This compresses an irregular family of row supports into the
dyadic parameters consisting of a coordinate level, a common width, and a
multiplicity, while retaining the relevant $\ell_2$ mass.  It therefore
allows the probabilistic argument to treat one width at a time and then sum
the resulting estimates.  We then encode the leaves by a square function and
show that it controls the canonical Flat majorant up to the
$O(\sqrt{pr})$ error in Proposition~\ref{thm:flat-extraction}.  This
representation prepares the fixed-width probabilistic estimates of
Section~\ref{sec:fixed-width-estimates}.

\begin{definition}[Two-level partition]
\label{def:two-level-partition}
Let $\eta\in\{0,1\}^{k\times n}$ be a mask and let $x\in\R^n$.  A
\emph{two-level partition associated with $(\eta,x)$} is a finite rooted forest
with the following labels and properties.

\begin{itemize}
\item Every root is labeled by a nonempty level set $I_j(x)$ and its level
      index $j$.  Distinct roots correspond to distinct levels.

\item A leaf below the root $I_j(x)$ is labeled by
      \[
          (E,j(E),w(E),b(E),m(E),H(E)),
          \qquad
          j(E)=j,
          \qquad
          w(E)=2^{-j(E)}.
      \]
      The sets $E$ are pairwise disjoint subsets of their roots.  The
      parameters $b(E)$ and $m(E)$ are positive dyadic integers, and
      $H(E)\subset[k]$.  They satisfy
      \[
          \abs E=b(E)m(E),
          \qquad
          \abs{H(E)}=m(E).
      \]
      For every $u\in H(E)$ there is a set
      \[
          S_u(E)\subset E\cap\{i\in[n]:\eta_{ui}=1\}
      \]
      of cardinality $b(E)$, and
      \[
          E=\bigsqcup_{u\in H(E)}S_u(E).
      \]
      In particular, $d_u(\eta,E)\ge b(E)$ for every $u\in H(E)$.  Distinct
      leaves below the same root have distinct values of $b(E)$.  Moreover,
      for every pair of positive dyadic integers $(b,m)$, there is at most
      one leaf $E$ in the entire forest satisfying
      $b(E)=b$ and $m(E)=m$.
\end{itemize}
\end{definition}

\begin{proposition}[Equal-size leaf grouping]
\label{prop:equal-size-leaf-grouping}
Assume the hypotheses of
Proposition~\ref{thm:flat-extraction}, and work on the event in its
conclusion.  Fix $x\in{\mathcal N}_V$, and let
$\bigl(J_{j,q}(x),h_{j,q}(x)\bigr)_{j,q}$ be a family supplied by that
proposition.  Then one can construct a two-level partition such that
\begin{equation}
\label{eq:leaf-square-function-dominates-extraction}
\begin{aligned}
    &\left[
    \sum_{t=1}^k
    \left(
        \sum_{\substack{E\text{ leaf of the partition}}}
        w(E)b(E)\mathbf 1_{\{t\le m(E)\}}
    \right)^2
    \right]^{1/2} \\
    &\qquad\ge
    \frac14
    \left[
    \sum_{q=1}^k
    \left(
        \sum_{j\in\mathbb Z}2^{-j}\abs{J_{j,q}(x)}
    \right)^2
    \right]^{1/2},
\end{aligned}
\end{equation}
and such that every constructed leaf satisfies
\begin{equation}
\label{eq:leaf-scale-lower-bound}
    b(E)
    \ge
    \frac L{10}\max\{pr_*,p2^{2j(E)}\}.
\end{equation}
\end{proposition}

\begin{proof}
Suppress $x$ from the notation.  For every $j\in\mathbb Z$ and every
positive dyadic integer $b$, put
\[
    Q_{j,b}:=\{q\le y_j:\abs{J_{j,q}}=b\},
    \qquad
    r_{j,b}:=\abs{Q_{j,b}}.
\]
Put
\[
    m_{j,b}
    :=
    \begin{cases}
        2^{\lfloor\log_2r_{j,b}\rfloor},&r_{j,b}>0,\\
        0,&r_{j,b}=0.
    \end{cases}
\]
If $r_{j,b}=0$, create no leaf.  Otherwise, let $Q'_{j,b}$ consist of the
first $m_{j,b}$ indices in $Q_{j,b}$.
Define
\[
    E_{j,b}:=\bigcup_{q\in Q'_{j,b}}J_{j,q},
    \qquad
    H(E_{j,b}):=\{h_{j,q}:q\in Q'_{j,b}\},
\]
and set
\[
    j(E_{j,b}):=j,
    \qquad
    w(E_{j,b}):=2^{-j},
    \qquad
    b(E_{j,b}):=b,
    \qquad
    m(E_{j,b}):=m_{j,b}.
\]
For $u=h_{j,q}$, define
\[
    S_u(E_{j,b}):=J_{j,q}.
\]
For every $j$ for which at least one leaf $E_{j,b}$ is constructed, create
the root labeled by $(I_j(x),j)$ and attach all such leaves to it.

The sets $E_{j,b}$ are pairwise disjoint because the sets $J_{j,q}$ are
pairwise disjoint.  For fixed $j$, the assigned rows $h_{j,q}$ are distinct,
so
\[
    \abs{H(E_{j,b})}=m_{j,b}.
\]
The support containment from Proposition~\ref{thm:flat-extraction}
shows that every $S_u(E_{j,b})$ is contained in the support of its assigned
row.  Thus all structural conditions in
Definition~\ref{def:two-level-partition}, except possibly the global
uniqueness condition for $(b(E),m(E))$, are satisfied.  There is at most one
leaf below $I_j(x)$ for each value of $b$, and
\eqref{eq:leaf-scale-lower-bound} follows directly from the size lower bound
in Proposition~\ref{thm:flat-extraction}.

We now thin this preliminary family.  For positive dyadic integers $b,m$,
put
\[
    \mathcal J_{b,m}
    :=
    \{j\in\mathbb Z:m_{j,b}=m\}.
\]
If $\mathcal J_{b,m}\ne\varnothing$, let
\[
    j_*(b,m):=\min\mathcal J_{b,m}
\]
and retain only the leaf $E_{j_*(b,m),b}$ from this class.  Discard all
other leaves, and delete any root with no retained leaf.  The retained
family satisfies the global uniqueness condition in
Definition~\ref{def:two-level-partition}.  All other structural properties,
including \eqref{eq:leaf-scale-lower-bound}, are inherited from the
preliminary family.

It remains to prove the comparison.  As in
Proposition~\ref{thm:flat-extraction}, set
$J_{j,q}=\varnothing$ for $q>y_j$, and define nonnegative non-increasing vectors
$Z,G^{(0)},G\in\R^k$ by
\[
    Z_q:=\sum_{j\in\mathbb Z}2^{-j}\abs{J_{j,q}},
    \qquad
    G_t^{(0)}
    :=
    \sum_{\substack{j\in\mathbb Z,\ b\ge1\ {\rm dyadic}}}
    2^{-j}b\,\mathbf 1_{\{t\le m_{j,b}\}},
\]
and
\[
    G_t
    :=
    \sum_{\substack{b,m\ge1\ {\rm dyadic}:\\
        \mathcal J_{b,m}\ne\varnothing}}
    2^{-j_*(b,m)}b\,\mathbf 1_{\{t\le m\}}.
\]
For every fixed $j$ and every $s\le k$,
\begin{align*}
    \sum_{q=1}^s\abs{J_{j,q}}
    \le
    \sum_{\substack{b\ge1\\b\ \mathrm{dyadic}}}
    b\min\{s,r_{j,b}\} 
    \le
    2\sum_{\substack{b\ge1\\b\ \mathrm{dyadic}}}
    b\min\{s,m_{j,b}\} 
    =
    2\sum_{t=1}^s
    \sum_{\substack{b\ge1\\b\ \mathrm{dyadic}}}
    b\,\mathbf 1_{\{t\le m_{j,b}\}}.
\end{align*}
Here we used $m_{j,b}\le r_{j,b}<2m_{j,b}$ when $r_{j,b}>0$, while both
terms vanish when $r_{j,b}=0$.  Multiplying by $2^{-j}$ and
summing over $j$ gives
\[
    \sum_{q=1}^s Z_q\le2\sum_{t=1}^sG_t^{(0)}
    \qquad (s\le k).
\]
For each nonempty $\mathcal J_{b,m}$, the numbers $2^{-j}$ are distinct
dyadic values and $j_*(b,m)$ is their smallest index.  Therefore
\[
    \sum_{j\in\mathcal J_{b,m}}2^{-j}
    \le
    2\cdot 2^{-j_*(b,m)}.
\]
Consequently, $G_t^{(0)}\le2G_t$ for every $t$.  It follows that $Z$ is
weakly majorized by $4G$.  The Euclidean norm is monotone under weak
majorization of nonnegative vectors, and hence
\[
    \norm Z_2\le4\norm G_2.
\]
This is exactly
\eqref{eq:leaf-square-function-dominates-extraction}.
\end{proof}

The fixed tie-breaking rule in
Algorithm~\ref{alg:peeling}, followed by the deterministic rounding and
thinning in Proposition~\ref{prop:equal-size-leaf-grouping}, select a unique
partition whenever the heavy-row packing event of
Lemma~\ref{lem:scale-adapted-local-degree} holds.  We denote it by
$\mathfrak T_L(\eta,x)$; outside that event, let
$\mathfrak T_L(\eta,x)$ be the empty forest.  This convention makes the
partition a deterministic function of $(\eta,x)$ and removes any choice from
the probabilistic statements below.

\begin{definition}[Leaf square function]
\label{def:leaf-square-function}
For a two-level partition $\mathfrak T$, its \emph{leaf square function} is
\[
    \mathcal Q_{\rm leaf}(\mathfrak T)
    :=
    \left[
    \sum_{\substack{1\le M\le k\\M\ \mathrm{dyadic}}}
    M
    \left(
        \sum_{\substack{E\text{ leaf of }\mathfrak T:\ m(E)=M}}
        w(E)b(E)
    \right)^2
    \right]^{1/2}.
\]
\end{definition}

\begin{lemma}[Dyadic leaf square-function comparison]
\label{lem:dyadic-leaf-square-function-comparison}
Let $\mathfrak T$ be a two-level partition, and let
\[
    Y_t
    :=
    \sum_{\substack{E\text{ leaf of }\mathfrak T}}
    w(E)b(E)\mathbf 1_{\{t\le m(E)\}},
    \qquad 1\le t\le k.
\]
Then
\[
    \mathcal Q_{\rm leaf}(\mathfrak T)
    \le
    \norm Y_2
    \le
    C\mathcal Q_{\rm leaf}(\mathfrak T)
\]
for an absolute constant $C$.
\end{lemma}

\begin{proof}
For dyadic $M\le k$, put
\[
    A_M
    :=
    \sum_{\substack{E\text{ leaf of }\mathfrak T:\ m(E)=M}}
    w(E)b(E),
    \qquad
    (e_M)_t:=M^{-1/2}\mathbf 1_{\{t\le M\}},
    \quad 1\le t\le k.
\]
Then
\[
    Y=\sum_{\substack{1\le M\le k\\M\ \mathrm{dyadic}}}
    A_M\sqrt M\,e_M.
\]
Since all coefficients are nonnegative,
\[
    \norm Y_2^2
    =
    \sum_{\substack{1\le M,N\le k\\M,N\ \mathrm{dyadic}}}
    A_MA_N\min\{M,N\}
    \ge
    \sum_{\substack{1\le M\le k\\M\ \mathrm{dyadic}}}MA_M^2
    =
    \mathcal Q_{\rm leaf}(\mathfrak T)^2.
\]
For the reverse inequality, write $M=2^a$ and $N=2^b$.  Then
\[
    \langle e_M,e_N\rangle
    =
    2^{-\abs{a-b}/2}.
\]
The inner product decay implies
\[
\begin{aligned}
    \norm Y_2^2
    \leq
    C\mathcal Q_{\rm leaf}(\mathfrak T)^2,
\end{aligned}
\]
completing the proof.
\end{proof}

\begin{corollary}[Canonical majorant for the selected partition]
\label{cor:canonical-flat-two-level-partition}
For every $\alpha\ge0$ and $B\ge1$ there is $L_0=L_0(\alpha,B)\ge1$ with
the following property.  Assume
\[
    r^{10}\ge n\ge k\ge r\ge3,
    \qquad
    \frac{\log k}{k}\le p\le1.
\]
Let $V\subset\R^n$ be a fixed subspace of dimension $r$, and let
${\mathcal N}_V\subset V\cap S^{n-1}$ be deterministic with
$\abs{{\mathcal N}_V}\le\exp(\alpha r)$.  If $L\ge L_0$, then, with probability
at least $1-k^{-B}$, simultaneously for every $x\in{\mathcal N}_V$, the
two-level partition $\mathfrak T_L(\eta,x)$ satisfies
\[
    \FM_L^{\rm can}(x)
    \le
    C\mathcal Q_{\rm leaf}\bigl(\mathfrak T_L(\eta,x)\bigr)
    +C(\alpha,B,L)\sqrt{pr}.
\]
\end{corollary}

\begin{proof}
Take
\[
    L_0(\alpha,B)
    :=L_{\text{\tiny\ref*{thm:flat-extraction}}}(\alpha,B).
\]
Work on the event from Proposition~\ref{thm:flat-extraction}. 
Proposition~\ref{thm:flat-extraction} gives
\[
    \FM_L^{\rm can}(x)
    \le
    14
    \left[
    \sum_{q=1}^k
    \left(
        \sum_{j\in\mathbb Z}2^{-j}\abs{J_{j,q}(x)}
    \right)^2
    \right]^{1/2}
    +C(\alpha,B,L)\sqrt{pr}.
\]
Proposition~\ref{prop:equal-size-leaf-grouping} bounds the extraction
profile by four times the leaf profile, and
Lemma~\ref{lem:dyadic-leaf-square-function-comparison} bounds the leaf profile
by a universal multiple of
$\mathcal Q_{\rm leaf}\bigl(\mathfrak T_L(\eta,x)\bigr)$.  Hence
\[
    \FM_L^{\rm can}(x)
    \le
    C\mathcal Q_{\rm leaf}\bigl(\mathfrak T_L(\eta,x)\bigr)
    +C(\alpha,B,L)\sqrt{pr}.
\]
The construction of $\mathfrak T_L(\eta,x)$ is deterministic, so the
conclusion is simultaneous over the net.
\end{proof}

\section{Fixed-width estimates for two-level partitions}
\label{sec:fixed-width-estimates}

Section~\ref{sec:flat-majorants} reduces the canonical Flat majorant to
a square function indexed by the leaves of a two-level partition.  Each leaf
$E$ carries three dyadic parameters: its coordinate weight $w(E)$, its row
width $b(E)$, and the number $m(E)$ of supporting rows; its cardinality is
$\abs E=b(E)m(E)$.  The purpose of this section is to control the leaf square
function one fixed value of $b(E)$ at a time, uniformly over the net vectors
and the corresponding two-level partitions.

For a fixed width, we order the leaves by their coordinate levels and thin
them to subsequences whose cardinalities grow geometrically.  This produces
the separated extractions treated in
Proposition~\ref{prop:separated-fixed-width-estimate}.  The probabilistic
estimate for such an extraction divides its
leaves into two classes: regular leaves are charged to orthogonal vector
increments, whereas exceptional leaves are witnessed by a
lower-dimensional prefix and are controlled 
by balancing entropy and probability bounds.
Summing these two estimates gives the desired
fixed-width bound.  The auxiliary estimates are proved under a linear
leaf-size restriction, but Lemma~\ref{lem:total-degree-canonical-heavy-rows}
shows that this restriction holds automatically for every admissible
extraction associated with a vector in the prescribed net.

\subsection{Fixed-width profiles and proof setup}

\begin{definition}[Fixed-width leaf profile]
\label{def:fixed-width-leaf-profile}
For a two-level partition $\mathfrak T$ associated with $(\eta,x)$ and a
positive dyadic integer $b$, define
\[
    Y_t^{(b)}(\mathfrak T)
    :=
    \sum_{\substack{E\text{ leaf of }\mathfrak T:\ b(E)=b}}
    w(E)b\,\mathbf 1_{\{t\le m(E)\}},
    \qquad 1\le t\le k.
\]
\end{definition}

Throughout this section, assume
\[
    r^{10}\ge n\ge k\ge r\ge3,
    \qquad
    \frac{\log k}{k}\le p\le1,
\]
$\Pi$ follows the admissible sparse-entry model of
Definition~\ref{def:admissible-sparse-entry-model} with parameter $p$ and
support mask $\eta$, $V\subset\mathbb R^n$ is a fixed
$r$-dimensional subspace, $L\ge1$, and
${\mathcal N}_V\subset V\cap S^{n-1}$ is deterministic with
$\abs{{\mathcal N}_V}\le\exp(\alpha r)$ for fixed $\alpha\ge0$.  The goal is
the following simultaneous fixed-width estimate.

\begin{proposition}[Fixed-width target]
\label{thm:target-bound}
For every $\alpha\ge0$ and $B\ge1$, there are constants
$L_0=L_0(\alpha,B)$ and $C=C(\alpha,B)$ such that the following holds under
the standing assumptions of this section.  If $L\ge L_0$, then, with
probability at least $1-k^{-B}$, simultaneously for every
$x\in{\mathcal N}_V$ and every positive dyadic integer $b$,
\[
    \norm{Y^{(b)}\bigl(\mathfrak T_L(\eta,x)\bigr)}_2
    \le C\sqrt{pr}.
\]
\end{proposition}

The proof combines the entropy bound with a {\it delayed-prefix
dimension reduction} (to be discussed later). 

\subsection{Separated fixed-width extractions}

We next prove the probabilistic statement used for a fixed value of $b$.
The proposition is stated uniformly over every family
$E_1,\ldots,E_s$ satisfying the conditions below.  This uniformity is
needed because the selected partition $\mathfrak T_L(\eta,x)$, and hence
the record-leaf families used in the proof of
Proposition~\ref{thm:target-bound}, is constructed only after the mask
$\eta$ has been realized.  On the event supplied by the proposition, its
estimate can therefore be applied to these random families.

\begin{proposition}[Fixed-width estimate for exponentially growing set sizes]
\label{prop:separated-fixed-width-estimate}
For every $\alpha\ge0$ and $B\ge1$ there are constants
$L_0=L_0(\alpha,B)$, $\gamma>1$, and $C=C(\alpha,B)$ with the following
property.  Assume the
standing assumptions of this section, and let
$L\ge L_0$.  With probability at least $1-k^{-B}$, the following holds
simultaneously.

Let $b$ be a positive dyadic integer, let $x\in{\mathcal N}_V$, and let
$E_1,\ldots,E_s$ be pairwise disjoint sets with associated indices
$j_1<\cdots<j_s$ and $M_\ell:=\abs{E_\ell}$.
Suppose that for all admissible $\ell$
\begin{align}
    &E_\ell\subset I_{j_\ell}(x),
    \qquad
    M_\ell\text{ is dyadic},
    \qquad
    M_{\ell+1}\ge\gamma M_\ell,
    \notag\\
    &b\ge\frac L{10}\max\{pr_*,p2^{2j_\ell}\},
    \label{eq:separated-extraction-scale}
\end{align}
and suppose that each $E_\ell$ is the disjoint union of $m_\ell=M_\ell/b$
sets of cardinality $b$, each contained in the support of a different row of $\eta$.
Then
\begin{equation}
\label{eq:separated-fixed-width-conclusion}
    \sum_{\ell=1}^s M_\ell\, b\,2^{-2j_\ell}
    \le
    Cpr.
\end{equation}
\end{proposition}

\medskip

We postpone the proof of the above proposition till the end of the subsection, 
when all necessary auxiliary statements are proved.

\medskip

For the two auxiliary estimates below, fix $c_{\rm size}\ge1$,
$\gamma>1$, and a positive integer $h$. 
We will assume that the constants are large; their values can be extracted from the proofs below.
To formalize the proof, let us introduce the following local definition of an admissible data structure.

\begin{definition}[Local admissible data structure]
\label{def:fixed-width-admissible-data}
Fix a realization of the support mask $\eta$.  An admissible data
structure, with parameters $L,c_{\rm size},\gamma,h$, consists of a positive
integer $s$, a positive dyadic integer $b$, a vector
$x\in\mathcal N_V$, integer levels
\[
    j_1<\cdots<j_s,
\]
and sets $E_1,\ldots,E_s\subset[n]$.  Put
$M_\ell:=\abs{E_\ell}$.  We require the following conditions:
\begin{itemize}
    \item The sets $E_1,\ldots,E_s$ are pairwise disjoint, and, for every
    $1\le\ell\le s$,
    \[
        E_\ell\subset I_{j_\ell}(x),
        \qquad
        M_\ell\text{ is dyadic}.
    \]
    \item Their cardinalities grow geometrically:
    \[
        M_{\ell+1}\ge\gamma M_\ell,
        \qquad 1\le\ell<s.
    \]
    \item The common width satisfies, for every $1\le\ell\le s$,
    \[
        b\ge\frac L{10}\max\{pr_*,p2^{2j_\ell}\}.
    \]
    \item For every $1\le\ell\le s$, there is a set
    $H_\ell\subset[k]$ of cardinality
    \[
        m_\ell:=\frac{M_\ell}{b}
    \]
    and a prescribed row-chunk decomposition
    \[
        E_\ell=\bigsqcup_{u\in H_\ell}S_u(E_\ell),
        \qquad
        S_u(E_\ell)
        \subset E_\ell\cap\{i\in[n]:\eta_{ui}=1\},
        \qquad
        \abs{S_u(E_\ell)}=b.
    \]
\end{itemize}
In addition, assume the linear-size condition
\begin{equation}
\label{eq:linear-size-assumption}
    M_\ell\le c_{\rm size}\,r,
    \qquad 1\le\ell\le s.
\end{equation}

For such a data structure, define
\[
    U_\ell
    :=\operatorname{span}\{P_Ve_i:i\in E_1\cup\cdots\cup E_\ell\},
    \qquad U_0:=\{0\},
\]
and let $P_\ell$ be the orthogonal projection onto $U_\ell$.  Put
\begin{equation}\label{eq:delayed increment def}
    z_\ell:=(P_\ell-P_{\ell-h})x,
\end{equation}
where $P_q=0$ for $q\le0$.
Call a chunk $S_u(E_\ell)$ {\it regular} if at least $b/2$ of its coordinates
satisfy
\[
    |(z_\ell)_i|\ge2^{-j_\ell-1}.
\]
Call the set $E_\ell$ {\it regular} if at least half of its chunks are regular;
otherwise call the set {\it exceptional}.
\end{definition}

\begin{lemma}[Uniform regular-set estimate]
\label{lem:uniform-regular-set-estimate}
For every $B\ge1$ there are constants $L_0=L_0(B,c_{\rm size})$ and
$C=C(B,c_{\rm size})$ with the following property.  If $L\ge L_0$, then,
with probability at least $1-k^{-B}$, simultaneously for every $\gamma>1$,
every positive integer $h$, every choice of
$b,x,(E_\ell,j_\ell)_{\ell=1}^s$, and every prescribed row-chunk
decomposition
\[
    E_\ell=\bigsqcup_{u\in H_\ell}S_u(E_\ell),
    \qquad 1\le\ell\le s,
\]
which together form an admissible data structure in the sense of
Definition~\ref{def:fixed-width-admissible-data}, every regular $E_\ell$
satisfies
\begin{equation}
\label{eq:regular-set-bound}
    M_\ell b2^{-2j_\ell}
    \le
    C rp\,\norm{z_\ell}_2^2.
\end{equation}
\end{lemma}

\begin{proof}
Take for a moment any regular $E_\ell$. Select exactly
$\lceil m_\ell/2\rceil$ regular chunks and exactly $b/2$ good coordinates
from each selected chunk (for concreteness, we can select using lexicographic ordering).  This produces a set $J\subset E_\ell$ with
\begin{equation}
\label{eq:regular-certificate-size}
    \frac{M_\ell}{4}\le q:= \lceil m_\ell/2\rceil\,\frac{b}{2} =\abs J\le M_\ell,
\end{equation}
such that $R_{b/2}(\eta,J)\ge q$.  Moreover, $\norm{z_\ell}_2>0$, and the unit
vector $z_\ell/\norm{z_\ell}_2\in V$ has absolute coordinates at least
$2^{-j_\ell-1}/\norm{z_\ell}_2$ on $J$.
We call such a set $J$ a {\it regular certificate} for $E_\ell$.
Let $\bar a_\ell$ be the smallest dyadic number not smaller than
$\norm{z_\ell}_2^2$.  Then
$\norm{z_\ell}_2^2\le\bar a_\ell<2\norm{z_\ell}_2^2$, and
every regular certificate belongs to
\[
    \mathcal F_V\left(
        \frac{2^{-j_\ell-1}}{\sqrt{\bar a_\ell}},q
    \right).
\]
For the union bound, we will group the regular
certificates according to the discrete profiles
\[
    (b,M_\ell,j_\ell,\bar a_\ell).
\]
Our goal is to show that the probability that
a regular certificate exists from some profile 
{\it not satisfying} \eqref{eq:regular-set-bound},
is small.

For a fixed profile, apply
Corollary~\ref{lem:entropy-support-certificate} with
\[
    W=V,
    \qquad
    \beta=\frac{2^{-j_\ell-1}}{\sqrt{\bar a_\ell}},
    \qquad
    \kappa=\frac b2.
\]
The hypotheses of that corollary follow from
\eqref{eq:separated-extraction-scale}, after increasing $L_0$.  Since
$q\le c_{\rm size}r$ and $q\ge M_\ell/4$, the corollary sums the support
probabilities over every possible certificate set $J$ and shows that the
probability that a regular certificate with this profile exists is at most
\begin{equation}
\label{eq:regular-certificate-cost}
    \left(
        C\frac{rp\bar a_\ell\,2^{2j_\ell}}
        {bM_\ell}
    \right)^q.
\end{equation}

Note that \eqref{eq:regular-certificate-size} gives
\[
    \norm{z_\ell}_2^2
    \ge \frac14q2^{-2j_\ell}
    \ge cM_\ell2^{-2j_\ell}
    \ge cLp.
\]
Here the last inequality follows because $M_\ell=bm_\ell\ge b$ and
\eqref{eq:separated-extraction-scale} gives
$b\,2^{-2j_\ell}\ge (L/10)p$.
Thus $cLp\le\bar a_\ell\le2$,
implying that 
there are $O(\log k)$ possible dyadic
values of $\bar a_\ell$. Further, there are
$O(\log k)$ choices for each of the dyadic parameters $b$
and $M_\ell$.  Further, $I_{j_\ell}(x)\ne\varnothing$ implies
$j_\ell\ge1$, while \eqref{eq:separated-extraction-scale} and
$b\le M_\ell\le c_{\rm size}r$ give
\[
    2^{2j_\ell}\le \frac{10c_{\rm size}r}{Lp},
\]
so there are $O(\log k)$ relevant levels.
We conclude that in total there are of order $O(\log^4 k)$
profiles.
If
\[
    M_\ell b2^{-2j_\ell}>C_{\rm reg}rp\,\norm{z_\ell}_2^2
\]
for a large constant $C_{\rm reg}$, then
$\bar a_\ell<2\norm{z_\ell}_2^2$ implies
\[
    M_\ell b2^{-2j_\ell}
    >\frac{C_{\rm reg}}2rp\,\bar a_\ell.
\]
Consequently,
\eqref{eq:regular-certificate-cost} bounds the probability for the
corresponding profile by
\begin{equation}
\label{eq:regular-large-profile-probability}
    \left(\frac{C'}{C_{\rm reg}}\right)^q
    \le
    \left(\frac{C'}{C_{\rm reg}}\right)^{M_\ell/4}.
\end{equation}
Choose $C_{\rm reg}=C_{\rm reg}(B,c_{\rm size})$ and then
$L_0(B,c_{\rm size})$ sufficiently large.  Since
$M_\ell\ge b\ge(L/10)\log k$, summing
\eqref{eq:regular-large-profile-probability} over the at most
$O(\log k)^4$ parameter choices gives failure probability at most
$k^{-B}$.  Hence every regular $E_\ell$ satisfies
\eqref{eq:regular-set-bound}.
This proves the lemma.
\end{proof}

\begin{lemma}[Uniform exceptional-set estimate]
\label{lem:uniform-exceptional-set-estimate}
For every $B\ge1$ there are constants $L_0=L_0(B,c_{\rm size})$,
$\gamma>1$, a positive integer $h$, and $C=C(B,c_{\rm size})$ with the
following property.  If $L\ge L_0$, then, with probability at least
$1-k^{-B}$, simultaneously for every choice of
$b,x,(E_\ell,j_\ell)_{\ell=1}^s$ and every prescribed row-chunk
decomposition
\[
    E_\ell=\bigsqcup_{u\in H_\ell}S_u(E_\ell),
    \qquad 1\le\ell\le s,
\]
which together form an admissible data structure in the sense of
Definition~\ref{def:fixed-width-admissible-data} with these values of
$\gamma$ and $h$, every exceptional $E_\ell$ satisfies
\begin{equation}
\label{eq:exceptional-set-bound}
    M_\ell b2^{-2j_\ell}
    \le
    C rp\sqrt{\frac{M_\ell}{r}}.
\end{equation}
\end{lemma}

\begin{proof}
If $E_\ell$ is
exceptional, then on a substantial subset $J\subset E_\ell$ the delayed
increment $z_\ell$ defined in \eqref{eq:delayed increment def} is small; consequently $P_{\ell-h}x$ has large
coordinates on $J$. 
The key observation is that geometric growth and the
delay $h$ ensure that
$\dim U_{\ell-h}$ is much smaller than $M_\ell$ (the size of $E_\ell$).
The entropy bound associated to $P_{\ell-h}x$,
is computed with respect to the dimension of $U_{\ell-h}$
and the {\it prefix} size, which is much smaller than $M_\ell$. This
creates substantial savings when balancing entropy and probability estimates. Below, we provide the detailed argument.

We assume that the separation factor $\gamma\ge2$ is sufficiently large,
and the same for the integer $h$, depending on $\gamma$ and $c_{\rm size}$; its precise magnitude
can be extracted from the proof below.

\smallskip

Given the admissible data, set
\[
    M_{\rm pre}:=\sum_{c=1}^{\ell-h}M_c,
\]
and note that $\dim U_{\ell-h}\le M_{\rm pre}$.  By geometric growth,
\[
    M_{\rm pre}
    \le
    \frac{\gamma^{-h}}{1-\gamma^{-1}}M_\ell
    \le
    2\gamma^{-h}M_\ell.
\]
We assume that $h$ is so large that
\begin{equation}
\label{eq:prefix-small-fraction}
    2\gamma^{-h}\le\varepsilon_0,
    \qquad
    \varepsilon_0:=\min\left\{10^{-3},\frac1{2c_{\rm size}}\right\}.
\end{equation}
In particular, since
$M_\ell\le c_{\rm size}r$, also $M_{\rm pre}\le r/2$.

\smallskip
\noindent\emph{Construction of the exceptional certificate in the prefix
space.}
If $E_\ell$ is exceptional, at least half
of its chunks contain at least $b/2$ coordinates satisfying
$|(z_\ell)_i|<2^{-j_\ell-1}$.  Select exactly
$\lceil m_\ell/2\rceil$ such chunks and $b/2$ such coordinates from each,
and denote their union by $J$.  Again
\[
    \frac{M_\ell}{4}\le q:=\abs J\le M_\ell,
    \qquad
    R_{b/2}(\eta,J)\ge q.
\]
For $i\in E_\ell$, the vector $P_Ve_i$ belongs to $U_\ell$, and hence
\[
    x_i=(P_\ell x)_i
       =(P_{\ell-h}x)_i+(z_\ell)_i.
\]
Since $|x_i|>2^{-j_\ell}$ on $E_\ell$, every $i\in J$ satisfies
\begin{equation}
\label{eq:exceptional-prefix-witness}
    |(P_{\ell-h}x)_i|\ge2^{-j_\ell-1}.
\end{equation}

If $\ell\le h$, then $P_{\ell-h}=0$, so $E_\ell$ cannot be exceptional.
Assume $\ell>h$.  By \eqref{eq:exceptional-prefix-witness},
$P_{\ell-h}x\ne0$, so we may define
\[
    y_\ell
    :=\frac{P_{\ell-h}x}{\norm{P_{\ell-h}x}_2}
    \in U_{\ell-h}\cap S^{n-1}.
\]
Since $\norm{P_{\ell-h}x}_2\le1$,
\eqref{eq:exceptional-prefix-witness} implies
\[
    J\in\mathcal F_{U_{\ell-h}}(2^{-j_\ell-1},q),
    \qquad
    R_{b/2}(\eta,J)\ge q.
\]
Thus the set $J$ is the {\it exceptional certificate} constructed in this
step, and $y_\ell$ is its large-coordinate witness in the prefix space $U_{\ell-h}$.
Similarly to the previous lemma, our goal is to show that
the probability of finding an exceptional certificate
{\it not satisfying}
\eqref{eq:exceptional-set-bound}
for some choice of the profile (i.e some collection
of admissible parameter values) is small.

\smallskip
\noindent\emph{Probability-weighted counting of prefix tuples.}
Fix $\ell>h$, the common chunk width $b$, and the numerical prefix data
\[
    (M_c)_{c=1}^{\ell-h}
    \qquad\text{and}\qquad
    (j_c)_{c=1}^{\ell-h}.
\]
We will sum over all possible prefix tuples
$(E_c)_{c=1}^{\ell-h}$ consistent with these data.  The actual extraction
satisfies $R_b(\eta,E_c)\ge M_c$ for every $c$, because each $E_c$ is covered by
its distinct row chunks.  Moreover, for every $1\le c\le\ell-h$,
\eqref{eq:separated-extraction-scale} and
$M_c\le\abs{I_{j_c}(x)}\le2^{2j_c}$ give
\[
    b2^{-2j_c}\ge\frac L{10}p,
    \qquad
    b\ge\frac L{10}pM_c.
\]
After increasing $L_0$, these estimates give
$b\ge2\log(e k)$ and $b\ge e^3pM_c$.  Together with
$M_{\rm pre}\le r/2$ from
\eqref{eq:prefix-small-fraction}, they verify every hypothesis of
Lemma~\ref{lem:weighted-prefix-count} with its parameter $L$ replaced by
$L/10$.  Define
\[
\begin{aligned}
    \mathcal E_{\rm pre}
    :=\bigl\{(F_1,\ldots,F_{\ell-h}):\
        &F_c\subset[n],\ \abs{F_c}=M_c
            &&(c\le\ell-h),\\
        &F_c\cap F_d=\varnothing
            &&(c\ne d),\\
        &\exists\,y\in V\cap S^{n-1}\text{ such that }F_c\subset I_{j_c}(y)
            &&(c\le\ell-h)
      \bigr\}.
\end{aligned}
\]
Thus $\mathcal E_{\rm pre}$ is precisely the prefix family appearing in
Lemma~\ref{lem:weighted-prefix-count}.  For
$\mathbf E=(E_1,\ldots,E_{\ell-h})\in\mathcal E_{\rm pre}$, set
\[
    E_{\rm pre}(\mathbf E):=\bigcup_{c=1}^{\ell-h}E_c
\]
and
\[
    A_{\mathbf E}
    :=\{R_b(\eta,E_c)\ge M_c\text{ for every }c\le\ell-h\}.
\]
After changing the absolute constant, that lemma gives
\begin{equation}
\label{eq:exceptional-prefix-cost}
    \sum_{\mathbf E\in\mathcal E_{\rm pre}}\mathbb P(A_{\mathbf E})
    \le
    \left(C\frac r{L M_{\rm pre}}\right)^{M_{\rm pre}}.
\end{equation}

\smallskip
\noindent\emph{The fixed-profile probability bound.}
Fix $\mathbf E\in\mathcal E_{\rm pre}$ and put
\[
    W_{\mathbf E}
    :=\operatorname{span}\{P_Ve_i:i\in E_{\rm pre}(\mathbf E)\}.
\]
The family $\mathcal E_{\rm pre}$ is defined solely in terms of the fixed
subspace $V$ and the fixed numerical prefix data.  Hence, once
$\mathbf E$ is fixed, both $E_{\rm pre}(\mathbf E)$ and $W_{\mathbf E}$ are
deterministic; in particular, they do not depend on the random mask $\eta$.
Further,
$\dim W_{\mathbf E}\le M_{\rm pre}$.  The possible target certificates
associated with the fixed prefix tuple are contained in
\[
    \mathcal C_{\mathbf E}
    :=\left\{J\in
        \mathcal F_{W_{\mathbf E}}(2^{-j_\ell-1},q):
        J\cap E_{\rm pre}(\mathbf E)=\varnothing
      \right\}.
\]
The scale condition, $q\le M_\ell\le2^{2j_\ell}$, and a sufficiently large
$L_0$ verify the hypotheses of
Corollary~\ref{lem:entropy-support-certificate} with $\kappa=b/2$.  Since
$M_{\rm pre}\le q/10$,
Corollary~\ref{lem:entropy-support-certificate} gives
\begin{equation}
\label{eq:exceptional-certificate-cost}
    \sum_{J\in\mathcal C_{\mathbf E}}
    \mathbb P\{R_{b/2}(\eta,J)\ge q\}
    \le
    \left(C\frac{p2^{2j_\ell}}b\right)^q.
\end{equation}
For $J\in\mathcal C_{\mathbf E}$, the events $A_{\mathbf E}$ and
$\{R_{b/2}(\eta,J)\ge q\}$ are increasing functions of disjoint column
blocks.  Therefore \eqref{eq:negative-association-block-product} and a union
bound give
\[
\begin{aligned}
    \mathbb P\left(
        A_{\mathbf E}\cap
        \bigcup_{J\in\mathcal C_{\mathbf E}}
        \{R_{b/2}(\eta,J)\ge q\}
      \right) \le
      \mathbb P(A_{\mathbf E})
      \sum_{J\in\mathcal C_{\mathbf E}}
      \mathbb P\{R_{b/2}(\eta,J)\ge q\}.
\end{aligned}
\]
Suppose that
\begin{equation}
\label{eq:exceptional-large-contribution}
    M_\ell b2^{-2j_\ell}
    >C_{\rm exc}rp\sqrt{\frac{M_\ell}{r}}
\end{equation}
for a sufficiently large constant $C_{\rm exc}$.  Then
\[
    \frac{p2^{2j_\ell}}b
    <
    \frac1{C_{\rm exc}}\sqrt{\frac{M_\ell}{r}}.
\]
Summing the preceding probability estimate over
$\mathbf E\in\mathcal E_{\rm pre}$ and using
\eqref{eq:exceptional-prefix-cost} and
\eqref{eq:exceptional-certificate-cost}, we find that the probability that
there exist $\mathbf E\in\mathcal E_{\rm pre}$ and
$J\in\mathcal C_{\mathbf E}$ such that both $A_{\mathbf E}$ and
$\{R_{b/2}(\eta,J)\ge q\}$ occur is at most
\[
    \left(C\frac r{L M_{\rm pre}}\right)^{M_{\rm pre}}
    \left(
        \frac C{C_{\rm exc}}\sqrt{\frac{M_\ell}{r}}
    \right)^{M_\ell/4}.
\]
Consequently, the logarithm of this probability bound, divided by
$M_\ell$, is at most
\[
    \frac {M_{\rm pre}}{M_\ell}
    \log\left(
        \frac{C(r/M_\ell)}{L(M_{\rm pre}/M_\ell)}
    \right)
    -\frac14\log\left(
        \frac{C_{\rm exc}\sqrt{r/M_\ell}}{C}
    \right).
\]
Fix a number $D>0$ whose value will be chosen in terms of $B$. The last exponent equals
\[
    \left(\frac {M_{\rm pre}}{M_\ell}-\frac18\right)
       \log\left(\frac r{M_\ell}\right)
    +\frac {M_{\rm pre}}{M_\ell}
       \log\left(\frac{C}{L(M_{\rm pre}/M_\ell)}\right)
    -\frac14\log\left(\frac {C_{\rm exc}}C\right).
\]
Recall that, by our assumptions on parameters, $M_{\rm pre}\leq \varepsilon_0 M_\ell$,
$L\geq 1$, and $M_\ell\leq c_{\rm size}\,r$.
An inspection of the expression above then show that 
there is large enough
$C_{\rm exc}=C_{\rm exc}(D,c_{\rm size})$ so that
the above expression is less than $-D$.
It follows that the probability of an exceptional certificate satisfying
\eqref{eq:exceptional-large-contribution}, for a fixed parameter profile, is
at most
\begin{equation}
\label{eq:exceptional-fixed-profile-probability}
    \exp(-DM_\ell).
\end{equation}

\smallskip
\noindent\emph{Union bound over profiles.}
We now sum over the parameter profiles.  The complete discrete profile is
\[
    \left(
        b,M_\ell,\ell,
        (M_c)_{c=1}^{\ell-h},
        (j_c)_{c=1}^{\ell-h},
        j_\ell
    \right).
\]
For a target of size $M_\ell$,
geometric growth gives
\begin{equation}
\label{eq:exceptional-profile-length}
    \ell\le1+C\log(M_\ell/b).
\end{equation}
Further, there are $O(\log k)$ choices for each dyadic cardinality.  The scale
condition and $b\le M_\ell\le c_{\rm size}r$ give
$2^{2j_c}\le10c_{\rm size}r/(Lp)$ for every relevant level, so there are
$O_{c_{\rm size}}(\log k)$ choices for each $j_c$.  By
\eqref{eq:exceptional-profile-length}, the logarithm of the number of
choices of the prefix length, its cardinality and level profiles, and the
target level, for fixed $b$ and $M_\ell$, is at most
\begin{equation}
\label{eq:exceptional-profile-count}
    C\bigl(1+\log(M_\ell/b)\bigr)\log\log(e k).
\end{equation}
Thus
\eqref{eq:exceptional-profile-count} is absorbed by half of the exponent in
\eqref{eq:exceptional-fixed-profile-probability}.

After summing over the remaining dyadic choices of $b$ and $M_\ell$, the
total exceptional failure probability is bounded by
\[
    C(\log k)^2
    \sum_{\substack{(L/10)\log k\le M\le c_{\rm size}r:\\
        M\ {\rm dyadic}}}
    \exp\left(-\frac D2M\right)
    \le k^{-B},
\]
where the sum is over dyadic $M\le c_{\rm size}r$,
and where we choose $D$ sufficiently large depending on $B$.
This proves the lemma.
\end{proof}

\begin{proof}[Proof of Proposition~\ref{prop:separated-fixed-width-estimate}]
Put $B_0:=B+2$.  Let $L_{\rm deg}$ and
$C_0=C_0(\alpha,B_0)$ be, respectively, the threshold and the constant
supplied by Lemma~\ref{lem:total-degree-canonical-heavy-rows} with failure
exponent $B_0$, and set $c_{\rm size}:=C_0$.  For this value of
$c_{\rm size}$ and the same failure exponent, let $L_{\rm reg}$ be the
threshold supplied by Lemma~\ref{lem:uniform-regular-set-estimate}, and let
$L_{\rm exc},\gamma,h$ be the threshold and parameters supplied by
Lemma~\ref{lem:uniform-exceptional-set-estimate}.  Set
\[
    L_0:=\max\{10L_{\rm deg},L_{\rm reg},L_{\rm exc}\}.
\]

Now fix $L\ge L_0$.  Apply the total-degree lemma with threshold parameter
$L/10$, and apply the two auxiliary lemmas with parameter $L$.  The
intersection of the resulting three events has probability
at least
\[
    1-3k^{-B_0}\ge1-k^{-B},
\]
where we used $k\ge3$.  Work on this intersection event.  In particular,
the total-degree lemma gives, simultaneously for every
$x\in{\mathcal N}_V$,
\[
    \sum_{j\in\mathbb Z}\sum_{u=1}^k
    d_u(\eta,I_j(x))
    \mathbf 1_{\{d_u(\eta,I_j(x))\ge K_j^{(L/10)}(x)\}}
    \le C_0r.
\]

Fix an admissible separated extraction.  Since $x$ is a unit vector,
\[
    \abs{I_{j_\ell}(x)}\le2^{2j_\ell}.
\]
The scale condition gives
\[
    b
    \ge
    \frac L{10}\max\{pr_*,p2^{2j_\ell}\}
    \ge
    K_{j_\ell}^{(L/10)}(x).
\]
Every prescribed row chunk has cardinality $b$ and is contained in
$I_{j_\ell}(x)$ and in the support of its assigned row.  Therefore,
\[
\begin{aligned}
    M_\ell=bm_\ell
    \le
    \sum_{u\in H_\ell}
    d_u(\eta,I_{j_\ell}(x))
    \mathbf 1_{\{d_u(\eta,I_{j_\ell}(x))
        \ge K_{j_\ell}^{(L/10)}(x)\}}
    \le C_0r.
\end{aligned}
\]
Thus the auxiliary linear-size condition
\eqref{eq:linear-size-assumption} holds automatically.

Form $U_\ell$, $P_\ell$, and $z_\ell$ as above.  The projection increments
$(P_q-P_{q-1})x$ are pairwise
orthogonal.  Since each increment occurs in at most $h$ of the vectors
$z_\ell$,
\begin{equation}
\label{eq:delayed-window-energy}
    \sum_{\ell=1}^s\norm{z_\ell}_2^2\le h.
\end{equation}
Summing the regular bounds on this intersection and using
\eqref{eq:delayed-window-energy} gives
\[
    \sum_{\substack{1\le\ell\le s:\\ E_\ell\ \mathrm{regular}}}
    M_\ell b2^{-2j_\ell}
    \le C(\alpha,B)rp.
\]
For the exceptional sets, geometric growth and
$M_s\le c_{\rm size}r$ give
\[
    \sum_{\substack{1\le\ell\le s:\\ E_\ell\ \mathrm{exceptional}}}
    \sqrt{\frac{M_\ell}{r}}
    \le
    \sum_{\ell=1}^s\sqrt{\frac{M_\ell}{r}}
    \le C\sqrt{\frac{M_s}{r}}
    \le C\sqrt{c_{\rm size}}.
\]
Hence \eqref{eq:exceptional-set-bound} yields the same bound for their
total contribution.  This proves \eqref{eq:separated-fixed-width-conclusion}.
\end{proof}

\subsection{Proof of the fixed-width target}

\begin{proof}[Proof of Proposition~\ref{thm:target-bound}]
Work on the event in
Proposition~\ref{prop:separated-fixed-width-estimate}, with failure
parameter $B+2$, and on the event in
Proposition~\ref{thm:flat-extraction}, also with failure parameter $B+2$.
On the latter event, Proposition~\ref{prop:equal-size-leaf-grouping}
produces the selected partition $\mathfrak T_L(\eta,x)$ and guarantees
\eqref{eq:leaf-scale-lower-bound}.  Since $k\ge3$, the intersection has
probability at least $1-2k^{-(B+2)}\ge1-k^{-B}$.  Fix
$x\in{\mathcal N}_V$ and a positive dyadic integer $b$, and put
$\mathfrak T:=\mathfrak T_L(\eta,x)$.
Order the leaves with $b(E)=b$ by increasing $j(E)$.  Their weights then
form a strictly decreasing dyadic sequence.  Retain the first leaf and,
subsequently, retain a leaf precisely when its value of $m(E)$ is larger
than all previously retained values.  Call these the record leaves.

Between a record leaf $E$ and the next record leaf, every discarded leaf
$F$ has $m(F)\le m(E)$, while the sum of the dyadically decreasing weights
in that block is at most $2w(E)$.  It follows pointwise that
\begin{equation}
\label{eq:record-leaf-domination}
    Y_t^{(b)}(\mathfrak T)
    \le
    2\sum_{\substack{E\text{ record leaf of }\mathfrak T:\\ b(E)=b}}
    w(E)b\mathbf 1_{\{t\le m(E)\}}.
\end{equation}
The record values of $m(E)$ are strictly increasing positive dyadic
integers, and hence increase by a factor at least two.

Choose a fixed integer $h_0$ such that $2^{h_0}\ge\gamma$, where $\gamma$ is
the separation constant from
Proposition~\ref{prop:separated-fixed-width-estimate}.  Split the record
leaves into $h_0$ classes according to their positions modulo $h_0$.
Within every class the cardinalities $b\,m(E)$ grow by a factor at least
$\gamma$, while the level indices remain increasing.
The scale condition \eqref{eq:leaf-scale-lower-bound}, the row-chunk
structure in Definition~\ref{def:two-level-partition} show that
Proposition~\ref{prop:separated-fixed-width-estimate} applies to each class.
Therefore
\begin{equation}
\label{eq:record-class-square-sum}
    \sum_{\substack{E\text{ in the fixed record class}}}
    m(E)w(E)^2b^2
    \le Cpr.
\end{equation}

Fix one record class and regard its leaves, together with their roots, as a
two-level subforest $\mathfrak T_{\rm cl}$ of $\mathfrak T$.  Since the
values $m(E)$ are strictly increasing within the class, there is at most one
leaf for each dyadic value of $m(E)$.  Therefore
\[
    \mathcal Q_{\rm leaf}(\mathfrak T_{\rm cl})^2
    =
    \sum_{\substack{E\text{ in the fixed record class}}}
    m(E)w(E)^2b^2
    \le Cpr
\]
by \eqref{eq:record-class-square-sum}.  The upper comparison in the
statement of Lemma~\ref{lem:dyadic-leaf-square-function-comparison} now gives
\[
    \left\|
        \sum_{\substack{E\text{ in the fixed record class}}}
        w(E)b\mathbf 1_{\{t\le m(E)\}}
    \right\|_2
    \le C\sqrt{pr}.
\]
There are only the fixed number $h_0$ of classes.  Thus
\eqref{eq:record-leaf-domination} and the triangle inequality imply
\[
    \norm{Y^{(b)}(\mathfrak T)}_2\le C\sqrt{pr}.
\]
The uniformity in
Proposition~\ref{prop:separated-fixed-width-estimate} applies to every
vector in ${\mathcal N}_V$ and every admissible extraction, so this
conclusion holds simultaneously over $x\in{\mathcal N}_V$ and over dyadic
$b$.
This proves the proposition.
\end{proof}

\section{Completion of the main results}
\label{sec:proofs-main-results}

We now assemble the Tall estimate from Section~\ref{sec:ks-tall-argument}
and the Flat estimates from Sections~\ref{sec:flat-majorants}--
\ref{sec:fixed-width-estimates}.  We first prove the symmetric-entry estimate
and pass to centered entries by symmetrization.  We then verify the concrete
support models and record the hybrid and model-specific consequences.

\subsection{The symmetric-entry estimate}

\begin{theorem}[Main theorem]
\label{thm:main-ose-bound}
For every $B\ge1$ there is a constant
$C_{\text{\tiny\ref*{thm:main-ose-bound}}}=C(B)$ with the following
property.  Let
\[
    n\ge k\ge r\ge3,
    \qquad
    n\le r^{10},
    \qquad
    k\ge r\log^2\left(\frac{en}{r}\right),
    \qquad
    \frac{\log k}{k}\le p\le1.
\]
Let $\Pi=(\pi_{ui})$ follow the admissible sparse-entry model with parameter $p$
and a symmetric entry variable $\xi$ satisfying $|\xi|\le1$ almost surely.
Then, for every non-random $r$-dimensional subspace $V\subset\R^n$,
\[
    \mathbb P\left\{
        \norm{\Pi U_V}
        >
        C_{\text{\tiny\ref*{thm:main-ose-bound}}}\sqrt{kp}
    \right\}
    \le
    k^{-B}.
\]
\end{theorem}

\begin{proof}
Let ${\mathcal N}_V\subset V\cap S^{n-1}$ be a deterministic $1/2$-net with
$\abs{{\mathcal N}_V}\le5^r$, and fix a sufficiently large constant $L$, to be
chosen in terms of $B$.  We apply all estimates below with their failure
exponents increased by a fixed amount.  Put
\[
    r_*:=\max\left\{r,\frac{\log k}{p}\right\}.
\]

The parameter assumptions used in the preceding sections are satisfied:
$n\le r^{10}$ and $p\ge(\log k)/k$.  The exponent $10$ is not structural;
the restriction $n\le r^{10}$ is used to absorb polynomial
ambient-dimension factors in the union bounds and in the symmetrization step
below, and any fixed polynomial relation would only change the constants.
The stronger condition $k\ge r\log^2(en/r)$ is used only at the final
aggregation over the $O(\log(en/r))$ possible leaf widths, in
\eqref{eq:main-theorem-flat-bound}.

By Proposition~\ref{prop:ks-tall-contribution-on-net}, with probability at
least $1-k^{-(B+5)}$,
\begin{equation}
\label{eq:main-theorem-tall-bound}
    \sup_{z\in{\mathcal N}_V}
    \norm{\mathcal T_L^{\rm can}(\Pi,z)z}_2
    \le C(B,L)\sqrt{kp}.
\end{equation}
We next estimate the canonical Flat majorant.  Work on the events in
Corollary~\ref{cor:canonical-flat-two-level-partition},
Proposition~\ref{thm:target-bound}, and
Lemma~\ref{lem:row-degree-upper-bound}.  Their intersection has
probability at least $1-Ck^{-(B+5)}$.  Fix $z\in{\mathcal N}_V$, and consider
the selected partition $\mathfrak T_L(\eta,z)$ supplied by the corollary.

Every leaf width satisfies, by \eqref{eq:leaf-scale-lower-bound},
\[
    b(E)\ge \frac L{10}pr_*.
\]
If $u\in H(E)$ is one of its supporting rows, then
$b(E)\le d_u(\eta,[n])$.  The row-degree upper-bound event therefore gives
\[
    b(E)\le C(B)\max\{\log(ek),pn\}
    \le C(B)pn,
\]
where we used $pr_*\ge\log k$ and $r_*\le k\le n$.
Since the widths are positive dyadic integers, their number is bounded by
\[
    \#\{b(E):E\text{ leaf of }\mathfrak T_L(\eta,z)\}
    \le C(B,L)\log\left(\frac{en}{r_*}\right)
    \le C(B,L)\log\left(\frac{en}{r}\right).
\]

Proposition~\ref{thm:target-bound} gives, for every positive dyadic width
$b$,
\[
    \norm{Y^{(b)}(\mathfrak T_L(\eta,z))}_2
    \le C(B)\sqrt{pr}.
\]
Define
\[
    Y_t:=
    \sum_{\substack{E\text{ leaf of }\mathfrak T_L(\eta,z)}}
    w(E)b(E)\mathbf 1_{\{t\le m(E)\}},
    \qquad 1\le t\le k.
\]
Then
$Y=\sum_{\substack{b\ge1\\b\ {\rm dyadic}}}
Y^{(b)}(\mathfrak T_L(\eta,z))$.  Hence the triangle inequality, followed by
Lemma~\ref{lem:dyadic-leaf-square-function-comparison}, gives
\[
\begin{aligned}
    \mathcal Q_{\rm leaf}(\mathfrak T_L(\eta,z))
    &\le \norm Y_2
     \le
     \sum_{\substack{b\ge1\\b\ {\rm dyadic}}}
     \norm{Y^{(b)}(\mathfrak T_L(\eta,z))}_2 \\
    &\le C(B,L)\log\left(\frac{en}{r}\right)\sqrt{pr}.
\end{aligned}
\]
Corollary~\ref{cor:canonical-flat-two-level-partition} now yields,
simultaneously for every $z\in{\mathcal N}_V$,
\begin{equation}
\label{eq:main-theorem-flat-bound}
\begin{aligned}
    \FM_L^{\rm can}(z)
    &\le
    C(B,L)\left[
        \log\left(\frac{en}{r}\right)\sqrt{pr}
        +\sqrt{pr}
    \right] \\
    &\le C(B,L)\sqrt{kp},
\end{aligned}
\end{equation}
where the second inequality follows from
$k\ge r\log^2(en/r)$.

Finally,
$\|\Pi U_V\|=\sup_{x\in V\cap S^{n-1}}\|\Pi x\|_2$, so the standard
norming-net estimate gives
\[
    \norm{\Pi U_V}
    \le
    2\sup_{z\in{\mathcal N}_V}\norm{\Pi z}_2.
\]
For each $z\in{\mathcal N}_V$, the canonical Tall--Flat decomposition
\eqref{eq:technical-overview-canonical-tall-flat} and the Flat-majorant
estimate \eqref{eq:technical-overview-flat-contribution-bound} give
\[
    \norm{\Pi z}_2
    \le
    \norm{\mathcal T_L^{\rm can}(\Pi,z)z}_2
    +2\FM_L^{\rm can}(z).
\]
Consequently,
\[
    \norm{\Pi U_V}
    \le
    2\sup_{z\in{\mathcal N}_V}
        \norm{\mathcal T_L^{\rm can}(\Pi,z)z}_2
    +4\sup_{z\in{\mathcal N}_V}\FM_L^{\rm can}(z).
\]
Combining \eqref{eq:main-theorem-tall-bound} and
\eqref{eq:main-theorem-flat-bound} proves the asserted norm estimate.  A
union bound over the finitely many good events, with the increased failure
exponents $B+5$ used in
\eqref{eq:main-theorem-tall-bound} and
\eqref{eq:main-theorem-flat-bound} makes the total failure probability at most
$k^{-B}$.
\end{proof}

\subsection{Centered entry variables}

The main result stated in the introduction permits a centered, not necessarily
symmetric, entry variable.  It follows from the symmetric-entry estimate by a
standard symmetrization argument.

\begin{proof}[Proof of Theorem~\ref{thm:centered-entry-extension}]
Write $\pi_{ui}=\eta_{ui}\xi_{ui}$.  On an extension of the probability
space, let $(\xi'_{ui})$ be an independent copy of $(\xi_{ui})$, independent
of the support mask, and define
\[
    \Pi':=(\eta_{ui}\xi'_{ui}).
\]
Thus $\Pi$ and $\Pi'$ share the same support mask. Put
\[
    t:=2C_{\text{\tiny\ref*{thm:main-ose-bound}}}(B+7)\sqrt{kp}.
\]
Theorem~\ref{thm:main-ose-bound}, applied to $(\Pi-\Pi')/2$, gives
\begin{equation}
\label{eq:centered-symmetrized-tail}
    \mathbb P\{\|(\Pi-\Pi')U_V\|>t\}\le k^{-(B+7)}.
\end{equation}
The function
\[
    M\longmapsto (\|MU_V\|-t)_+
\]
is convex.  Conditional Jensen's inequality therefore gives
\[
    (\|\Pi U_V\|-t)_+
    \le
    \mathbb E_{\xi'}\left[
        (\|(\Pi-\Pi')U_V\|-t)_+
        \,\middle|\,\eta,(\xi_{ui})
    \right].
\]
Since $|\xi|\le1$,
\[
    \|(\Pi-\Pi')U_V\|
    \le
    \|\Pi-\Pi'\|_{\rm F}
    \le2\sqrt{kn}.
\]
Taking expectations and using \eqref{eq:centered-symmetrized-tail}, we get
\[
    \mathbb E(\|\Pi U_V\|-t)_+
    \le
    2\sqrt{kn}\,k^{-(B+7)}.
\]
Consequently, Markov's inequality yields
\[
\begin{aligned}
    \mathbb P\{\|\Pi U_V\|>2t\}
    \le
        \frac{2\sqrt{kn}}{t}\,k^{-(B+7)} 
    \le k^{-B}.
\end{aligned}
\]
Here we used $n\le r^{10}\le k^{10}$,
$p\ge(\log k)/k$, and $k\ge3$; the fixed numerical slack in
the exponent $B+7$ absorbs the resulting polynomial factor.  Taking
\[
    C_{\text{\tiny\ref*{thm:centered-entry-extension}}}(B)
    :=4C_{\text{\tiny\ref*{thm:main-ose-bound}}}(B+7)
\]
completes the proof.
\end{proof}

\subsection{Concrete admissible models}

\begin{definition}[I.i.d. sparse-entry model]
\label{def:iid-sparse-entry-model}
Let $0<p\le1$, and let $\xi$ be a real random variable.  A $k\times n$
random matrix $\Pi=(\pi_{ui})$ follows the \emph{i.i.d. sparse-entry model}
with parameters $p$ and $\xi$ if
\[
    \pi_{ui}=b_{ui}\xi_{ui},
    \qquad u\in[k],\quad i\in[n],
\]
where the variables $b_{ui}$ are independent Bernoulli$(p)$ variables, the
variables $\xi_{ui}$ are independent copies of $\xi$, and the two families
are independent.
\end{definition}

\begin{definition}[Fixed-column-degree sparse-entry model]
\label{def:fixed-column-degree-model}
Let $1\le d\le k$ be an integer and put $p=d/k$.  Independently for every
column $i\in[n]$, choose a uniformly random set $S_i\subset[k]$ of
cardinality $d$, and set
\[
    \eta_{ui}:=\mathbf 1_{\{u\in S_i\}},
    \qquad u\in[k],\quad i\in[n].
\]
Let $(\xi_{ui})_{u\in[k],\,i\in[n]}$ be independent copies of a real random
variable $\xi$, independent of the sets $(S_i)_{i\in[n]}$, and define
\[
    \pi_{ui}:=\eta_{ui}\xi_{ui},
    \qquad u\in[k],\quad i\in[n].
\]
We say that $\Pi=(\pi_{ui})$ follows the
\emph{fixed-column-degree sparse-entry model} with parameters $d$ and $\xi$.
Each column has exactly $d$ designated support locations; if
$\mathbb P\{\xi=0\}=0$, these are exactly its nonzero entries.
\end{definition}

\begin{definition}[Unnormalized SparseStack$^{\mathsf T}$ model]
\label{def:transposed-sparsestack-model}
Let $1\le s\le k$ be an integer dividing $k$, and put $p=s/k$.
Partition $[k]$ into sets $B_1,\ldots,B_s$, each of cardinality $k/s$.
Independently for every $\ell\in[s]$ and $i\in[n]$, choose
$U_{\ell i}$ uniformly from $B_\ell$, and set
\[
    \eta_{ui}:=\mathbf 1_{\{u=U_{\ell i}\}},
    \qquad \ell\in[s],\quad i\in[n],\quad u\in B_\ell.
\]
Let $(\xi_{ui})_{u\in[k],\,i\in[n]}$ be independent copies of a real random
variable $\xi$, independent of the choices
$(U_{\ell i})_{\ell\in[s],\,i\in[n]}$, and set
\[
    \pi_{ui}:=\eta_{ui}\xi_{ui},
    \qquad u\in[k],\quad i\in[n].
\]
The resulting matrix $\Pi=(\pi_{ui})$ follows the
\emph{unnormalized SparseStack$^{\mathsf T}$ model} with parameters $s$ and
$\xi$.  Each column has one designated support location in every block and
hence $s$ designated support locations.  For Rademacher $\xi$, this is the
unnormalized transpose of the SparseStack model
\cite{HuangRudelsonTikhomirov26}.
\end{definition}

\begin{proposition}[Basic examples of admissible support laws]
\label{prop:basic-admissible-support-laws}
The i.i.d. sparse-entry model with Bernoulli parameter $p$ follows the
admissible sparse-entry model with density $p$.  The fixed-column-degree
model with column degree $d$ follows the admissible sparse-entry model with
density $p=d/k$.  The unnormalized SparseStack$^{\mathsf T}$ model with $s$
blocks follows the admissible sparse-entry model with density $p=s/k$.
\end{proposition}

\begin{proof}
Independent Bernoulli variables are negatively associated.  For the
fixed-column-degree model, the indicator vector of a uniformly random
$d$-subset of $[k]$ is negatively associated by the standard theorem for
sampling without replacement.  Independent unions of negatively associated
families are negatively associated.  Since the columns are sampled
independently and $\mathbb P\{u\in S_i\}=d/k$, the full support mask is
negatively associated with common marginal $d/k$.

For the SparseStack$^{\mathsf T}$ model, for each pair $(\ell,i)$ the vector
$(\eta_{ui})_{u\in B_\ell}$ is the indicator vector of a uniformly chosen
one-element subset of $B_\ell$, and hence is negatively associated.  These
vectors are independent over $(\ell,i)\in[s]\times[n]$, so their union is
negatively associated.  Finally,
$\mathbb E\eta_{ui}=1/\abs{B_\ell}=s/k$ for $u\in B_\ell$.
\end{proof}

Combining the proposition with
Theorem~\ref{thm:centered-entry-extension} gives the following three direct
specializations.

\begin{corollary}[Concrete sparse models]
\label{cor:concrete-sparse-models}
Assume the dimensional hypotheses of
Theorem~\ref{thm:centered-entry-extension}, let
$V\subset\mathbb R^n$ be a fixed $r$-dimensional subspace, and let $\xi$ be
centered with $|\xi|\le1$ almost surely.  For every $B\ge1$, the following
statements hold.
\begin{enumerate}
\item If $(\log k)/k\le p\le1$ and $\Pi$ follows the i.i.d.
      sparse-entry model of Definition~\ref{def:iid-sparse-entry-model}
      with parameters $p$ and $\xi$, then
      \[
          \mathbb P\left\{\|\Pi U_V\|>C(B)\sqrt{kp}\right\}
          \le k^{-B}.
      \]
\item If $\lceil\log k\rceil\le d\le k$ and $\Pi$ follows the
      fixed-column-degree model of
      Definition~\ref{def:fixed-column-degree-model} with parameters $d$
      and $\xi$, then
      \[
          \mathbb P\left\{\|\Pi U_V\|>C(B)\sqrt d\right\}
          \le k^{-B}.
      \]
\item If $s\mid k$, $\lceil\log k\rceil\le s\le k$, and $\Pi$ follows the
      unnormalized SparseStack$^{\mathsf T}$ model of
      Definition~\ref{def:transposed-sparsestack-model} with parameters $s$
      and $\xi$, then
      \[
          \mathbb P\left\{\|\Pi U_V\|>C(B)\sqrt s\right\}
          \le k^{-B}.
      \]
\end{enumerate}
\end{corollary}

\begin{proof}
Proposition~\ref{prop:basic-admissible-support-laws} gives admissible support
density $p$ in the first case, $p=d/k$ in the second, and $p=s/k$ in the
third.  Each density is at least $(\log k)/k$, so
Theorem~\ref{thm:centered-entry-extension} gives the three conclusions.
\end{proof}

\subsection{The hybrid leverage-score argument}

We next remove the ambient-dimension restriction for the concrete models by
splitting the subspace according to its coordinate leverage scores.
The next proposition is a direct application of \cite{BrailovskayaVanHandel24}:

\begin{proposition}
\label{prop:low-leverage-universality}
For every $A\ge1$ and $B\ge1$ there are constants
$c_0=c_0(A,B)$ and $C=C(A,B)$ with the following property.  Let
$k\ge r\ge3$, let $p\ge(\log k)/k$, let
$U:\mathbb R^r\to\mathbb R^n$ be an isometry, and denote by
$U^*:\mathbb R^n\to\mathbb R^r$ its adjoint (equivalently, its transpose).
Set
\[
    q:=\left\lceil c_0\log(e k)\right\rceil,
    \qquad
    \tau:=q^{-4}.
\]
Let $I\subset[n]$ satisfy $\norm{U^*e_i}_2^2\le\tau$ for every $i\in I$.
Suppose that $(g_i)_{i\in I}$ are independent centered random vectors in
$\mathbb R^k$ such that, for some $0\le v\le p$,
\[
    \mathbb E g_i g_i^*=vI_k,
    \qquad
    \left(\mathbb E\norm{g_i}_2^{2q}\right)^{1/(2q)}
    \le A\sqrt{kp}
    \quad (i\in I).
\]
Then
\[
    \mathbb P\left\{
        \left\|\sum_{i\in I}g_i(U^*e_i)^*\right\|>C\sqrt{kp}
    \right\}
    \le k^{-B}.
\]
\end{proposition}

\begin{proof}
For a matrix $Y\in\mathbb R^{k\times r}$, write
\[
    \operatorname{dil}(Y)
    :=
    \begin{pmatrix}
        0&Y\\
        Y^*&0
    \end{pmatrix}
    \in\mathbb R^{(k+r)\times(k+r)}.
\]
Put
\[
    Y_I:=\sum_{i\in I}g_i(U^*e_i)^*,
    \qquad
    X:=\operatorname{dil}(Y_I)=\sum_{i\in I}Z_i,
\]
where
\[
    Z_i:=
    \begin{pmatrix}
        0&g_i(U^*e_i)^*\\
        U^*e_i\,g_i^*&0
    \end{pmatrix}.
\]
Applying \cite[Theorem~2.9]{BrailovskayaVanHandel24} with
its moment parameters $p_{\rm BvH}=q$ and $q_{\rm BvH}=2q$, we obtain
\begin{equation}
\label{eq:low-leverage-bvh-moment-comparison}
    \left|
    \left(\mathbb E\operatorname{tr}|X|^{2q}\right)^{1/(2q)}
    -
    \left(\mathbb E\operatorname{tr}|X_G|^{2q}\right)^{1/(2q)}
    \right|
    \le Cq^2R_{2q}^{\rm BvH}(X).
\end{equation}
Here $X_G$ is the centered Gaussian matrix with the same covariance as $X$,
and
\[
    R_{2q}^{\rm BvH}(X)
    :=\left(\sum_{i\in I}
        \mathbb E\operatorname{tr}|Z_i|^{2q}\right)^{1/(2q)}.
\]
Writing $\ell_i=\norm{U^*e_i}_2^2$ and using
$\operatorname{tr}|Z_i|^{2q}
=2\norm{g_i}_2^{2q}\ell_i^q$ and $\sum_i\ell_i=r$, we obtain
\begin{align*}
    R_{2q}^{\rm BvH}(X)
    &\le 2^{1/(2q)}A\sqrt{kp}
        \left(\sum_{i\in I}\ell_i^q\right)^{1/(2q)} \\
    &\le 2^{1/(2q)}A\sqrt{kp\tau}
        \left(\frac r\tau\right)^{1/(2q)}
     \le C A\sqrt{kp\tau}.
\end{align*}
The last step uses $r/\tau\le kq^4$ and
$q\ge c_0\log(e k)$.

We next identify the Gaussian comparator required by the cited theorem.  It is
\[
    X_G=\operatorname{dil}\bigl(\sqrt v\,GP_IU\bigr),
\]
where $G$ is a $k\times n$ matrix with independent standard Gaussian entries
and $P_I$ is the coordinate projection onto $\mathbb R^I$.  Indeed, if
$Y_I=(Y_{ab})$, independence and centering of the $g_i$ give
\[
    \mathbb E Y_{ab}Y_{cd}
    =v\,\mathbf 1_{\{a=c\}}(U^*P_IU)_{bd},
\]
while independence of the entries of $G$ gives
\[
    \mathbb E
    (\sqrt v\,GP_IU)_{ab}(\sqrt v\,GP_IU)_{cd}
    =
    v\,\mathbf 1_{\{a=c\}}
    \sum_{i\in I}U_{ib}U_{id}
    =
    v\,\mathbf 1_{\{a=c\}}(U^*P_IU)_{bd}.
\]
Thus $Y_I$ and $\sqrt v\,GP_IU$ have the same mean and entrywise covariance,
and their Hermitizations do as well.

Put $B_I:=P_IU$.  Since $U$ is an isometry,
\[
    \norm{B_I}\le1,
    \qquad
    \norm{B_I}_{\rm F}^2
    =\operatorname{tr}(U^*P_IU)
    \le r.
\]
Then \cite[Proposition~10.1]{HalkoMartinssonTropp11} (see also \cite{Gordon85})
gives
\[
    \mathbb E\norm{GB_I}
    \le
    \sqrt{k}\,\norm{B_I}+\norm{B_I}_{\rm F}
    \le \sqrt{k}+\sqrt r.
\]
Moreover, the map $G\mapsto\norm{GB_I}$ is
$\norm{B_I}$-Lipschitz with respect to the Frobenius norm, since
\[
    \bigl|\norm{GB_I}-\norm{G'B_I}\bigr|
    \le\norm{(G-G')B_I}
    \le\norm{B_I}\norm{G-G'}_{\rm F}.
\]
The Gaussian concentration inequality for Lipschitz functions
(see, for example, \cite[Proposition~10.3]{HalkoMartinssonTropp11}), followed by integration of
its tail, therefore yields
\[
    \left(\mathbb E\norm{GB_I}^{2q}\right)^{1/(2q)}
    \le
    \mathbb E\norm{GB_I}+C\norm{B_I}\sqrt q
    \le C(\sqrt k+\sqrt r+\sqrt q).
\]
Since $\operatorname{tr}|M|^{2q}\le(k+r)\norm M^{2q}$ for every
$(k+r)\times(k+r)$ matrix $M$, and
$\norm{\operatorname{dil}(GB_I)}=\norm{GB_I}$, we conclude that
\[
    \left(\mathbb E\operatorname{tr}|X_G|^{2q}\right)^{1/(2q)}
    \le
    (k+r)^{1/(2q)}
    \sqrt v\left(\mathbb E\norm{GB_I}^{2q}\right)^{1/(2q)}
    \le C\sqrt p\,(\sqrt k+\sqrt r+\sqrt q)
    \le C\sqrt{kp}.
\]
Combining this estimate and the bound on $R_{2q}^{\rm BvH}(X)$ with
\eqref{eq:low-leverage-bvh-moment-comparison} gives
\[
    \left(\mathbb E\operatorname{tr}|X|^{2q}\right)^{1/(2q)}
    \le C(A)\sqrt{kp},
\]
because $q^2\sqrt\tau=1$.  Since
$\norm{Y_I}^{2q}=\norm X^{2q}\le\operatorname{tr}|X|^{2q}$,
Markov's inequality, after choosing
$c_0=c_0(A,B)$ and increasing $C=C(A,B)$, proves the claim.
\end{proof}

\begin{proof}[Proof of Corollary~\ref{rem:hybrid-leverage-score-argument}]
We prove the i.i.d. assertion first, using failure exponent $12$.  Let
$U:\mathbb R^r\to\mathbb R^n$ be an isometric embedding with range $V$, and
write
\[
    \ell_i:=\norm{U^*e_i}_2^2,
    \qquad
    \sum_{i=1}^n\ell_i=r.
\]
Choose the sufficiently large absolute constant $c_0$ supplied by
Proposition~\ref{prop:low-leverage-universality} for an absolute moment
constant and failure exponent $12$, and set
\[
    q:=\left\lceil c_0\log(ek)\right\rceil,
    \qquad
    \tau:=q^{-4},
\]
and split the coordinates into
\[
    H:=\{i:\ell_i>\tau\},
    \qquad
    H^{\mathsf c}:=\{i:\ell_i\le\tau\}.
\]
Then $\abs H\le r/\tau=rq^4$.

We first treat the large-leverage coordinates.  Put
\[
    N:=\max\{k,r,\abs H\},
    \qquad
    R:=\max\left\{r,\left\lceil N^{1/10}\right\rceil\right\}.
\]
If $N>n$, extend $U$ to an isometry into $\mathbb R^N$ by adjoining zero
rows, and extend $\Pi$ by adjoining independent columns with the same law as
its original columns.  If $N\le n$, no extension is needed.  In either case,
we may choose a coordinate set $S$ of cardinality $N$ containing $H$; the
restriction $\Pi_S$ follows the same concrete sparse model as $\Pi$.

We record the parameter verification needed to apply
Theorem~\ref{thm:centered-entry-extension}.  After increasing the universal
constant in the assumption of the corollary, one has
\begin{equation}
\label{eq:hybrid-padded-parameters}
    N\ge k\ge R\ge r,
    \qquad
    N\le R^{10},
    \qquad
    k\ge R\log^2\left(\frac{eN}{R}\right).
\end{equation}
Indeed, $k\ge r$ and
\[
    N\le\max\{k,rq^4\}\le kq^4.
\]
If $R=r$, then
\[
    \log\left(\frac{eN}{r}\right)
    \le
    C\left[
        \log\left(\frac{ek}{r}\right)+\log\log(er)
    \right].
\]
The assumption
$k/r\ge C(\log\log r)^2$ gives
$k/r\ge C_1\log^2\log(er)$; further,
$k/r\ge C_1\log^2(ek/r)$.  Thus
$k/r\ge\log^2(eN/r)$ after adjusting constants.  If $R>r$, then
$R\le2N^{1/10}$ and hence
\[
\begin{aligned}
    R\log^2\left(\frac{eN}{R}\right)
    \le
    C(kq^4)^{1/10}\log^2(ekq^4) 
    \le k,
\end{aligned}
\]
and \eqref{eq:hybrid-padded-parameters} is certified.
Let $W_H:=\operatorname{range}(P_HU)$.  Identify $\mathbb R^S$ with its
coordinate subspace in the ambient space.  Since
$\dim W_H\le r\le R\le N$, enlarge $W_H$ inside $\mathbb R^S$ to an
$R$-dimensional subspace $\widetilde W_H$, and let
$U_{\widetilde W_H}:\mathbb R^R\to\mathbb R^S$ be an isometry onto it.
Then
\[
    \norm{\Pi P_HU}
        \le
    \norm{\Pi_SU_{\widetilde W_H}}.
\]
Theorem~\ref{thm:centered-entry-extension}, with dimension parameter $R$ and
failure exponent $12$, applies by
\eqref{eq:hybrid-padded-parameters} and gives
\begin{equation}
\label{eq:hybrid-large-leverage-bound}
    \norm{\Pi P_HU}\le C\sqrt{kp}
\end{equation}
outside an event of probability at most $k^{-12}$.

It remains to control the small-leverage coordinates.  Set
\[
    Y:=\Pi P_{H^{\mathsf c}}U
      =\sum_{i\in H^{\mathsf c}}\Pi_{\cdot\, i}(U^*e_i)^*.
\]
The columns $\Pi_{\cdot\, i}$ are independent and centered, and their
covariance matrices equal $vI_k$, where
$v=p\mathbb E\xi^2\le p$.  Moreover,
\begin{equation}
\label{eq:hybrid-column-moment}
    \left(\mathbb E\norm{\Pi_{\cdot\, i}}_2^{2q}\right)^{1/(2q)}
    \le C\sqrt{kp}.
\end{equation}
Indeed, the square of the column norm is bounded by a
$\operatorname{Bin}(k,p)$ random variable, while
\[
    \norm{\operatorname{Bin}(k,p)}_{L_q}
    \le C(kp+q)
    \le Ckp.
\]
Here we used $kp\ge\log k$ and $q\le C\log(ek)$.
Thus Proposition~\ref{prop:low-leverage-universality}, with
$I=H^{\mathsf c}$ and $B=12$, yields
\begin{equation}
\label{eq:hybrid-small-leverage-bound}
    \mathbb P\left\{\norm{Y}>C\sqrt{kp}\right\}
    \le k^{-12}.
\end{equation}

Finally,
\[
    \Pi U=\Pi P_HU+\Pi P_{H^{\mathsf c}}U.
\]
Combining \eqref{eq:hybrid-large-leverage-bound} and
\eqref{eq:hybrid-small-leverage-bound}, and using
$2k^{-12}\le r^{-10}$ for $k\ge r\ge3$, proves the i.i.d.
assertion.

For the fixed-column-degree model, the columns are independent,
$\mathbb E\Pi_{\cdot\,i}\Pi_{\cdot\,i}^*
=(d/k)\mathbb E\xi^2\,I_k$ with $p=d/k$, and
$\|\Pi_{\cdot\,i}\|_2^2\le d=kp$ deterministically.  Thus
\eqref{eq:hybrid-column-moment} holds.  Restrictions to a coordinate set and
padding by independent columns preserve the model, so the large-leverage
argument and Proposition~\ref{prop:low-leverage-universality} apply unchanged.

For the SparseStack$^{\mathsf T}$ model, columns are likewise independent.
Every row is selected with probability $s/k=p$, distinct blocks contribute
no common row, and the independent centered entry variables eliminate off-diagonal
covariances; hence
$\mathbb E\Pi_{\cdot\,i}\Pi_{\cdot\,i}^*
=p\mathbb E\xi^2\,I_k$.  Also
$\|\Pi_{\cdot\,i}\|_2^2\le s=kp$ deterministically, so
\eqref{eq:hybrid-column-moment} holds.  Restrictions and independent padding
again preserve the model, so both parts of the argument apply.  Combining the
two bounds proves the remaining assertions.  The independent-column structure
is essential only for Proposition~\ref{prop:low-leverage-universality}, so the
corollary is not asserted for an arbitrary negatively associated mask.
\end{proof}

\begin{remark}[Constant-distortion oblivious subspace embedding]
Consider the Rademacher specializations of the i.i.d. sparse-entry and
unnormalized SparseStack$^{\mathsf T}$ models in
Corollary~\ref{cor:concrete-sparse-models}.  Write $p$ for the Bernoulli
parameter in the first model and set $p=s/k$ in the second.  Combining the
upper-edge estimates, together with the ambient-dimension reduction of
Corollary~\ref{rem:hybrid-leverage-score-argument}, with Tropp's lower-edge
estimate for the i.i.d. model and its fixed-sparsity extension for
SparseStack$^{\mathsf T}$ \cite[Theorem~6.3 and Remark~6.5]{Tropp26}, gives
the following consequence.  There are universal constants $c_0,C_0,C>0$
such that, if
\[
    p\ge C\frac{\log k}{k}
    \qquad\text{and}\qquad
    k\ge Cr\bigl(\log\log r\bigr)^2,
\]
then, for every fixed $r$-dimensional subspace $V\subset\mathbb R^n$, with
high probability,
\[
    c_0\norm{x}_2
    \le \frac{1}{\sqrt{kp}}\norm{\Pi x}_2
    \le C_0\norm{x}_2
    \qquad\text{for every }x\in V.
\]
Thus $(kp)^{-1/2}\Pi$ in the i.i.d. model, and equivalently
$s^{-1/2}\Pi$ in the SparseStack$^{\mathsf T}$ model, is a
constant-distortion oblivious subspace embedding.
\end{remark}

\appendix

\section{Proof of the standard envelope summation lemma}
\label{app:ks-classical-envelope-summation}

We briefly recall the setup of Lemma~\ref{lem:ks-classical-envelope-summation}.
The parameters $k\ge d\ge3$ and $0<p\le1$ determine the cutoff
$\rho=\sqrt{kp}/k$, while $S$ bounds the total mass of the two dyadic level
sequences.  The finitely supported sequences $(s_j)$ and $(t_\ell)$ represent
the corresponding level sizes; they satisfy $s_j\le4d$ and $t_\ell\le k$, and
their normalized masses $m_j=2^{-2j+2}s_j$ and
$n_\ell=2^{-2\ell+2}t_\ell$ both have total mass at most $S$.  For each
active pair $(j,\ell)$, set
\[
    \mu_{j\ell}:=ps_jt_\ell,
    \qquad
    Q_{j\ell}:=
    s_j\log\frac{4ed}{s_j}
    +t_\ell\log\frac{ek}{t_\ell}.
\]
Thus $\mu_{j\ell}$ is the mean edge-count scale,
$Q_{j\ell}$ is the combined entropy cost, and
$U_{j\ell}^{\rm cl}$ is the mean term plus the smallest of the trivial,
degree, and entropy-sensitive bounds.  The goal is to sum this envelope over
the region $2^{-j-\ell+2}>\rho$.

\begin{proof}[Proof of Lemma~\ref{lem:ks-classical-envelope-summation}]
Define
\[
    \mathcal P
    :=
    \{(j,\ell):j,\ell\ge1,\ s_jt_\ell>0,\
        2^{-j-\ell+2}>\rho\}.
\]

\medskip
\noindent\textbf{Claim (One-sided envelope summation).}
Let $D,\Lambda>0$, and let $E_{j\ell}\ge0$ be supported on
$\mathcal P$.  Assume that for every $(j,\ell)\in\mathcal P$ such that
$E_{j\ell}>e\mu_{j\ell}$ one has
\begin{equation}
\label{eq:one-sided-envelope-hypotheses}
    E_{j\ell}\le Dkp s_j,
    \qquad
    E_{j\ell}\log\frac{E_{j\ell}}{\mu_{j\ell}}
    \le \Lambda\, t_\ell\log\frac{ek}{t_\ell}.
\end{equation}
Then
\[
    \sum_{(j,\ell)\in\mathcal P}
    2^{-j-\ell+2}E_{j\ell}
    \le C(S,D,\Lambda)\sqrt{kp} .
\]

\smallskip
\noindent\emph{Proof of the claim.}
The pairs with $E_{j\ell}\le e\mu_{j\ell}$ contribute at most
\[
    e\sum_{\substack{j,\ell\ge1\\2^{-j-\ell+2}>\rho}}
    2^{-j-\ell+2}ps_jt_\ell
    =
    e\sum_{\substack{j,\ell\ge1\\2^{-j-\ell+2}>\rho}}
    \frac p{2^{-j-\ell+2}}m_jn_\ell
    \le e\frac p\rho
    \left(\sum_{j\ge1}m_j\right)
    \left(\sum_{\ell\ge1}n_\ell\right)
    \le eS^2\sqrt{kp} .
\]
It remains to consider the excess pairs
$\mathcal P_{\rm ex}
:=\{(j,\ell)\in\mathcal P:E_{j\ell}>e\mu_{j\ell}\}$.
For such pairs put
\[
    \tau_{j\ell}:=
    \frac{2^{-j-\ell+2}E_{j\ell}}{\sqrt{pd}\,m_jn_\ell}
    =
    \sqrt{\frac{k}{d}}
    \frac{E_{j\ell}}{\mu_{j\ell}}
    \frac{\rho}{2^{-j-\ell+2}}.
\]
Thus
$2^{-j-\ell+2}E_{j\ell}=\sqrt{pd}\,m_jn_\ell\tau_{j\ell}$.
The excess pairs with $\tau_{j\ell}\le\sqrt{k/d}$ contribute at most
\[
    \sqrt{pd}\sqrt{\frac{k}{d}}
    \left(\sum_{j\ge1}m_j\right)
    \left(\sum_{\ell\ge1}n_\ell\right)
    \le S^2\sqrt{kp} .
\]
It remains to consider pairs with $\tau_{j\ell}>\sqrt{k/d}$.  Write
\[
    W_\ell:=\log\frac{ek}{t_\ell}
    =
    \log\frac{e2^{-2\ell+2}k}{n_\ell}.
\]

The remaining excess pairs, for which $\tau_{j\ell}>\sqrt{k/d}$, are partitioned
into the following four disjoint classes:
\begin{itemize}
\item[(A)] $2^{-\ell+1}<2^{-j+1}/\sqrt{kp}$;
\item[(B)] $2^{-\ell+1}\ge2^{-j+1}/\sqrt{kp}$ and
      $\log(E_{j\ell}/\mu_{j\ell})>W_\ell/4$;
\item[(C)] $2^{-\ell+1}\ge2^{-j+1}/\sqrt{kp}$,
      $\log(E_{j\ell}/\mu_{j\ell})\le W_\ell/4$, and
      $\log(e2^{-2\ell+2}k)\ge-\log n_\ell$;
\item[(D)] $2^{-\ell+1}\ge2^{-j+1}/\sqrt{kp}$,
      $\log(E_{j\ell}/\mu_{j\ell})\le W_\ell/4$, and
      $\log(e2^{-2\ell+2}k)<-\log n_\ell$.
\end{itemize}

First consider class \emph{(A)}, in which
$2^{-\ell+1}<2^{-j+1}/\sqrt{kp}$.  For fixed $j$, the first bound in
\eqref{eq:one-sided-envelope-hypotheses} gives
\[
\begin{aligned}
    \sum_{\substack{\ell\ge1:\,(j,\ell)\in\mathcal P_{\rm ex}\\
        \tau_{j\ell}>\sqrt{k/d},\ 
        2^{-\ell+1}<2^{-j+1}/\sqrt{kp}}}
    2^{-j-\ell+2}E_{j\ell}
    &\le
    Dkp\,2^{-j+1}s_j
    \sum_{\substack{\ell\ge1\\
        2^{-\ell+1}<2^{-j+1}/\sqrt{kp}}}2^{-\ell+1} \\
    &\le
    CD\sqrt{kp}\,m_j .
\end{aligned}
\]
After summing in $j$, this class contributes at most $CDS\sqrt{kp}$.

Next consider class \emph{(B)}.  Here
\[
    \log\frac{E_{j\ell}}{\mu_{j\ell}}>\frac14 W_\ell,
\]
so the second bound in \eqref{eq:one-sided-envelope-hypotheses} gives
\[
    m_j\tau_{j\ell}
    =
    \frac{2^{-j+1}E_{j\ell}}
    {\sqrt{pd}\,2^{-\ell+1}t_\ell}
    \le
    4\Lambda\frac{2^{-j+1}}{\sqrt{pd}\,2^{-\ell+1}}.
\]
Since the degree case does not occur,
$2^{-j+1}\le\sqrt{kp}\,2^{-\ell+1}$.
Thus, for fixed $\ell$,
\[
    \sum_{\substack{j\ge1:\,(j,\ell)\in\mathcal P_{\rm ex}\\
        \tau_{j\ell}>\sqrt{k/d},\ 
        2^{-j+1}\le\sqrt{kp}\,2^{-\ell+1}\\
        \log(E_{j\ell}/\mu_{j\ell})>W_\ell/4}}
    m_j\tau_{j\ell}
    \le
    4\Lambda
    \sum_{\substack{j\ge1\\
        2^{-j+1}\le\sqrt{kp}\,2^{-\ell+1}}}
    \frac{2^{-j+1}}{\sqrt{pd}\,2^{-\ell+1}}
    \le C\Lambda\sqrt{\frac{k}{d}} .
\]
Multiplying by $\sqrt{pd}\,n_\ell$ and summing in $\ell$, this case contributes
at most $C\Lambda S\sqrt{kp}$.

For class \emph{(C)} we have
\begin{equation}
\label{eq:class-c-small-log}
    \log\frac{E_{j\ell}}{\mu_{j\ell}}\le\frac14W_\ell .
\end{equation}
Moreover,
\[
    \log(e2^{-2\ell+2}k)\ge-\log n_\ell .
\]
Then $W_\ell\le2\log(e2^{-2\ell+2}k)$, and
\eqref{eq:class-c-small-log} gives
\[
    \frac{E_{j\ell}}{\mu_{j\ell}}
    \le\sqrt e\,2^{-\ell+1}\sqrt k .
\]
Since $\tau_{j\ell}>\sqrt{k/d}$, equivalently
$(E_{j\ell}/\mu_{j\ell})\rho/2^{-j-\ell+2}>1$, this implies
\[
    2^{-j+1}<\sqrt e\,\sqrt p .
\]
Moreover, the present case and $n_\ell\le2^{-2\ell+2}k$ imply that
$2^{-\ell+1}\sqrt k$ is bounded
below by an absolute constant, and hence
$W_\ell\le C2^{-\ell+1}\sqrt k$.  Since
$E_{j\ell}>e\mu_{j\ell}$, the logarithm in the second bound of
\eqref{eq:one-sided-envelope-hypotheses} is at least $1$, so that bound gives
\[
\begin{aligned}
    m_j\tau_{j\ell}
    =
    \frac{2^{-j+1}E_{j\ell}}
    {\sqrt{pd}\,2^{-\ell+1}t_\ell} 
    \le
    \Lambda\frac{2^{-j+1}W_\ell}
    {\sqrt{pd}\,2^{-\ell+1}}
    \le
    C\Lambda\frac{2^{-j+1}\sqrt k}{\sqrt{pd}}.
\end{aligned}
\]
For fixed $\ell$,
\[
    \sum_{\substack{j\ge1:\,(j,\ell)\in\mathcal P_{\rm ex}\\
        \tau_{j\ell}>\sqrt{k/d},\ 
        \log(E_{j\ell}/\mu_{j\ell})\le W_\ell/4\\
        \log(e2^{-2\ell+2}k)\ge-\log n_\ell,\ 
        2^{-j+1}<\sqrt e\,\sqrt p}}
    m_j\tau_{j\ell}
    \le
    C\Lambda
    \sum_{\substack{j\ge1\\2^{-j+1}<\sqrt e\,\sqrt p}}
    \frac{2^{-j+1}\sqrt k}{\sqrt{pd}}
    \le C\Lambda\sqrt{\frac{k}{d}} .
\]
After summing against $\sqrt{pd}\,n_\ell$, this case contributes at most
$C\Lambda S\sqrt{kp}$.

Finally, in class \emph{(D)},
\[
    \log(e2^{-2\ell+2}k)<-\log n_\ell .
\]
Let $A_\ell:=\log(e2^{-2\ell+2}k)$.  Since
$n_\ell\le2^{-2\ell+2}k$, the present case would be impossible if
$n_\ell\ge1$; hence $-\log n_\ell>0$.  Moreover,
$A_\ell<-\log n_\ell$ and $W_\ell=A_\ell-\log n_\ell$, so
\[
    \frac14W_\ell<-\frac12\log n_\ell\le-\log n_\ell .
\]
Together with \eqref{eq:class-c-small-log}, which also holds in class
\emph{(D)}, this gives
\[
    \frac{E_{j\ell}}{\mu_{j\ell}}\le n_\ell^{-1}.
\]
Consequently
\[
\begin{aligned}
    n_\ell\tau_{j\ell}
    =
    \sqrt{\frac{k}{d}}\,
    n_\ell\frac{E_{j\ell}}{\mu_{j\ell}}
    \frac{\rho}{2^{-j-\ell+2}}
    \le
    \sqrt{\frac{k}{d}}\frac{\rho}{2^{-j-\ell+2}}.
\end{aligned}
\]
For fixed $j$, using the heavy condition $2^{-j-\ell+2}>\rho$,
\[
\begin{aligned}
    \sum_{\substack{\ell\ge1:\,(j,\ell)\in\mathcal P_{\rm ex}\\
        \tau_{j\ell}>\sqrt{k/d},\ 
        \log(E_{j\ell}/\mu_{j\ell})\le W_\ell/4\\
        \log(e2^{-2\ell+2}k)<-\log n_\ell}}
    n_\ell\tau_{j\ell}
    &\le
    \sqrt{\frac{k}{d}}\rho\,2^{j-1}
    \sum_{\substack{\ell\ge1\\2^{-j-\ell+2}>\rho}}2^{\ell-1}
    \le C\sqrt{\frac{k}{d}} .
\end{aligned}
\]
Multiplying by $\sqrt{pd}\,m_j$ and summing in $j$ completes the proof of
the claim. \hfill$\square$

\medskip
\noindent\textbf{Claim (Transposed envelope summation).}
Let $D,\Lambda>0$, and let $E_{j\ell}\ge0$ be supported on
$\mathcal P$.  Assume that for every $(j,\ell)\in\mathcal P$ such that
$E_{j\ell}>e\mu_{j\ell}$ one has
\[
    E_{j\ell}\le Dpd t_\ell,
    \qquad
    E_{j\ell}\log\frac{E_{j\ell}}{\mu_{j\ell}}
    \le
    \Lambda\, s_j\log\frac{4ed}{s_j}.
\]
Then
\[
    \sum_{(j,\ell)\in\mathcal P}
    2^{-j-\ell+2}E_{j\ell}
    \le C(S,D,\Lambda)\sqrt{kp}.
\]

The proof is identical to that of the preceding claim after transposing the
two level sequences and using $d\le k$.

We now apply the two envelope-summation claims to complete the proof of the lemma.
Recall that, for $(j,\ell)\in\mathcal P$,
\[
    U_{j\ell}^{\rm cl}
    =
    \mu_{j\ell}
    +
    \min\left\{
        s_jt_\ell,\,
        kp s_j,\,
        pd t_\ell,\,
        \frac{Q_{j\ell}}{\log(e+Q_{j\ell}/\mu_{j\ell})}
    \right\}.
\]

The mean term in
$U_{j\ell}^{\rm cl}$ contributes at most
\[
    \sum_{(j,\ell)\in\mathcal P}
    2^{-j-\ell+2}\mu_{j\ell}
    \le S^2\sqrt{kp} .
\]
For $(j,\ell)\in\mathcal P$, let
\[
    E_{j\ell}:=
    \min\left\{
        s_jt_\ell,\,
        kp s_j,\,
        pd t_\ell,\,
        \frac{Q_{j\ell}}{\log(e+Q_{j\ell}/\mu_{j\ell})}
    \right\}.
\]
It remains to bound
$\sum_{(j,\ell)\in\mathcal P}2^{-j-\ell+2}E_{j\ell}$.

Put
\[
    A_{j\ell}:=s_j\log\frac{4ed}{s_j},
    \qquad
    B_{j\ell}:=t_\ell\log\frac{ek}{t_\ell},
    \qquad
    Q_{j\ell}=A_{j\ell}+B_{j\ell}.
\]
Writing $x=Q_{j\ell}/\mu_{j\ell}$, the last term in the minimum gives
$E_{j\ell}/\mu_{j\ell}\le x/\log(e+x)\le x$, and hence
\[
    E_{j\ell}\log\left(e+\frac{E_{j\ell}}{\mu_{j\ell}}\right)
    \le
    \frac{Q_{j\ell}}{\log(e+x)}\log(e+x)
    =Q_{j\ell}.
\]
Partition $\mathcal P=\mathcal P_1\sqcup\mathcal P_2$ according as
$A_{j\ell}\le B_{j\ell}$ or $A_{j\ell}>B_{j\ell}$, and set
$E_{j\ell}^{(q)}:=E_{j\ell}\mathbf 1_{\{(j,\ell)\in\mathcal P_q\}}$.
For excess pairs in $\mathcal P_1$,
\[
    E_{j\ell}^{(1)}\le kp s_j,
    \qquad
    E_{j\ell}^{(1)}\log\frac{E_{j\ell}^{(1)}}{\mu_{j\ell}}
    \le Q_{j\ell}\le2B_{j\ell},
\]
whereas for excess pairs in $\mathcal P_2$,
\[
    E_{j\ell}^{(2)}\le pd t_\ell,
    \qquad
    E_{j\ell}^{(2)}\log\frac{E_{j\ell}^{(2)}}{\mu_{j\ell}}
    \le Q_{j\ell}<2A_{j\ell}.
\]
The one-sided and transposed envelope-summation claims now yield
\[
    \sum_{(j,\ell)\in\mathcal P}
    2^{-j-\ell+2}E_{j\ell}
    =
    \sum_{q=1}^2\sum_{(j,\ell)\in\mathcal P}
    2^{-j-\ell+2}E_{j\ell}^{(q)}
    \le C(S)\sqrt{kp}.
\]
Together with the mean-term estimate, this proves the lemma.
\end{proof}

\end{document}